\documentclass[leqno,12pt]{article}

\def\today{${\scriptscriptstyle\number\day-\number\month-\number\year}$}
\usepackage[cp1252]{inputenc}
\usepackage{graphicx}
\usepackage{fancyhdr}
\usepackage{amsmath}
\usepackage{amsthm}
\usepackage{amsfonts}
\newtheorem{theorem}{Theorem}[section]
\newtheorem{lemma}[theorem]{Lemma}
\newtheorem{definition}[theorem]{Definition}

\newtheorem{remark}[theorem]{Remark}

\def\address#1{{\center{#1}}}
\date{}
\def\m@th{\mathsurround=0pt}
\def\eqal#1{\null\,\vcenter{\openup\jot\m@th
 \ialign{\strut\hfil$\displaystyle{##}$&&$\displaystyle{{}##}$\hfil
 \crcr#1\crcr}}\,}
\def\matrix#1{\null\,\vcenter{\normalbaselines\m@th
 \ialign{\hfil$##$\hfil&&\quad\hfil$##$\hfil\crcr
 \mathstrut\crcr\noalign{\kern-\baselineskip}
 #1\crcr\mathstrut\crcr\noalign{\kern-\baselineskip}}}\,}
\def\D{{\mathbb D}}
\def\N{{\mathbb N}}
\def\R{{\mathbb R}}
\def\divv{{\rm div}\,}

\def\supp{{\rm supp\,}}
\def\\{\hfill\break}

\def\nad#1#2{{\mathop{#1}\limits^{#2}\null\!}\null}
\numberwithin{equation}{section}

\title{On the Stokes system in cylindrical domains}
\author{J. Renc\l awowicz$^{\left.1\right)}$,\quad Wojciech M. Zaj\c{a}czkowski$^{\left.2\right),\left.3\right)}$}

\begin{document}
\input amssym.def
\input amssym.tex
\maketitle
\thispagestyle{fancy}

\address{$^{\left.1\right)}$\ Institute of Mathematics, Polish Academy of Sciences,\\
\'Sniadeckich 8, 00-656 Warsaw, Poland\\
e-mail: jr@impan.pl\\
$^{\left.2\right)}$\ Institute of Mathematics, Polish Academy of Sciences,\\
\'Sniadeckich 8, 00-656 Warsaw, Poland\\
e-mail: wz@impan.pl\\
$^{\left.3\right)}$\ Institute of Mathematics and Cryptology, 
Cybernetics Faculty, \\
Military University of Technology,\\
S. Kaliskiego 2, 00-908 Warsaw, Poland\\}

\begin{abstract}
The existence of solutions to some initial-boundary value problem for the Stokes system is proved. The result is shown in Sobolev-Slobodetskii spaces such that the velocity belongs to $W_r^{2+\sigma,1+\sigma/2}(\Omega^T)$ and gradient of pressure to $W_r^{\sigma,\sigma/2}(\Omega^T)$, where $r\in(1,\infty)$, $\sigma\in(0,1)$, $\Omega^T=\Omega\times(0,T)$. These are special Besov spaces: $B_{r,r}^{2+\sigma,1+\sigma/2}(\Omega^T)$ and $B_{r,r}^{\sigma,\sigma/2}(\Omega^T)$, respectively. The existence is proved by the technique of regularizer.
\end{abstract}

\section{Introduction}\label{se1}

We consider the following initial-boundary value problem for the Stokes system in a cylindrical domain $\Omega\subset\R^3$,
\begin{equation}\eqal{
&v_t-\nu\Delta v+\nabla p=f\quad &{\rm in}\ \ \Omega\times(0,T),\cr
&\divv v=g\quad &{\rm in}\ \ \Omega\times(0,T),\cr
&\bar n\cdot\D(v)\cdot\bar\tau_\alpha=b_\alpha,\ \ \alpha=1,2,\quad &{\rm on}\ \ S\times(0,T),\cr
&v\cdot\bar n=b_3\quad &{\rm on}\ \ S\times(0,T),\cr
&v|_{t=0}=v_0\quad &{\rm in}\ \ \Omega,\cr}
\label{1.1}
\end{equation}
where $S=S_1\cup S_2$. Introducing the Cartesian system of coordinates $x=(x_1,x_2,x_3)$ we assume that $\Omega$ and $S_1$ are parallel to the $x_3$-axis and $S_2$ is perpendicular to it.
By $\bar n$ we denote the unit outward vector normal to $S$ and $\bar\tau_1$, $\bar\tau_2$ are tangent to $S$.

The boundaries $S_1$ and $S_2$ of cylindrical domain $\Omega$ meet along curves $L_1$ and $L_2$ under the angle $\pi/2$. This means that the considered domain $\Omega$ is geometrically complicated. Therefore, we prove existence of solutions to problem (\ref{1.1}) by the technique of regularizer. For this we need a complex partition of unity.

Using the partition of unity we can localize problem (\ref{1.1}) to local problems near $S_1$, near $S_2$ and also near edges $L_1$, $L_2$. Then localized problems near $S_1$ and $S_2$ are transformed to problems in the half space by an appropriate transformation of coordinates. We can assume that the half space in local coordinates is such that $x_3>0$. Applying the Fourier-Laplace transform with respect to time and tangent derivatives to $x_3=0$ we obtain a system of ordinary differential equations with respect to $x_3$. Solving these equations we can estimate solutions in Besov spaces by using the Triebel definition, (see \cite{Tr1}, Def 2.1 and Def 2.3.1).

However, to prove the existence of solutions by the technique of regularizer we need the formulation of  Besov spaces introduced by Besov (see \cite[Ch. 3, Sect. 18]{BIN}). Fortunately, it is shown (see Lemma 2.12) that both definitions are equivalent.

The most difficult local problem are problems near edges $L_\alpha$, $\alpha=1,2$. This can be also transformed to a problem in the half space.

The main result is the following

\begin{theorem}\label{t1.1}
Assume that $f\in W_r^{\sigma,\sigma/2}(\Omega^T)$, $v_0\in W_r^{2+\sigma-2/r}(\Omega)$,\break $g\in W_r^{1+\sigma,1/2+\sigma/2}(\Omega^T)$, $b_\alpha\in W_r^{1+\sigma-1/r,1/2+\sigma/2-1/2r}(S_i^T)$, $\alpha=1,2$, and\break $b_3\in W_r^{2+\sigma-1/r,1+\sigma/2-1/2r}(S_i^T)$, where $r\in(1,\infty)$, $\sigma\in(0,1)$, $i=1,2$.\\
Assume that $S_1\in C^3$ Assume some compatibility conditions (see Remark \ref{r1.2}).\\
Then there exists a solution to problem (\ref{1.1}) such that $v\in W_r^{2+\sigma,1+\sigma/2}(\Omega^T)$, $\nabla p\in W_r^{\sigma,\sigma/2}(\Omega^T)$, and the estimate holds
\begin{equation}\eqal{
&\|v\|_{W_r^{2+\sigma,1+\sigma/2}(\Omega^T)}+\|\nabla p\|_{W_r^{\sigma,\sigma/2}(\Omega^T)}\cr
&\le c\bigg(\|f\|_{W_r^{\sigma,\sigma/2}(\Omega^T)}+ \|g\|_{W_r^{1+\sigma,1/2+\sigma/2}(\Omega^T)}\cr
&\quad+\sum_{i=1}^2\bigg(\sum_{\alpha=1}^2 \|b_\alpha\|_{W_r^{1+\sigma-1/r,1/2+\sigma/2-1/2r}(S_i^T)}\cr
&\quad+\|b_3\|_{W_r^{2+\sigma-1/r,1+\sigma/2-1/2r}(S_i^T)}\bigg)+ \|v_0\|_{W_r^{2+\sigma-2/r}(\Omega)}\bigg),\cr}
\label{1.2}
\end{equation}
where $c$ does not depend on $v$ neither $p$.
\end{theorem}

\begin{remark}\label{r1.2}
Compatibility conditions
$$\eqal{
&\bar n\cdot\D(v_0)\cdot\bar\tau_\alpha|_{S_1}=b_\alpha|_{S_1}, \alpha=1,2,\quad &{\rm in}\quad W_r^{1+\sigma-3/r}(S_1),\cr
&v\cdot\bar n|_S=b_3|_{t=0}\quad &{\rm in}\quad W_r^{2+\sigma-3/r}(S),\cr
&\divv v_0=g|_{t=0}\quad &{\rm in}\quad W_r^{1+\sigma-2/r}(\Omega).\cr}
$$
\end{remark}

The existence of solutions to problem (\ref{1.1}) for sufficiently smooth boundary is proved in \cite{ZZ2} in Besov spaces such that $v\in B_{p,q}^{\sigma+2,{\sigma\over2}+1}(\Omega^T)$, $\nabla p\in B_{p,q}^{\sigma,\sigma/2}(\Omega^T)$, $\sigma\in\R_+$, $p,q\in(1,\infty)$. However, in \cite{ZZ2} the classical fundamental approach as solvability of problems in the whole space, in the half space and finally the existence in a bounded domain by the technique of regularizer are not used.

In \cite{ZZ1} the existence of solutions to the initial boundary value problem for the heat equation is proved in Besov spaces using the Triebel definition of Besov spaces (see \cite{Tr1}, Def. 2.3.1) and using the techniques developed by Triebel in \cite{Tr1}, Section 2.3.6. Continuing the approach from \cite{ZZ1} we prove the existence of solutions to problem (\ref{1.1}) in the cylindrical domain~$\Omega$.

We mention that the existence of solutions to the nonstationary Stokes system in domains with smooth boundaries was extensively studied by many authors (see \cite{Al}, \cite{S2}, \cite{S3}, \cite{S4}).
In these papers the anisotropic Sobolev and H\"older spaces were applied.

Besov spaces were used in some approach to Navier-Stokes equations. We mention some of these results below.

In \cite{B}, the author proves analicity of Stokes operator in Besov space $B^a_{p,q}(\R^n_+).$ First, it is shown that the semigroup obtained by the Stokes solution on $\R^n_+$ is analytic in Besov spaces in time. Then, the asymptotic behavior of solutions for the nonstationary Navier-Stokes equations is estimated and the decay rate in the Besov space is studied.

The paper \cite{DM} is devoted to the boundary value problem for the incompressible inhomogeneous Navier–-Stokes equations in the half-space in the case of small data with critical regularity. The result states that in
dimension $n \ge 3$, if the initial density is close to a positive constant in $L_{\infty} \cap \dot{W}^1_n(\R^n_+)$ and
the initial velocity is small with respect to the viscosity in the homogeneous Besov space $\dot{B}^0_{n,1}(\R^n_+)$ then the equations have a unique global solution. The proof strongly relies on maximal regularity estimates for the Stokes system in the half-space in $L_1(0,T ;\dot{B}^0_{p,1}(\R^n_+)).$ Namely, it was necessary to obtain time-independent maximal estimates for the linearized velocity equation, i.e. the evolutionary Stokes system. They obtained standard estimates in Lebesgue spaces for the transport equation for the density and estimates in homogeneous Besov spaces for the Stokes system and in order to bound nonlinear terms in the linearized problem, bilinear estimates in Besov spaces have been used.

In \cite{R}, the author studied global well-posedness for the nonnhomogeneous Navier-Stokes equations on $\R^n, n \ge 2,$ with initial velocity
in endpoint critical Besov spaces $B^{-1+n/q}_{q,\infty}(\R^n), n \le q < 2n,$ and merely bounded initial density with a positive lower bound. He considered a multiplication property of $L^{\infty}$-functions in some Bessel potential and Besov spaces. Based on it and on maximal regularity of the Stokes operator in little Nicolskii spaces, it was shown solvability for the momentum equations
with fixed bounded density. The proof for existence of a solution to the nonnhomogeneous Navier-Stokes equations was done
via an iterative scheme when $B^{-1+n/q}_{q,\infty}$ -norm of initial velocity and relative variation of initial density are small, while
uniqueness of a solution was proved via a Lagrangian approach when initial velocity belongs to $B^{-1+n/q}_{q,\infty}(\R^n)\cap B^{-1+n/q}_{r,\infty}(\R^n)$
for slightly larger $r>q.$

In \cite{S}, a local in time solution was constructed for the Cauchy problem
of the $n$-dimensional Navier–-Stokes equations when the initial velocity
belongs to Besov spaces of nonpositive order. The space contains $L^{\infty}$
in some exponents, so the solution may not decay at space infinity.
In order to use iteration scheme it was necessary to establish the H\"{o}lder type
inequality for estimating bilinear term by dividing the sum of Besov
norm with respect to levels of frequency. Moreover, by regularizing
effect solutions belong to $L^{\infty}$ for any positive time.

In \cite{OS}, authors considered global well-posedness of the Cauchy problem of the incompressible Navier–Stokes equations under the Lagrangian coordinates in scaling critical Besov spaces. They proved the system is globally well-posed in the homogeneous Besov space $\dot{B}^{- 1 + n/p}_{p,1}(\R^n)$ with $1\le p \le \infty$. They improved the former result, restricted for $1 \le p \le 2n$ and the main reason why the well-posedness space was enlarged is that the quasi-linear part of the system has a special feature called a multiple divergence structure and the bilinear estimate for the nonlinear terms are improved by such a structure. The result indicates that the Navier-Stokes equations can be transferred from the Eulerian coordinates to the Lagrangian coordinates even for the solution in the limiting critical Besov spaces.

 In \cite{KOT1},
authors got the full regularity of weak solutions to Navier-Stokes equations under some assumptions on the
velocity $u,$ namely, if \\
$u \in L^2(0, T; \dot{B}^0_{\infty,\infty}(\R^3)),$
then the solution $u$ is regular. By the definition of Besov space and Bernstein
inequality, the condition is equivalent to:
$\nabla u \in L^2(0, T; \dot{B}^{-1}_{\infty,\infty}(\R^3)).$ On the other hand, in \cite{FQ}, if $\nabla u_{h} \in L^{8/3}(0, T; \dot{B}^{-1}_{
\infty,\infty}(\R^3))$ or  $\nabla u_3 \in L^{\frac{8}{5-2/s}}(0,T; \dot{B}^{-s}_{
\infty,\infty}(\R^3)), 0<s<1,$ where $\nabla u_h$ denotes horizontal gradient components, then $u$ is a regular solution.

In \cite{KOT2}, the local existence theorem for the Navier-Stokes equations in $\R^n, n\ge 2,$ with the initial data in $B^0_{\infty,\infty}$ containing functions that do not decay at infinity, was proved. Moreover, authors established the extension criterion on local solutions in terms of the vorticity in the homogeneous Besov space $\dot{B}^0_{\infty,\infty}.$

In \cite{CZ}, authors established the space-time estimates in the Besov spaces of the
solution to the Navier-Stokes equations in $\R^n, n \ge 3$. As an application, they improved some known results about the regularity criterion of weak solutions and the blow-up criterion of smooth solutions. Instead of the logarithmic Sobolev inequality, the main tools were the frequency localization and the Littlewood-Paley trichotomy decomposition, as a basic way to analyze bilinear expressions.

In \cite{KS}, authors show the existence theorem of global mild solutions to Navier-Stokes equations in the whole space $\R^n$ with small initial data and external forces in the time-weighted Besov space which is an invariant space under the change of scaling. The result on local existence of solutions for large data is also discussed. The method is based on the $L^p-L^q$ estimate of the Stokes equations in Besov spaces. Using various estimates for the Stokes semi-group in the homogeneous Besov space and based on the paraproduct formula, they established a bilinear estimate related to the nonlinear term of
Navier-Stokes equation to apply the implicit function theorem and show the local existence.

In \cite{KM}, the local existence, uniqueness and regularity of solutions of the initial-value problem for non-stationary Navier-Stokes equations were studied via abstract Besov spaces. Authors could prove an estimate of  semigroups in abstract Besov spaces instead of fractional powers. In the paper, the domain is a bounded domain in $\R^n$,  a half-space of $\R^n$ with $n \ge 2$, or an exterior domain in $\R^n$ with $n \ge 3,$ and the boundary is smooth.

The remaining part of the paper is divided into the following parts.
\begin{itemize}
\item[*] In Section \ref{se2} we introduce all necessary notations, definitions and auxiliary results.
\item[*] In Section \ref{se3} we consider the Stokes system in the whole space.
\item[*] In Section \ref{se4} the Stokes system in the half-space is examined.
\item[*] In Section \ref{se5} the existence of solutions in the cylindrical domain is proved with the technique of regularizer.
\item[*] Finally, in Section \ref{se6} solvability of the Stokes system in Sobolev spaces is discussed.
\end{itemize}

\section{Notation and preliminaries}\label{se2}

In this Section we begin the part of this paper devoted to the problem of existence of solutions to the Stokes system in anisotropic Besov spaces.
Let $\R^3$ be a three-dimensional real Euclidean space. Let $x=(x_1,x_2,x_3)\in\R^3$ be the system of Cartesian coordinates.

Throughout the paper we use the notation: $x'=(x_1,x_2)$, \\ $\bar x=(x_1,x_2,x_3,t)\equiv(x,t)$, $\bar x'=(x_1,x_2,t)\equiv(x',t)$.
For $\bar x\in\R^4$ we introduce the anisotropic distance from the origin of coordinates,
$$
|\bar x|_a=\bigg(|t|+\sum_{i=1}^3|x_i|^2\bigg)^{1/2}
$$
and, similarly for $\bar x'\in\R^3$, we have
$$
|\bar x'|_a=\bigg(|t|+\sum_{i=1}^2|x_i|^2\bigg)^{1/2}.
$$
Let $S(\R^n)$ and $S'(\R^n)$ be the Schwartz space and the space of tempered distributions on $\R^n$, respectively.

\begin{definition}\label{d2.1}
By $\Phi_a(\R^n)$ we denote the collection of all systems \\ $\varphi=\{\varphi_k(\bar x)\}_{k=0}^\infty\subset S(\R^4)$ with the following properties
\begin{itemize}
\item[$1^\circ$] $\supp\varphi_0\subset\{\bar x\colon|\bar x|_a\le 2\}$;
\item[$2^\circ$] $\supp\varphi_k\subset\{\bar x\colon 2^{k-1}\le|\bar x|_a\le 2^{k+1}\}$, $k=1,2,\dots$;
\item[$3^\circ$] for every multi-index $\bar\alpha=(\alpha_1,\alpha_2,\alpha_3,\alpha_0)$, there exists a positive number $c_{\bar\alpha}$ such that
$$
2^{k(2\alpha_0+\sum_{i=1}^3\alpha_i)}|D^{\bar\alpha}\varphi_k(\bar x)|\le c_{\bar\alpha}
$$
for all $k\in\N\cup\{0\}\equiv\N_0$ and all $\bar x\in\R^4$, where
$$
D^{\bar\alpha}u=\partial_t^{\alpha_0}\partial_{x_1}^{\alpha_1}\partial_{x_2}^{\alpha_2} \partial_{x_3}^{\alpha_3}u
$$
and $|\bar\alpha|=\alpha_0+\alpha_1+\alpha_2+\alpha_3$, $\alpha_j\in\N_0$, $j\in\{0,1,2,3\}$;
\item[$4^\circ$] $\sum_{k=0}^\infty\varphi_k(\bar x)=1$ for all $\bar x\in\R^4$.
\end{itemize}
\end{definition}

A similar definition can be introduced for functions $\varphi_k$ depending on $\bar x'$.

\begin{definition}\label{d2.2}
Introduce the following Fourier transforms. Let $f\in S'(\R^4).\!$ Then
$$
(Ff)(\xi,\xi_0)=\intop_{\R^3\times\R}e^{-i(t\xi_0+x\cdot\xi)}f(x,t)dxdt
$$
and
$$
(F^{-1}f)(x,t)={1\over(2\pi)^4}\intop_{\R^3\times\R} e^{i(t\xi_0+x\cdot\xi)}f(\xi,\xi_0)d\xi d\xi_0,
$$
where $x\cdot\xi=x_1\xi_1+x_2\xi_2+x_3\xi_3$, $d\xi=d\xi_1d\xi_2d\xi_3$.\\
Then the corresponding Laplace-Fourier transform has the form
$$
({\cal L}f)(\xi,\xi_0,\gamma)=(Ff_\gamma)(\xi,\xi_0)
$$
where
$$
f_\gamma=\left\{\eqal{&e^{-\gamma t}f\quad &{\rm for}\ \ t>0,\cr
&0\quad &{\rm for}\ \ t<0.\cr}\right.
$$
Let $f\in S'(\R^3)$. Then
$$
(F_1f)(\xi',\xi_0)=\intop_{R^2\times\R}e^{-i(t\xi_0+x'\cdot\xi')} f(x',t)dx'dt
$$
and
$$
(F_1^{-1}f)(x',t)={1\over(2\pi)^3}\intop_{\R^2\times\R}e^{i(t\xi_0+x'\xi')} f(\xi',\xi_0)d\xi'd\xi_0,
$$
where $x'\cdot\xi'=x_1\xi_1+x_2\xi_2$, $d\xi'=d\xi_1d\xi_2$, $dx'=dx_1dx_2$.
Hence
$$
({\cal L}_1f)(\xi',\xi_0,\gamma)=F_1(f_\gamma)(\xi',\xi_0,\gamma).
$$
\end{definition}

Now we are going to define anisotropic Besov spaces
$$
B_{p,q}^{\sigma,\sigma/2}(\R^3\times\R),\quad B_{p,q}^{\sigma,\sigma/2}(\R^2\times\R),\quad B_{p,q}^{\sigma,\sigma/2}(\R_+^3\times\R),
$$
where $\R_+^3=\{x\in\R^3\colon x_3>0\}$.

To this end, for $f\in S'(\R^4)$ we introduce the Fourier transform
$$
(F_2f)(\xi',x_3,\xi_0)=\intop_{\R^2\times\R}e^{-i(t\xi_0+x'\cdot\xi')} f(x',x_3,t)dx'dt.
$$
Then the corresponding Laplace-Fourier transform has the form
$$
({\cal L}_2f)(\xi',x_3,\xi_0,\gamma)=F_2(f_\gamma)(\xi',x_3,\xi_0)
$$

\begin{definition}(see \cite{ZZ1}, \cite[Sect. 2,3,1]{Tr1}),\label{d2.3}
Let $p,q\in[1,\infty]$, $\sigma\in\R+$. The anisotropic Besov space $B_{p,q,\gamma}^{\sigma,\sigma/2}(\R^3\times\R)$ is the space of functions \\ $u=u(x,t)\in S'(\R^4)$ with the finite norm
$$
\|u\|_{B_{p,q,\gamma}^{\sigma,\sigma/2}(\R^3\times\R)}=\bigg[\sum_{k=0}^\infty\bigg( \intop_{\R^3\times\R}|2^{\sigma k}(F^{-1}(\varphi_kFu_\gamma)(x,t)|^pdxdt\bigg)^{q/p}\bigg]^{1/q},
$$
where the Fourier transform $F$ is defined above and $\varphi_k\in\Phi_a(\R^4)$ is introduced in Definition \ref{d2.1}.
\end{definition}

\begin{definition} (see \cite{ZZ1}, \cite[Sect. 2,3,1]{Tr1})\label{d2.4}
Let $p,q\in[1,\infty]$, $\sigma\in\R_+$. The anisotropic Besov space $B_{p,q,\gamma}^{\sigma,\sigma/2}(\R^2\times\R)$ is the space of functions \\ $u=u(x',t)\in S'(\R^3)$ with the finite norm
$$
\|u\|_{B_{p,q,\gamma}^{\sigma,\sigma/2}(\R^2\times\R)}=\bigg[\sum_{k=0}^\infty\bigg( \intop_{\R^2\times\R}|2^{\sigma k}(F_1^{-1}(\varphi_kF_1u_\gamma))(x',t)|^pdx'dt\bigg)^{q/p}\bigg]^{1/q},
$$
where $F_2$ and $\varphi_k$, $k\in\N_0$, are defined above.
\end{definition}

\begin{definition}[Besov spaces defined on $\R_+^3\times\R$]\label{d2.5}
Let $p,q\in[1,\infty]$ and $\sigma\in\R_+$. The anisotropic Besov space $B_{p,q,\gamma}^{\sigma,\sigma/2}(\R_+^3\times\R)$ is the space of functions $u=u(x,t)\in S'(\R_+^3\times\R)$ with the finite norm
$$\eqal{
&\|u\|_{B_{p,q,\gamma}^{\sigma,\sigma/2}(\R_+^3\times\R)}=\|u\|_{L_p(\R_+^3\times\R)}\cr
&\quad+\bigg[\sum_{k=0}^\infty\bigg(\sum_{j\le[\sigma]}\intop_{\R_+^3\times\R} |2^{(\sigma-j)k}(F_2^{-1}\varphi_kF_2\partial_{x_3}^ju_\gamma)(x,t)|^pdxdt\bigg)^{q/p} \bigg]^{1/q}\cr
&\quad+\bigg[\sum_{k=0}^\infty\bigg(\intop_\R dt\intop_{\R_+}dx_3\intop_{\R_+}dz\intop_{\R^2}dx'\cdot\cr
&\qquad\cdot{|F_2^{-1}\varphi_kF_2(\partial_{x_3}^{[\sigma]}u_\gamma(\bar x',x_3+z)-\partial_{x_3}^{[\sigma]}u_\gamma(\bar x',x_3))|^p\over z^{1+p(\sigma-[\sigma])}}\bigg)^{q/p}\bigg]^{1/q},\cr}
$$
where $[\sigma]$ is the integer part of $\sigma$.
\end{definition}

\begin{definition}\label{d2.6}
By $B_{p,q,\gamma}^{\sigma,\sigma/2}$ we denote such Besov space that in the definition of $B_{p,q}^{\sigma,\sigma/2}$ the Fourier transform is replaced by the Fourier-Laplace transform.
\end{definition}

\begin{lemma}[see Ch. 4, Sect. 18 \cite{BIN}]\label{l2.7}
Let $f\in B_{p,q}^{\sigma,\sigma/2}(\R_+^3\times\R)$, $\sigma\in\R_+$, $p,q\in(1,\infty)$, $\R_+^3=\{x\in\R^3\colon x_3>0\}$.\\
Then there exists an extension of $f$ onto $\R^3\times\R$ denoted by $f'$ such that $f'|_{\R_+^3\times\R}=f$ and
$$
\|f'\|_{B_{p,q}^{\sigma,\sigma/2}(\R^3\times\R)}\le c\|f\|_{B_{p,q}^{\sigma,\sigma/2}(\R_+^3\times\R)},
$$
where $c$ does not depend on $f$.
\end{lemma}

\begin{lemma}[see \cite{N}]\label{l2.8}
Let $f=f(x,t)\in B_{p,q}^{\sigma,\sigma/2}(\Omega\times\R_+)$, $\sigma\in\R_+$, $\sigma>2/p$, $p,q\in(1,\infty)$, $x\in\Omega$, $t\in\R_+$, where $\Omega$ stands for either $\R_+^3$ or $\R^3$.\\
Then $f|_{t=0}=\varphi\in B_{p,q}^{\sigma-2/p}(\Omega)$ and
$$
\|\varphi\|_{B_{p,q}^{\sigma-2/p}(\Omega)}\le c\|f\|_{B_{p,q}^{\sigma,\sigma/2}(\Omega\times\R_+)},
$$
where $c$ does not depend on $f$.
\end{lemma}

\begin{lemma}[see \cite{N}]\label{l2.9}
Let $\varphi^{(k)}\in B_{p,q} ^{\sigma-2/p-2k}(\Omega)$, where $\Omega$ stands for either $\R_+^3$ or $\R^3$, and
$$
k=0,1,\dots,l,\quad l=
\begin{cases}
\sigma/2-1/p-1&\textrm{ if }{\sigma\over2}-{1\over p}\in\N,\cr
\left[{\sigma\over2}-{1\over p}\right]&\textrm{ if }{\sigma\over2}-{1\over p}\not\in\N.\cr
\end{cases}
$$
Then there exists a function $f=f(x,t)\in B_{p,q}^{\sigma,\sigma/2}(\Omega\times\R)$ such that \\ $\partial_t^kf|_{t=0}=\varphi^{(k)}$, $k=0,1,\dots,l$, and
$$
\|f\|_{B_{p,q}^{\sigma,\sigma/2}(\Omega\times\R_+)}\le c\sum_{k=0}^l\|\varphi^{(k)}\|_{B_{p,q}^{\sigma-2/p-2k}(\Omega)},
$$
where $c$ does not depend on $f$.
\end{lemma}

Let $\Omega$ be either $\R^3$ or $\R_+^3$. In Definitions \ref{d2.4} and \ref{d2.5} the Besov spaces are defined by using Fourier transforms. Applying differences we can define the Besov space in more classical way (see papers \cite{BIN}, \cite{N} of Besov and Nikolskii).

To define the spaces we introduce differences. Let $x,z\in\Omega$, $t\in\R^1$. Then for $\N\ni m>1$ we set
\begin{equation}\eqal{
&\Delta_i(t)f(x)=f(x+te_i)-f(x),\cr
&\Delta_i^m(t)f(x)=\Delta_i(t)[\Delta_i^{m-1}(t)f(x)]=\sum_{j=0}^m(-1)^{m-j}c_{jm}f(x+jte_i)\cr
&\Delta(z)f(x)=f(x+z)-f(x),\cr
&\Delta^m(z)f(x)=\Delta(z)[\Delta^{m-1}(z)f(x)]=\sum_{j=0}^m(-1)^{m-j}c_{jm}f(x+jz),\cr}
\label{2.1}
\end{equation}
where $c_{jm}={m\choose j}$ and $e_i$ is the unit vector directed along the axis $x_i.$

\begin{definition}[see \cite{BIN}, \cite{N}]\label{d2.10}
The Besov space $B_{p,q}^\sigma(\Omega)$, where $\Omega$ stands either for $\R^3$ or $\R_+^3$, is a space of functions $u=u(x)$ with the finite norm
$$
\|u\|_{B_{p,q}^\sigma(\Omega)}=\|u\|_{L_p(\Omega)}+\sum_{i=1}^3\bigg(\intop_0^{h_0} {\|\Delta_i^m(h,\Omega)\partial_{x_i}^ku\|_{L_p(\Omega)}^q\over h^{1+q(\sigma-k)}}dh\bigg)^{1/q},
$$
where $x=(x_1,x_2,x_3)$, $m>\sigma-k$, $k\in\N_0=\N\cup\{0\}$, $\sigma\in\R_+$ and $\sigma>k$.

Moreover, we assumed that
\begin{equation}
\Delta_i^m(h,\R_+^3)u(x)=\begin{cases}\Delta_i^m(h)u(x)&if \ \ [x,x+hme_i]\in\R_+^3,\cr 0&\rm otherwise.\cr
\end{cases}
\label{2.2}
\end{equation}
Now, let
\begin{equation}
\Delta_t^{m_0}(h)u(t)=\sum_{j=0}^{m_0}(-1)^{m_0-j}c_{jm_0}u(t+jh)
\label{2.3}
\end{equation}
and
$$
\Delta_t^{m_0}(h,(\tau,T))u(t)=\begin{cases}
\Delta_t^{m_0}(h)u(t)&if [t,t+m_0h]\subset(\tau,T),\cr 0&\rm otherwise,\cr
\end{cases}
$$
where $m_0\in\N$ and $c_{jm_0}$ is defined above.
\end{definition}

\begin{definition}[see \cite{BIN}, \cite{N}]\label{d2.11}
The Besov space $B_{p,q}^{\sigma,\sigma/2}(\Omega\times(\tau,T))$, where $-\infty<\tau<T<\infty$, is a space of functions $u=u(x,t)$ with the finite norm
$$\eqal{
&\|u\|_{B_{p,q}^{\sigma,\sigma/2}(\Omega\times(\tau,T))}= \|u\|_{L_p(\Omega\times(\tau,T))}\cr
&\quad+\sum_{i=1}^3\bigg(\intop_0^{h_0} {\|\Delta_i^m(h,\Omega)\partial_{x_i}^{k_i}u\|_{L_p(\Omega\times(\tau,T)}^q\over h^{1+q(\sigma-k)}}dh\bigg)^{1/q}\cr
&\quad+\bigg(\intop_0^{h_0}{\|\Delta_t^{m_0}(h,(\tau,T)) \partial_t^{k_0}u\|_{L_p(\Omega\times(\tau,T))}^q\over h^{1+q(\sigma/2-k_0)}}dh\bigg)^{1/q},\cr}
$$
where $h_0>0$, $m>\sigma-k$, $m_0>\sigma/2-k_0$, $k,k_0\in\N_0$, $\sigma\in\R_+$, $\sigma>k$ and $\sigma/2>k_0$.
\end{definition}

From \cite{G} and Theorem 18.2 from \cite{BIN}, Ch. 4, Sect. 18 we have

\begin{lemma}\label{l2.12}
Norms of spaces $B_{p,q}^{\sigma,\sigma/2}(\Omega\times(\tau,T))$ are equivalent for any open set in $\Omega\times(\tau,T)$ and for any $m,k$ and $m_0,k_0$ satisfying conditions $m+k>\sigma>k>0$, $m_0+k_0>\sigma/2>k_0>0$.
\end{lemma}

Lemma \ref{l2.12} implies that the norm of $B_{p,q}^{\sigma,\sigma/2}(\Omega\times(\tau,T))$ from Definition \ref{d2.11} is equivalent to the following:
\begin{equation}\eqal{
&\|u\|_{B_{p,q}^{\sigma,\sigma/2}(\Omega\times(\tau,T))}= \|u\|_{L_p(\Omega\times(\tau,T))}\cr
&\quad+\sum_{i=1}^3\bigg(\intop_0^{h_0} {\|\Delta_i(h,\Omega)\partial_{x_i}^{[\sigma]}u\|_{L_p(\Omega\times(\tau,T))}^q\over h^{1+q(\sigma-[\sigma])}}dh\bigg)^{1/q}\cr
&\quad+\bigg(\intop_0^{h_0} {\|\Delta_t(h,(\tau,T))\partial_t^{[\sigma/2]}u\|_{L_p(\Omega\times(\tau,T))}^q\over h^{1+q(\sigma/2-[\sigma/2])}}dh\bigg)^{1/q}.\cr}
\label{2.4}
\end{equation}
From Lemma 7.44 from \cite{A} we have

\begin{lemma}\label{l2.13}
Let $p=q$. Then the norm (\ref{2.4}) is equivalent to the following
\begin{equation}\eqal{
&\|u\|_{\tilde B_{p,p}^{\sigma,\sigma/2}(\Omega\times(\tau,T))}= \|u\|_{L_p(\Omega\times(\tau,T))}\cr
&\quad+\bigg(\intop_\tau^Tdt\intop_\Omega dx'\intop_\Omega dx'' {|D_{x'}^{[\sigma]}u(x',t)-D_{x''}^{[\sigma]}u(x'',t)|^p\over |x'-x''|^{n+p(\sigma-[\sigma])}}\bigg)^{1/p}\cr
&\quad+\bigg(\intop_\Omega dx\intop_\tau^Tdt'\intop_\tau^Tdt'' {|\partial_{t'}^{[\sigma/2]}u(x,t')-\partial_{t''}^{[\sigma/2]}u(x,t'')|^p\over |t'-t''|^{1+p(\sigma/2-[\sigma/2])}}\bigg)^{1/p},\cr}
\label{2.5}
\end{equation}
where $D_x^{[\sigma]}= \partial_{x_1}^{\sigma_1}\partial_{x_2}^{\sigma_2}\partial_{x_3}^{\sigma_3}$, $\sigma_i\in\N_0$, $i=1,2,3$, $\sigma_1+\sigma_2+\sigma_3=[\sigma]$.
\end{lemma}

The norm (\ref{2.5}) for $\sigma\not\in\N$ is denoted also by $\|u\|_{W_p^{\sigma,\sigma/2}(\Omega\times(\tau,T))}$, where $W_p^{\sigma,\sigma/2}(\Omega\times(\tau,T))$ is called the Sobolev-Slobodetskii space.

\begin{lemma}\label{l2.14}
Let $f\in W_p^{\sigma,\sigma/2}(\Omega^T)$, $\sigma\not\in\N$, $p\in(1,\infty)$. Let
$$
f'=\begin{cases}
f&\textrm{for $t>0$,}\cr 0&\textrm{for $t<0$.}\cr
\end{cases}
$$
Then
\begin{equation}\eqal{
&\|f'\|_{W_p^{\sigma,\sigma/2}(\Omega\times(-\infty,T))}= \|f\|_{W_p^{\sigma,\sigma/2}(\Omega\times(0,T))}\cr
&\quad+c\bigg(\intop_0^T\intop_\Omega {|\partial_t^{[\sigma/2]}f|^p\over t^{p(\sigma/2-[\sigma/2])}}dxdt\bigg)^{1/p}.\cr}
\label{2.6}
\end{equation}
If $p\big({\sigma\over2}-\big[{\sigma\over2}\big]\big)>1$ then $\partial_t^{[\sigma/2]}f|_{t=0}=0$.
\end{lemma}

This is a compatibility condition.

\begin{proof}
We have
$$\eqal{
&\|f'\|_{W_p^{\sigma,\sigma/2}(\Omega\times(-\infty,T))}= \|f'\|_{L_p(\Omega\times(-\infty,T))}\cr
&\quad+\bigg(\intop_{-\infty}^Tdt\intop_\Omega dx'\intop_\Omega dx'' {|D_{x'}^{[\sigma]}f(x',t)-D_{x''}^{[\sigma]}f(x'',t)|^p\over |x'-x''|^{3+p(\sigma-[\sigma])}}\bigg)^{1/p}\cr
&\quad+\bigg(\intop_\Omega dx\intop_{-\infty}^Tdt'\intop_{-\infty}^Tdt'' {|\partial_{t'}^{[\sigma/2]}f'(x,t')-\partial_{t''}^{[\sigma/2]}f'(x,t'')|^p\over |t'-t''|^{1+p(\sigma/2-[\sigma/2])}}\bigg)^{1/p}\cr
&\equiv I_1+I_2+I_3.\cr}
$$
In view of the definition of $f'$ we have
$$\eqal{
&I_1=\|f\|_{L_p(\Omega\times(0,T))},\cr
&I_2=\bigg(\intop_0^Tdt\intop_\Omega dx'\intop_\Omega dx'' {|D_{x'}^{[\sigma]}f(x',t)-D_{x''}^{[\sigma]}f(x'',t)|^p\over |x'-x''|^{3+p(\sigma-[\sigma])}}\bigg)^{1/p}\cr}
$$
and
$$\eqal{
I_3^p&=\intop_\Omega dx\intop_0^Tdt'\intop_0^Tdt'' {|\partial_{t'}^{[\sigma/2]}f(x,t')-\partial_{t''}^{[\sigma/2]}f(x,t'')|^p\over |t'-t''|^{1+p(\sigma/2-[\sigma/2])}}\cr
&\quad+\intop_\Omega dx\intop_0^Tdt'\intop_{-\infty}^0dt'' {|\partial_{t'}^{[\sigma/2]}f(x,t')|^p\over|t'-t''|^{1+p(\sigma/2-[\sigma/2])}} \equiv I_{31}^p+I_{32}^p.\cr}
$$
Integrating with respect to $t''$ in $I_{32}$ yields
$$
I_{32}^p=c\intop_\Omega dx\intop_0^Tdt'{|\partial_{t'}^{[\sigma/2]}f(x,t')|^p\over |t'|^{p(\sigma/2-[\sigma/2])}}.
$$
Hence (\ref{2.6}) holds.

If $p(\sigma/2-[\sigma/2])>1$ then the condition $I_{32}<\infty$ implies that
$$
\|\partial_t^{[\sigma/2]}f(x,t)\|_{L_p(\Omega)}|_{t=0}=0.
$$
Consequently, conditions $I_{32}<\infty$ and $p\big({\sigma\over2}-\big[{\sigma\over2}\big]\big)>1$ imply that\break $\|\partial_t^{[\sigma/2]}f(x,t)\|_{L_p(\Omega)}$ must converge to zero sufficiently fast as $t$ goes to zero. This ends the proof.
\end{proof}

\begin{lemma}\label{l2.15}
Let $f\in W_p^{\sigma,\sigma/2}(\Omega^T)$, $\sigma\not\in\N$. Then
\begin{equation}
\bigg(\intop_0^T\intop_\Omega{|\partial_t^{[\sigma/2]}f(x,t)|^p\over t^{p(\sigma/2-[\sigma/2])}}dxdt\bigg)^{1/p}\le c\|f\|_{W_p^{\sigma,\sigma/2}(\Omega^T)},
\label{2.7}
\end{equation}
where $c$ does not depend on $f$.
\end{lemma}

\begin{proof}
Let $\sigma/2-[\sigma/2]<1/p$. Then Lemma 2 from \cite{S1} implies
\begin{equation}\eqal{
&\bigg(\intop_0^\infty\intop_\Omega{|\partial_t^{[\sigma/2]}f(x,t)|^p\over t^{p(\sigma/2-[\sigma/2])}}dxdt\bigg)^{1/p}\le c\|f\|_{L_p(\Omega;W_p^{\sigma/2-[\sigma/2]}(0,\infty))}\cr
&\le c\|f\|_{W_p^{\sigma,\sigma/2}(\Omega^T)},\cr}
\label{2.8}
\end{equation}
where an appropriate extension with respect to time from $(0,T)$ to $(0,\infty)$ was used.

Let $1/p<\sigma/2-[\sigma/2]<1+1/p$ and let $\partial_t^{[\sigma/2]}f|_{t=0}=0$. Then Lemma 2 from \cite{S1} implies (\ref{2.8}) as well. This concludes the proof.
\end{proof}

\begin{definition}(see [Tr 1, Sect. 2.3.5])\label{d2.16}
Let $L>0$ be a given natural number. By $A_{aL}(\R^4)$ we denote the collection $\varphi=\{\varphi_j(\bar x)\}_{j=0}^\infty \in S(\R^4)$ of functions with compact supports such that
$$\eqal{
C(\varphi)&=\sup_{\bar x\in\R^4}|\bar x|_a^L\sum_{|\bar\alpha|\le L}|D_{\bar x}^{\bar\alpha}\varphi_0(\bar x)|\cr
&\quad+\sup_{\bar x\in\R^4\backslash\{0\}\atop j=1,2,\dots}(|\bar x|_a^L+|\bar x|_a^{-L})\sum_{|\bar\alpha|\le L}|D_{\bar x}^{\bar\alpha}\varphi_j(2^jx,2^{2j} x_0)|<\infty,\cr}
$$
where $D_{\bar x}^{\bar\alpha}=\partial_{x_0}^{\alpha_0}\partial_{x_1}^{\alpha_1} \partial_{x_2}^{\alpha_2}\partial_{x_3}^{\alpha_3}$ and $\sum_{i=0}^4\alpha_i=|\bar\alpha|$.
\end{definition}

\begin{definition}[Partition of unity] (see Ch.4, Sect. 4)\cite{LSU}\label{d2.17}
Let $\Omega$ be the cylindrical domain with boundary $S=S_1\cup S_2$ defined previously.. We assume that $S_2$ is flat and $S_1$, $S_2$ meet along a curve $L=\bar S_1\cap\bar S_2$ under the angle $\pi/2$. The boundary $S_1$ must be sufficiently smooth so that at each point of $S_1$ there must exist a tangent plane.

Let $\bar n(\xi_k)$, $k\in{\frak N}_1$, be the unit outward vector normal to $S_1$ at the point~$\xi_k$.

A Cartesian coordinate system $y=(y_1,y_2,y_3)$ with origin at $\xi_k$, $k\in{\frak N}_1$ and the $y_3$ axis directed along $\bar n(\xi_i)$, is usually called a local coordinate system. Similarly, we can introduce a local coordinate system with origin at $\xi_k\in S_2$, $k\in{\frak N}_2$ and at $\xi_k\in L_i$, $k\in{\frak N}_3$, where $L_i=\bar S_1\cap\bar S_2(a_i)$, $i=1,2$.

Assuming that $\Omega$ is defined by a global Cartesian system $x=(x_1,x_2,x_3)$ we can transform it to the local Cartesian system with origin at $\xi_k$, $k\in{\frak N}_1\cup{\frak N}_2\cup{\frak N}_3,$ by some composition of a translation and a rotation. We denote the mapping by $Y_k$. Restrict our definition to $S_1$ part of the boundary. It is assumed that there exists a number $d>0$ such that, in a sphere of radius $d$ with center ot any point $\xi_k\in S_1$, $k\in{\frak N}_1$, the surface $S_1$ is given in a local system at the point $\xi_k$ by the equation
\begin{equation}
y_3=F_k(y'),\quad y'=(y_1,y_2),\quad k\in{\frak N}_1,
\label{2.9}
\end{equation}
where $F_k$ is a single-valued function.

We will say that $S_1\in C^m$ if $F_k(y')\in C^m(K_0)$, $k\in{\frak N}_1$, for any $\xi_k\in S_1$, where $K_0$ is the ball $|y'|\le d/2$ and if the norms $\|F_k\|_{C^m(K_0)}$ are bounded by a common constant.

We require that at least $S_1\in C^{1+\alpha}$, $\alpha\in(0,1]$. Under this assumption, in a neighborhood of $\xi_k$ we have the inequality
\begin{equation}
\bigg|{\partial F_k\over\partial y_\beta}\bigg|\le c|y'|^\alpha,\quad \beta=1,2.
\label{2.10}
\end{equation}
We assume that it is possible to construct in domain $\Omega$ for any $\lambda>0$, no matter how small, a finite number of subdomains $\omega^{(k)}$ and $\Omega^{(k)}$ possessing the following properties

\begin{itemize}
\item[1.] $\omega^{(k)}\subset\Omega^{(k)}\subset\Omega$, $\bigcup_k\omega^{(k)}=\bigcup_k\omega^{(k)}=\Omega$.
\item[2.] For any point $x\in\Omega$ there exists an $\omega^{(k)}$ such that $x\in\omega^{(k)}$ and the distance from to $\Omega\setminus\Omega^{(k)}$ is not less than $d\lambda$.
\item[3.] There exists a number $N_0,$ not depending on $\lambda,$ such that the intersection of any $N_0+1$ distinct sets $\Omega^{(k)}$ (and consequently any $N_0+1$ distinct sets $\omega^{(k)}$) is empty.
\item[4.] Sets $\omega^{(k)}$ and $\Omega^{(k)}$, $k\in{\frak M}$, are separated from the boundary $S=S_1\cup S_2$ by a positive distance. We assume that they are 3-dimensional cubes with common center $\xi_k\in\Omega$ whose linear dimensions are equal $\beta\lambda$ and $2\beta\lambda$ $(\beta>0)$, respectively.
\end{itemize}

The sets $\omega^{(k)}$ and $\Omega^{(k)}$, $k\in{\frak N}_1\cup{\frak N}_2$, are defined in local coordinates at a point $\xi_k\in S$ by the inequalities
$$\eqal{
&|y_\alpha|<\lambda/2,\quad &\alpha=1,2,\quad &0<y_3-F_k(y')<\lambda,\cr
&|y_\alpha|<\lambda,\quad &\alpha=1,2,\quad &0<y_3-F_k(y')<2\lambda.\cr}
$$
Consider $\omega^{(k)}$ and $\Omega^{(k)}$, $k\in{\frak N}_3$. Then we introduce a local system of coordinates with the origin at $\xi_k\in L_i$, $i=1,2$, such that $S_2$ is perpendicular to $y_3$ and $S_1$ is described by $y_1=F_k(y_2,y_3)$.

Therefore, $\omega^{(k)}$ and $\Omega^{(k)}$, $k\in{\frak N}_3$, are defined by
$$
|y_\beta|<\lambda/2,\quad \beta=2,3,\quad 0<y_1-F_k(y_2,y_3)<\lambda,
$$
$$
|y_\beta|<\lambda,\quad \beta=2,3,\quad 0<y_1-F_k(y_2,y_3)<2\lambda.
$$
The change of variables
\begin{equation}
z_\alpha=y_\alpha,\quad z_3=y_3-F_k(y'),\quad k\in{\frak N}_1
\label{2.11}
\end{equation}
transforms domains $\omega^{(k)}$, $\Omega^{(k)}$, $k\in{\frak N}_1$, into the cubes
$$\eqal{
&|z_\alpha|<\lambda/2,\quad &\alpha=1,2,\quad &0<z_3<\lambda,\cr
&|z_\alpha|<\lambda,\quad &\alpha=1,2,\quad &0<z_3<2\lambda.\cr}
$$
Since $S_2$ is flat, coordinates $x$, $y$ and $z$ can be taken as the same.

For $k\in{\frak N}_3$, the change of variables has the form
\begin{equation}
z_\alpha=y_\alpha,\quad \alpha=2,3,\quad z_1=y_1-F_k(y_2,y_3).
\label{2.12}
\end{equation}
We introduce functions $\zeta^{(k)}(x)$ having the properties
$$
0\le\zeta^{(k)}(x)\le 1,\quad |D_x^s\zeta^{(k)}|\le c_s/\lambda^s,\quad \zeta^{(k)}(x)=\begin{cases}1\ &{\rm for\ }x\in\omega^{(k)}\cr
0\ &{\rm for\ }x\in\Omega\setminus\Omega^{(k)}\cr\end{cases}
$$
By virtue of property 3 of the domains $\Omega^{(k)},$
$$
1\le\sum_k(\zeta^{(k)})^2\le N_0
$$
and hence the functions
$$
\eta^{(k)}(x)={\zeta^{(k)}(x)\over\sum_j\zeta^{(j)^2}(x)}
$$
possess the following properties:
$$
\eta^{(k)}=0\quad {\rm in}\ \ \Omega\setminus\Omega^{(k)},\quad |D^s\eta^{(k)}(x)|\le c_s\lambda^s
$$
and moreover,
$$
\sum_k\eta^{(k)}(x)\zeta^{(k)}(x)=1.
$$
\end{definition}
\begin{remark}\label{r2.18}
The norms of Besov spaces described by  Fourier transforms and the norms defined by differences are equivalent (see \cite{Am}).
\end{remark}

\begin{lemma}\label{l2.19}
Let $e_0(x_3)=e^{-\tau x_3}$, $e_1(x_3)={e^{-\tau x_3}-e^{-|\xi|x_3}\over\tau-|\xi|}$, $e_2(x_3)=e^{-|\xi|x_3}$, $\tau^2=s+\xi'^2$, $\xi'=(\xi_1,\xi_2)$, $s=\gamma+i\xi_0$, $\xi'\in\R^2$, $\xi_0\in\R$, ${\rm Re} s=\gamma>0$. Let $j\in\N\cup\{0\}\equiv\N_0$, $\partial_{\xi'}^{k'}=\partial_{\xi_1}^{k_1}\partial_{\xi_2}^{k_2}$, $k'=k_1+k_2$, $\varkappa\in(0,1)$.\\
Then for $e_0(x_3)$ we have
\begin{equation}\eqal{
&\sum_{2k+k'\le2}\intop_0^\infty|\partial_{x_3}^j\partial_{\xi_0}^k \partial_{\xi'}^{k'}e_0|^pdx_3\le c\sum_{2k+k'\le 2}|\tau|^{pj-p(2k+k')-1},\cr
&\sum_{j+2k+k'\le 2}\intop_0^\infty\intop_0^\infty {|\partial_{x_3}^j\partial_{\xi_0}^k\partial_{\xi'}^{k'}e_0(x_3+z)- \partial_{x_3}^j\partial_{\xi_0}^k\partial_{\xi'}^{k'}e_0(x_3)|^p\over z^{1+p\varkappa}}\cr
&\le c\sum_{j+2k+k'\le 2}c|\tau|^{p(j+\varkappa)-p(2k+k')-1}.\cr}
\label{2.13}
\end{equation}
\begin{equation}\eqal{
&\sum_{j+2k+k'\le 2}\intop_0^\infty\intop_0^\infty {|\partial_{x_3}^j\partial_{\xi_0}^k\partial_{\xi'}^{k'}e_0(x_3+z)- \partial_{x_3}^j\partial_{\xi_0}^k\partial_{\xi'}^{k'}e_0(x_3)|^p\over z^{1+p\varkappa}}dx_3dz\cr
&\le\sum_{j+2k+k'\le 2}c|\tau|^{p(j+\varkappa)-p(2k+k')-1}.\cr}
\label{2.14}
\end{equation}
Next $e_1(x_3)$ satisfies the estimates
\begin{equation}\eqal{
&\sum_{j+2k+k'\le 2}\intop_0^\infty|\partial_{x_3}^j\partial_{\xi_0}^k\partial_{\xi'}^{k'} e_1(x_3)|^pdx_3\cr
&\le c\sum_{j+2k+k'\le 2} {|\tau|^{pj-p(2k+k')-1}+|\xi|^{pj-p(2k+k')-1}\over |\tau|^p}\cr}
\label{2.15}
\end{equation}
\begin{equation}\eqal{
&\sum_{j+2k+k'\le 2}\intop_0^\infty\intop_0^\infty {|\partial_{x_3}^j\partial_{\xi_0}^k\partial_{x'}^{k'}e_1(x_3+z)- \partial_{x_3}^j\partial_{\xi_0}^k\partial_{\xi'}^{k'}e_1(x_3)|^p\over z^{1+2\varkappa}}dx_3dz\cr
&\le c\sum_{2k+k'\le 2} {|\tau|^{p(j+\varkappa)-p(2k+k')-1}+ |\xi|^{p(j+\varkappa)-p(2k+k')-1}\over|\tau|^p}\cr}
\label{2.16}
\end{equation}
Finally, $e_2=e_0-(\tau-|\xi|)e_1$.
\end{lemma}

\begin{proof}
First we prove (\ref{2.13}). Let $k=k'=j=0$. Then
\begin{equation}
\intop_0^\infty|e_0(x_3)|^pdx_3=\intop_0^\infty e^{-p{\rm Re}\,\tau x_3}dx_3\le {1\over p{\rm Re}\tau}\le{\sqrt{2}\over p}|\tau|^{-1},
\label{2.17}
\end{equation}
where we have used the fact that ${\rm Arg}\tau\in(-\pi/4,\pi/4)$. Thus $|\tau|\le\sqrt{2}{\rm Re}\tau$. Let $k=k'=0$ and $j\in\N$. Then
\begin{equation}
\partial_{x_3}^je_0=(-1)^j\tau^je^{-\tau x_3}
\label{2.18}
\end{equation}
and
\begin{equation}
\intop_0^\infty|\partial_{x_3}^je_0|^pdx_3\le|\tau|^{pj}\intop_0^\infty |e_0(x_3)|^pdx\le c|\tau|^{pj-1}
\label{2.19}
\end{equation}
Let $k=1$. Then
\begin{equation}
\partial_{xi_0}\partial_{x_3}^je_0=(c_1\tau^{j-1}+c_2\tau^jx_3)\tau_{,\xi_0}e^{-\tau x_3}.
\label{2.20}
\end{equation}
Since $|\tau_{,\xi_0}|\le c/|\tau|$ and
\begin{equation}
\sup_{x_3}x_3^pe^{-(p/2){\rm Re}\tau x_3}\le c/|\tau|^p
\label{2.21}
\end{equation}
we obtain
\begin{equation}
\intop_0^\infty|\partial_{\xi_0}\partial_{x_3}^je_0|^pdx_3\le c|\tau|^{pj-2p}\intop_0^\infty e^{-(p/2){\rm Re}\tau x_3}dx_3\le cc|\tau|^{pj-2p-1}.
\label{2.22}
\end{equation}
Let $k'=2$. Since $|\partial_{\xi'}\tau|\le c$, $|\partial_{\xi'}^2\tau|\le c/|\tau|$ we obtain
$$
|\partial_{\xi'}^2\partial_{x_3}^je_0|\le(c_1|\tau|^{j-2}+c_2|\tau|^{j-1}x_3+ c_3|\tau|^jx_3^2)e^{-{\rm Re}\tau x_3}.
$$
Using (\ref{2.21}) and
\begin{equation}
\sup_{x_3}x_3^{2p}e^{-(p/2){\rm Re}\tau x_3}\le c/|\tau|^2
\label{2.23}
\end{equation}
we obtain
\begin{equation}
\intop_0^\infty|\partial_\xi^2\partial_{x_3}^je_0|^pdx_3\le c|\tau|^{pj-2p-1}.
\label{2.24}
\end{equation}
Estimated (\ref{2.19}), (\ref{2.22}) adn (\ref{2.24}) imply (\ref{2.13}).

Next we prove (\ref{2.14}). For $k=k'=0$ and $j\in\N_0$ we have
$$\eqal{
&\intop_0^\infty\intop_0^\infty {|\partial_{x_3}^je_0(x_3+z)- \partial_{x_3}^je_0(x_3)|^p\over z^{1+p\varkappa}}dx_3dz\cr
&\le c\intop_0^\infty\intop_0^\infty|\tau|^{pj} {|e^{-\tau(x_3+z)}-e^{-\tau x_3}|^p\over z^{1+p\varkappa}}dx_3dz\cr
&\le c\intop_0^\infty|\tau|^{pj}e^{-{\rm Re}\tau px_3}dx_3\intop_0^\infty {|e^{-\tau z}-1|^p\over z^{1+\varkappa p}}dz\cr
&\le c|\tau|^{p(j+\varkappa)-1}.\cr}
$$
Finally, we prove (\ref{2.14}) for $k=1$, $k'=0$, $j\in\N_0$. Then we get
$$
\intop_0^\infty\intop_0^\infty {|\partial_{\xi_0}\partial_{x_3}^je_0(x_3+z)- \partial_{\xi_0}\partial_{x_3}^je(x_3)|^p\over z^{1+p\varkappa}}dx_3dz\equiv I_1.
$$
Using that
$$\eqal{
&\partial_{\xi_0}(\tau^je^{-\tau(x_3+z)})=(j\tau^{j-1}-\tau^j(x_3+z)) \partial_{\xi_0}\tau e^{-\tau(x_3+z)}\cr
&\partial_{\xi_0}(\tau^je^{-\tau x_3})=(j\tau^{j-1}-\tau^jx_3)\partial_{\xi_0} \tau e^{-\tau x_3}\cr}
$$
we obtain
$$\eqal{
&\partial_{\xi_0}(\tau^je^{-\tau(x_3+z)})-\partial_{\xi_0}(\tau^je^{-\tau x_3})= -\tau^jze^{-\tau(x_3+z)}\partial_{\xi_0}\tau\cr
&\quad+(j\tau^{j-1}-\tau^jx_3)(e^{-\tau(x_3+z)}-e^{-\tau x_3})\partial_{\xi_0}\tau.\cr}
$$
Using that $|\tau_{,\xi_0}|\le c/|\tau|$ we obtain
$$\eqal{
I_1&\le c|\tau|^{p(j-1)}\intop_0^\infty\intop_0^\infty {|z|^pe^{-p{\rm Re}\tau\,(x_3+z)}\over|z|^{1+\varkappa p}}dx_3dz\cr
&\quad+c|\tau|^{p(j-2)}\intop_0^\infty\intop_0^\infty {|e^{-\tau(x_3+z)}-e^{-\tau x_3}|^p\over|z|^{1+\varkappa p}}dx_3dz\cr
&\quad+c|\tau|^{p(j-1)}\intop_0^\infty\intop_0^\infty x_3^p {|e^{-\tau(x_3+z)}-e^{-\tau x_3}|^p\over|z|^{1+\varkappa p}}dx_3dz\cr
&\equiv J_1+J_2+J_3.\cr}
$$
To estimate $J_1$ we use
$$\eqal{
&\intop_0^\infty\intop_0^\infty{|z^pe^{-p{\rm Re}\tau(x_3+z)}\over z^{1+\varkappa p}}dzdx_3=\intop_0^\infty e^{-p{\rm Re}\tau x_3}dx_3 \intop_0^\infty{z^pe^{-p{\rm Re}\tau z}\over z^{1+\varkappa p}}dz\cr
&={1\over p{\rm Re}\tau}\intop_0^\infty{y^pe^{-py}\over y^{1+\varkappa p}}dy ({\rm Re}\tau)^{\varkappa p-p}\le c|\tau|^{\varkappa p-p-1}.\cr}
$$
To estimate $J_2$ we need
$$\eqal{
&\intop_0^\infty\intop_0^\infty {|e^{-\tau(x_3+z)}-e^{-\tau x_3}|^p\over |z|^{1+\varkappa p}}dxdx_3=\intop_0^\infty e^{-p{\rm Re}\tau x_3}dx_3\cdot\cr
&\quad\cdot\intop_0^\infty{|e^{-{\tau\over|\tau|}y}-1|^p\over|y|^{1+\varkappa p}}dy|\tau|^{\varkappa p}\le c|\tau|^{\varkappa p-1}.\cr}
$$
Finally, we examine $J_3$. Then we have
$$\eqal{
&\intop_0^\infty\intop_0^\infty{x_3^p|e^{-\tau(x_3+z)}-e^{-\tau x_3}|^p\over z^{1+\varkappa p}}dx_3dz\cr
&=\intop_0^\infty x_3^pe^{-p{\rm Re}\tau x_3}dx_3\intop_0^\infty{|e^{-\tau z}-1|^p\over z^{1+\varkappa p}}dz\equiv J_3^1,\cr}
$$
where the first integral is bounded by
$$
\sup_{x_3}x_3^pe^{-(p/2){\rm Re}\tau x_3}\intop_0^\infty e^{-(p/2){\rm Re}\tau x_3}dx_3\le c|{\rm Re}\tau|^{-p-1}.
$$
Hence
$$
J_3^1\le c|\tau|^{\varkappa p-p-1}.
$$
Summarizing, we proved (\ref{2.14}) in the case $k=1$, $j\in\N_0$. A corresponding estimate holds for $k'=2$, $j\in\N_0$.

Next, we show (\ref{2.15}). We consider $e_1(x_3)$ as the convolution
\begin{equation}
e_1(x_3)=-\intop_0^{x_3}e^{-|\xi|(x_3-y)}e^{-\tau y}dy.
\label{2.25}
\end{equation}
Hence, by the Young inequality, we obtain
$$
\intop_0^\infty|e_1(x_3)|^pdx_3\le\bigg(\intop_0^\infty e^{-{\rm Re}\tau\, y}dy\bigg)^p\intop_0^\infty e^{-p|\xi|x_3}dx_3\le{1\over|\tau|^p|\xi|}.
$$
We calculate
$$
\partial_{\xi_0}e_1(x_3)=\intop_0^{x_3}e^{-|\xi|(x_3-y)}ye^{-\tau y}\partial_{\xi_0}\tau dy.
$$
Using $|\partial_{\xi_0}\tau|\le c/|\tau|$ and the Young inequality yield
$$
\intop_0^\infty|\partial_{\xi_0}e_1(x_3)|^pdx_3\le{c\over|\tau|^{3p}|\xi|}.
$$
In view of the relation
$$
{d^je_1(x_3)\over dx_3^j}=(-|\xi|)^je_1(x_3)+(-1)^j(\tau^{j-1}+\tau^{j-2}|\xi|+ \cdots+|\xi|^{j-1})e_0(x_3)
$$
and the estimate
$$
\intop_0^\infty\bigg|{d^j\over dx_3^j}e_0(x_3)\bigg|^pdx_3\le c|\tau|^{2j-1}
$$
we obtain
$$
\intop_0^\infty\bigg|{d^j\over dx_3^j}e_1(x_3)\bigg|^pdx_3\le c{|\tau|^{pj-1}+|\xi|^{pj-1}\over|\tau|^p}.
$$
From the proof of Lemma \ref{l3.1} from \cite{S2} we have
$$
\intop_0^\infty\bigg|{d^je_1(x_3)\over dx_3^j}\bigg|^pdx_3\le c{|\tau|^{pj-1}+|\xi|^{pj-1}\over|\tau|^p}.
$$
Consider
$$\eqal{
&\intop_0^\infty|\partial_{\xi_0}e_1(x_3)|^pdx_3\le\intop_0^\infty {1\over|\tau-|\xi||^p}|\tau_{,\xi_0}|^p|e_1|^pdx_3\cr
&\quad+\intop_0^\infty\bigg|{x_3e^{-\tau x_3}\tau_{,\xi_0}\over\tau-|\xi|} \bigg|^pdx_3\equiv I_1+I_2.\cr}
$$
First we estimate
$$
I_1\le{c\over|\tau|^p}{1\over|\tau-|\xi||^p}\intop_0^\infty|e_1|^pdx_3\le {c\over|\tau-|\xi||^p|\tau|^{2p}|\xi|},
$$
where we used estimate of $\intop_0^\infty|e_1|^pdx_3$ from the proof of Lemma l3.1 in \cite{S2}.

Next
$$
I_2\le{c\over|\tau|^p}{1\over|\tau-|\xi||^p}\intop_0^\infty x_3^pe^{-p{\rm Re}\tau x_3}dx_3\le{c\over|\tau-|\xi||^p|\tau|^{2p}|\tau|}
$$
Summarizing, we obtain
$$
\intop_0^\infty|\partial_{\xi_0}e_1(x_3)|^pdx_3\le{c\over|\tau-|\xi||^p|\tau|^{2p}} \bigg({1\over|\tau|}+{1\over|\xi|}\bigg).
$$
Continuing the considerations we derive (\ref{2.15}).

Consider the function
$$
e_1(x_3+z)-e_1(x_3)=e^{-\tau x_3}e_1(z)+e_1(x_3)(e^{-|\xi|z}-1).
$$
We examine
$$\eqal{
I&\equiv\intop_0^\infty|(e_1(x_3+z)-e_1(x_3))_{,\xi_0}|^pdx_3\cr
&=\intop_0^\infty\intop_0^\infty|(e^{-\tau x_3}e_1(z))_{,\xi_0}+e_{1,\xi_0}(x_3) (e^{-|\xi|z}-1)|^p{dx_3dz\over z^{1+p\varkappa}}\cr
&\le\intop_0^\infty\intop_0^\infty|x_3e^{-\tau x_3}\tau_{,\xi_0}e_1(z)|^p {dx_3dz\over z^{1+p\varkappa}}\cr
&\quad+\intop_0^\infty\intop_0^\infty|e^{-\tau x_3}e_{1,\xi_0}(z)|^p {dx_3dz\over z^{1+p\varkappa}}\cr
&\quad+\intop_0^\infty\intop_0^\infty|e_{1,\xi_0}(x_3)(e^{-|\xi|z}-1)|^p {dx_3dz\over z^{1+p\varkappa}}\equiv I_1+I_2+I_3.\cr}
$$
First we estimate $I_1$,
$$\eqal{
I_1&\le{c\over|\tau|^p}\intop_0^\infty|x_3e^{-\tau x_3}|^pdx_3\intop_0^\infty |e_1(z)|^p{dz\over z^{1+p\varkappa}}\cr
&\le{c\over|\tau|^{2p+1}}\intop_0^\infty|e_1(z)|^p{dz\over z^{1+p\varkappa}}\equiv I_1^1.\cr}
$$
From the proof of Lemma 3.1 in \cite{S2} we have
\begin{equation}
\intop_0^\infty{|e_1(z)|^p\over z^{1+p\varkappa}}dz\le c|\xi|^{p\varkappa-2}.
\label{2.26}
\end{equation}
Using (\ref{2.26}) yields
$$
I_1^1\le{c\over|\tau|^{2p+1}}|\xi|^{p\varkappa-2}.
$$
To estimate $I_2$ we use the formula
\begin{equation}
e_{1,\xi_0}(z)={-ze^{-\tau z}\over\tau-|\xi|}\tau_{,\xi_0}-{e^{-\tau z}-e^{-|\xi|z}\over(\tau-|\xi|)^2}\tau_{,\xi_0}
\label{2.27}
\end{equation}
Then
$$\eqal{
I_2&\le\intop_0^\infty\intop_0^\infty\bigg|{e^{-\tau x_3}ze^{-\tau z}\over \tau-|\xi|}\tau_{,\xi_0}\bigg|^p{dx_3dz\over z^{1+p\varkappa}}\cr
&\quad+\intop_0^\infty\intop_0^\infty\bigg|{e^{-\tau x_3}(e^{-\tau z}-e^{-|\xi|z})\over(\tau-|\xi|)^2}\tau_{,\xi_0}\bigg|^p{dx_3dz\over z^{1+p\varkappa}}\cr
&\equiv I_2^1+I_2^2.\cr}
$$
Consider $I_2^1$. It is bounded by
$$\eqal{
I_2^1&\le{1\over|\tau-|\xi||^p}{c\over|\tau|^p}\intop_0^\infty e^{-p{\rm Re}\tau x_3}dx_3\intop_0^\infty{|z|^p\over z^{1+p\varkappa}}e^{-p{\rm Re}\tau z}dz\cr
&\le{c\over|\tau-|\xi||^p}{1\over|\tau|^{p+1}}\intop_0^\infty{|z|^p\over z^{1+p\varkappa}}e^{-p{\rm Re}\tau z}dz\equiv I_2^{11}.\cr}
$$
Changing variables $y={\rm Re}\tau z$ in the integral yields
$$\eqal{
I_2^{11}&={c\over|\tau-|\xi||^p}{1\over|\tau|^{p+1}}|\tau|^{p(\varkappa-1)} \intop_0^\infty{|y|^p\over y^{1+p\varkappa}}e^{-py}dy\cr
&={c\over|\tau-|\xi||^p}{1\over|\tau|^p}|\tau|^{p(\varkappa-1)-1}.\cr}
$$
Next we examine $I_2^2$. We estimate it in the form
$$
I_2^2\le{1\over|\tau-|\xi||^p}{c\over|\tau|^p}\intop_0^\infty e^{-p{\rm Re}\tau x_3}dx_3\intop_0^\infty{|e_1(z)|^p\over z^{1+p\varkappa}}dz\equiv I_2^{21}.
$$
Performing integration in the first integral and using (\ref{2.26}) yield
$$
I_2^{21}\le{1\over|\tau-|\xi||^p}{c\over|\tau|^{p+1}}|\xi|^{p\varkappa-2}.
$$
To estimate $I_3$ we use the formula
\begin{equation}
e_{1,\xi_0}(x_3)={-x_3e^{-\tau x_3}\over\tau-|\xi|}\tau_{,\xi_0}-{e^{-\tau x_3}-e^{-|\xi|x_3}\over(\tau-|\xi|)^2}\tau_{,\xi_0}
\label{2.28}
\end{equation}
Then
$$\eqal{
I_3&\le{c\over|\tau|^p}\intop_0^\infty\bigg|{x_3e^{-\tau x_3}\over\tau-|\xi|}\bigg|^pdx_3\intop_0^\infty{|e^{-|\xi|z}-1|^p\over z^{1+\varkappa p}}dz\cr
&\quad+{c\over|\tau|^p}\intop_0^\infty\bigg|{e^{-\tau x_3}-e^{-|\xi|x_3}\over(\tau-|\xi|)^2}\bigg|^pdx_3\intop_0^\infty{|e^{-|\xi|z}-1|^p\over z^{1+\varkappa p}}dz\cr
&\equiv I_3^1+I_3^2.\cr}
$$
Consider $I_3^1$,
$$
I_3^1\le{c\over|\tau|^p}{1\over|\tau-|\xi||^p}{1\over|\tau|^{p+1}}\intop_0^\infty {|e^{-|\xi|z}-1|^p\over z^{1+\varkappa p}}dz\equiv I_3^{11}.
$$
Changing variables in the integral and using that
\begin{equation}
\intop_0^\infty{|e^{-y}-1|^p\over y^{1+\varkappa p}}dy<\infty
\label{2.29}
\end{equation}
we obtain
$$
I_3^{11}\le{c\over|\tau-|\xi||^p}{1\over|\tau|^{2p+1}}|\xi|^{\varkappa p}.
$$
Finally, we estimate
$$
I_3^2\le{c\over|\tau|^p}{1\over|\tau-|\xi||^p}\intop_0^\infty|e_1(x_3)|^pdx_3 \intop_0^\infty{|e^{-|\xi|z}-1\over z^{1+\varkappa p}}dz\equiv I_3^{21}.
$$
Using (see the proof of Lemma 3.1 in \cite{S2})
$$
\intop_0^\infty|e_1(x_3)|^pdx_3\le c{1\over|\tau|^p|\xi|}
$$
and (\ref{2.29}) we obtain
$$
I_3^{21}\le{c\over|\tau-|\xi||^p}{1\over|\tau|^{2p}|\xi|}|\xi|^{p\varkappa}= {c\over|\tau-|\xi||^p}{1\over|\tau|^{2p}}|\xi|^{p\varkappa-1}.
$$
Summarizing,
$$\eqal{
I&\le{c\over|\tau|^{2p+1}}|\xi|^{p\varkappa-2}+{c\over|\tau-|\xi||^p} |\tau|^{p(\varkappa-2)-1}\cr
&\quad+{c\over|\tau-|\xi||^p}{1\over|\tau|^{p+1}}|\xi|^{p\varkappa-2}\cr
&\quad+{c\over|\tau-|\xi||^p}{|\xi|^{\varkappa p}\over|\tau|^{2p}}\bigg({1\over|\tau|}+{1\over|\xi|}\bigg).\cr}
$$
Estimates for $\partial_{x_3}^j$ are derived in Lemma \ref{l3.1} in \cite{S2}. Hence (\ref{2.16}) holds. This ends the proof.
\end{proof}

\section{The Stokes system in the whole space}\label{se3}

We consider the following Stokes system
\begin{equation}\eqal{
&v_t-\nu\Delta v+\nabla p=f\quad &{\rm in}\ \ \R^3\times\R_+,\cr
&\divv v=g\quad &{\rm in}\ \ \R^3\times\R_+,\cr
&v|_{t=0}=v_0\quad &{\rm in}\ \ \R^3.\cr}
\label{3.1}
\end{equation}
We are looking for solutions to (\ref{3.1}) under the following assumptions
\begin{equation}\eqal{
&f\in W_p^{\sigma,\sigma/2}(\R^3\times\R_+),\cr
&g\in W_p^{\sigma+1,\sigma/2+1/2}(\R^3\times\R_+),\cr
&v_0\in W_p^{\sigma+2-2/p}(\R^3),\cr}
\label{3.2}
\end{equation}
where $p\in(1,\infty)$, $\sigma\not\in\N$.

Since $v_0\in W_p^{\sigma+2-2/p}(\R^3)$ there exists the time extension\\ $\tilde v_0\in W_p^{\sigma+2,\sigma/2+1}(\R^3\times\R_+)$ such that
\begin{equation}
\tilde v_0|_{t=0}=v_0
\label{3.3}
\end{equation}
and
\begin{equation}
\|\tilde v_0\|_{W_p^{\sigma+2,\sigma/2+1}(\R^3\times\R_+)}\le c\|v_0\|_{W_p^{\sigma+2-2/p}(\R^3)},
\label{3.4}
\end{equation}
where $c$ does not depend on $v_0$.

Having the extension $\tilde v_0$ we introduce the new function
\begin{equation}
\tilde v=v-\tilde v_0
\label{3.5}
\end{equation}
such that $(\tilde v,p)$ is a solution to the Stokes system with vanishing initial data
\begin{equation}\eqal{
&\tilde v_t-\nu\Delta\tilde v+\nabla p=f-\tilde v_{0,t}+\nu\Delta\tilde v_0 \equiv\tilde f,\cr
&\divv\tilde v=g-\divv\tilde v_0\equiv\tilde g\cr
&\tilde v|_{t=0}=0.\cr}
\label{3.6}
\end{equation}
Using Lemmas \ref{l2.14} and \ref{l2.15} we can extend functions $\tilde f$ and $\tilde g$ by zero for $t<0$. We denote the extended functions by
$$\eqal{
&f'=\begin{cases}\tilde f&\textrm{for $t>0$,}\cr 0&\textrm{for $t<0$,}\cr\end{cases}\cr
&g'=\begin{cases}\tilde g&\textrm{for $t>0$,}\cr 0&\textrm{for $t<0$.}\cr\end{cases}\cr}
$$
The extensions are possible if the following compatibility conditions hold
\begin{equation}\eqal{
&\partial_t^i\tilde f|_{t=0}=0\quad &{\rm for}\ \ i\le[\sigma/2],\cr
&\partial_t^i\tilde g|_{t=0}=0\quad &{\rm for}\ \ i\le[\sigma/2+1/2].\cr}
\label{3.7}
\end{equation}
Employing the extensions in (\ref{3.6}) we get the problem
\begin{equation}\eqal{
&v'_t-\nu\Delta v'+\nabla p'=f'\quad &{\rm in}\ \ \R^3\times\R,\cr
&\divv v'=g'\quad &{\rm in}\ \ \R^3\times\R.\cr}
\label{3.8}
\end{equation}
Let $E(x)=c|x|^{-1}$ be the fundamental solution to the Laplace equation. Then any solution to the equation
$$
\Delta\varphi=g'
$$
has the form
\begin{equation}
\varphi=E\star g',
\label{3.9}
\end{equation}
where $\star$ means convolution.

Introducing the new functions
\begin{equation}
v''=v'-\nabla\varphi,\quad p''=p',\quad f''=f'-\nabla\varphi_t+\nu\Delta\varphi
\label{3.10}
\end{equation}
we replace $v''$, $p''$, $f''$ by $v$, $p$, $f$. Then we see that $(v,p)$ is a solution to the problem
\begin{equation}\eqal{
&v_t-\nu\Delta v+\nabla p=f'-\nabla\varphi_t+\nu\Delta\varphi\equiv f,\cr
&\divv v=0.\cr}
\label{3.11}
\end{equation}
To solve (\ref{3.11}) we use the Fourier and the Fourier-Laplace transforms.

First we introduce the Fourier transform and its inverse
\begin{equation}\eqal{
&\tilde f(\xi,\xi_0)\equiv(Ff)(\xi,\xi_0)=\intop\intop_{\R^4}e^{-i\xi_0t-i\xi\cdot x} f(x,t)dxdt\cr
&(F^{-1}f)(x,t)=\intop\intop_{\R^4}e^{i\xi_0t+i\xi\cdot x}f(\xi,\xi_0)d\xi d\xi_0,\cr}
\label{3.12}
\end{equation}
where $x\cdot\xi=x_1\xi_1+x_2\xi_2+x_3\xi_3$.

Next, the Fourier-Laplace transform and its inverse are defined by
\begin{equation}\eqal{
&(F_2f)(\xi,\tau)\equiv\hat f(\xi,\tau)\equiv(Ff_\gamma)(\xi,\tau)=\intop\intop_{\R^4}e^{-st-i\xi\cdot x}f(x,t)dxdt,\cr
&(F_2^{-1}f)(x,t)=\intop\intop_{\R^4}e^{st+i\xi\cdot x}f(\xi,s)d\xi d\xi_0,\cr}
\label{3.13}
\end{equation}
where $s=\gamma+i\xi_0$, $\tau=\sqrt{s+\xi^2}$, $0<\gamma\in\R_+$.

Comparing (\ref{3.12}) and (\ref{3.13}) we have
\begin{equation}\eqal{
&(F_2f)(\xi,\tau)=F(e^{-\gamma t}f)(\xi,\xi_0),\cr
&(F_2^{-1}f)(x,t)=e^{\gamma t}F^{-1}(f)(x,t).\cr}
\label{3.14}
\end{equation}
Applying the Fourier-Laplace transform to (\ref{3.11}) yields
\begin{equation}\eqal{
&(\tau+\xi^2)\hat v(\xi,\tau)+i\xi\hat p(\xi,\tau)=\hat f(\xi,\tau),\cr
&i\xi\cdot\hat v(\xi,\tau)=0.\cr}
\label{3.15}
\end{equation}
Solving (\ref{3.15}) we get
\begin{equation}
\hat v(\xi,\tau)={P(\xi)\hat f(\xi,\tau)\over\tau+|\xi|^2},\quad
\hat p(\xi,\tau)={\xi\cdot\hat f(\xi,\tau)\over i|\xi|^2},
\label{3.16}
\end{equation}
where $P(\xi)=\{\delta_{jk}-\xi_j\xi_k/|\xi|^2\}_{j,k=1,2,3}$.

\begin{lemma}\label{l3.1}
Assume that $p\in(1,\infty)$, $\sigma\in\R_+$ and $f\in B_{p,p}^{\sigma,\sigma/2}(\R^3\times\R)$. Then there exists a unique solution to problem (\ref{3.11}) such that $v\in B_{p,p}^{\sigma+2,\sigma/2+1}(\R^3\times\R)$, $\nabla p\in B_{p,p}^{\sigma,\sigma/2}(\R^3\times\R)$ and
\begin{equation}\eqal{
&\|v\|_{B_{p,p}^{\sigma+2,\sigma/2+1}(\R^3\times\R)}\le c\|f\|_{B_{p,p}^{\sigma,\sigma/2}(\R^3\times\R)},\cr
&\|\nabla p\|_{B_{p,p}^{\sigma,\sigma/2}(\R^3\times\R)}\le c\|f\|_{B_{p,p}^{\sigma,\sigma/2}(\R^3\times\R)}.\cr}
\label{3.17}
\end{equation}
The existence and uniqueness of solutions follow from (\ref{3.16}).
\end{lemma}

\begin{proof}
First we consider
$$\eqal{
&\|v\|_{B_{p,p}^{\sigma+2,\sigma/2+1}(\R^3\times\R)}=\sum_{k=0}^\infty\bigg( \intop_{\R^4}\bigg|2^{(\sigma+2)k}F_2^{-1}\varphi_k{P(\xi)\hat f(\xi,\tau)\over\tau+|\xi|^2}\bigg|^p dxdt\bigg)^{1/p}\cr
&=\sum_{k=0}^\infty\bigg(\intop_{\R^4}\bigg|2^{(\sigma+2)k}F_2^{-1}\varphi_k {P(\xi)\over\tau+\xi^2}F_2f\bigg|^pdxdt\bigg)^{1/p}\equiv I_1,\cr}
$$
where
$$
F_2^{-1}\varphi_k{P(\xi)\over\tau+\xi^2}F_2f=F_2^{-1} \bigg(\varphi_k{P(\xi)\over\tau+\xi^2}F_2f\bigg).
$$
Let us introduce a family of functions $\{\psi_j(\bar\xi)\}$, where $\bar\xi=(\xi,\xi_0)$, such that $\supp\psi_0\subset\{\bar\xi\colon|\bar\xi|_a\le4\}$, $\supp\psi_j\subset\{\bar\xi\colon2^{j-2}\le|\bar\xi|_a\le2^{j+2}\}$, and $\psi_j(\bar\xi)=1$ for $\bar\xi\in\supp\varphi_j$. Then
$$\eqal{
I_1&=\sum_{k=0}^\infty\bigg(\intop_{\R^4}\bigg|\sum_{l=0}^\infty 2^{(\sigma+2)k}F_2^{-1}\psi_l{P(\xi)\over\tau+\xi^2}\varphi_k\varphi_lF_2f\bigg|^p dxdt\bigg)^{1/p}\cr
&=\sum_{k=0}^\infty\bigg(\intop_{\R^4}\bigg|\sum_{l=0}^\infty 2^{(\sigma+2)k}F_2^{-1}\psi_l {P(\xi)\over\tau+\xi^2}F_2F_2^{-1}\varphi_kF_2F_2^{-1}\varphi_lF_2f\bigg|^p dxdt\bigg)^{1/p}.\cr}
$$
Continuing, we have
{\small$$\eqal{
&I_1=\sum_{k=0}^\infty\bigg(\intop_{\R^4}\bigg|\sum_{l=0}^\infty 2^{(\sigma+2)k}F_2^{-1}\psi_l {P(\xi)\over\tau+\xi^2}F_2(F_2^{-1}\varphi_k\star F_2^{-1}\varphi_lF_2f)\bigg|^p dxdt\bigg)^{1/p}\cr
&=\sum_{k=0}^\infty\!\bigg(\intop_{\R^4}\bigg|\!\sum_{l=0}^\infty 2^{(\sigma+2)k}\!\!\!\intop_{\R^4}\!\!d\bar y\bigg(\!\!F_2^{-1}\psi_l{P(\xi)\over\tau+\xi^2}F_2F_2^{-1}\varphi_k(\bar y)(F_2^{-1}\varphi_lF_2f)(\bar x-\bar y)\bigg|^p\!d\bar x\bigg)^{\!\!1/p},\cr}
$$}
where $\bar x=(x,t)$.

We recall the formula
\begin{equation}\eqal{
&\bigg(F_2^{-1}\psi_l{P(\xi)\over\tau+\xi^2}F_2F_2^{-1}\varphi_k\bigg)(\bar y)= F_2^{-1}\bigg(F_2F_2^{-1}\psi_l{P(\xi)\over\tau+\xi^2}F_2F_2^{-1}\varphi_k\bigg) (\bar y)\cr
&=F_2^{-1}F_2\bigg(F_2^{-1}\psi_l{P(\xi)\over\tau+\xi^2}\star F_2^{-1}\varphi_k\bigg)(\bar y)\cr
&=\bigg(F_2^{-1}\psi_l{P(\xi)\over\tau+\xi^2}\star F_2^{-1}\varphi_k\bigg)(\bar y).\cr}
\label{3.18}
\end{equation}
Moreover, we have the notation $f_i(2^l\cdot,2^{2l}\cdot)$, $i=1,2$, where $\cdot$ replaces the corresponding argument in convolution. Then we have (see Lemma \ref{l3.4})
\begin{equation}\eqal{
&[F_2^{-1}(f_1(2^l\cdot,2^{2l}\cdot))\star F_2^{-1}(f_2(2^l\cdot,2^{2l}\cdot))](\bar y)\cr
&=2^{-5l}(F_2^{-1}f_1\star F_2^{-1}f_2)(2^{-l}\tilde y),\cr}
\label{3.19}
\end{equation}
where $\tilde y=(y,2^{-l}y_0)$.

With the help of (\ref{3.18}) and (\ref{3.19}), assuming that $f_1=\psi_l{P(\xi)\over\tau+\xi^2}$, $f_2=\varphi_k$ the expression $I_1$ takes the form
$$\eqal{
I_1&=\sum_{k=0}^\infty\bigg(\intop_{\R^4}\bigg|\sum_{l=0}^\infty 2^{(\sigma+2)k}2^{-5l}\intop_{\R^4}d\bar y\bigg(F_2^{-1}\psi_l{P(\xi)\over\tau+\xi^2}\cr
&\quad\star F_2^{-1}\varphi_k)(2^{-l}\tilde y)(F_2^{-1}\varphi_lF_2f)(\bar x-2^{-l}\tilde y)\bigg|^pd\bar x\bigg)^{1/p}.\cr}
$$
In view of (\ref{3.19}), we obtain
$$\eqal{
I_1&=\sum_{k=0}^\infty\bigg(\intop_{\R^4}d\bar x\bigg|\sum_{l=0}^\infty 2^{(\sigma+2)k}\intop_{\R^4}d\bar y\bigg[F_2^{-1}\psi_l{P(\xi)\over\tau+\xi^2}(2^l\cdot,2^{2l}\cdot)\cr
&\quad\star(F_2^{-1}\varphi_k)(2^l\cdot,2^{2l}\cdot)\bigg](\bar y)(F_2^{-1}\varphi_lF_2f)(\bar x-\bar y)\bigg|^p\bigg)^{1/p}.\cr}
$$
Applying the Minkowski inequality with respect to $\bar x$ yields
$$\eqal{
I_1&\le\sum_{k=0}^\infty\bigg(\sum_{l=0}^\infty\bigg|2^{(\sigma+2)k} \intop_{\R^4}d\bar y\bigg[\bigg(F_2^{-1}\psi_l{P(\xi)\over\tau+\xi^2}\bigg)(2^l\cdot,2^{2l}\cdot)\cr
&\quad\star(F_2^{-1}\varphi_k)(2^l\cdot,2^{2l}\cdot)\bigg](\bar y)\bigg( \intop_{\R^4}d\bar x|F_2^{-1}\varphi_lF_2f|(\bar x-\bar y)\bigg|^p\bigg)^{1/p} \bigg|^p\bigg)^{1/p}\cr
&\equiv I_2.\cr}
$$
Changing variables $\bar z=\bar x-\bar y$ in the integral above with respect to $\bar x$ implies
$$\eqal{
I_2&=\sum_{k=0}^\infty\bigg(\bigg|\sum_{l=0}^\infty 2^{(\sigma+2)k}\intop_{\R^4} d\bar y\bigg[F_2^{-1}\psi_l{P(\xi)\over\tau+\xi^2}\bigg)(2^l\cdot,2^{2l}\cdot)\cr
&\quad\star(F_2^{-1}\varphi_k)(2^l\cdot,2^{2l}\cdot)\bigg](\bar y)\bigg( \intop_{\R^4}d\bar z|(F_2^{-1}\varphi_lF_2)(\bar z)|^p\bigg)^{1/p}\bigg|^p \bigg)^{1/p}.\cr}
$$
Using the H\"older inequality in the integral with respect to $\bar y$ and replacing $\bar z$ by $\bar x$, we obtain
$$\eqal{
I_2&\le\sum_{k=0}^\infty\bigg(\bigg|\sum_{l=0}^\infty 2^{(\sigma+2)k}\bigg(\intop_{\R^4}d\bar y{1\over(1+|\bar y|_a)^d}\bigg)^{1/2}\cdot\cr
&\quad\cdot\bigg(\intop_{\R^4}d\bar y\bigg|\bigg[\bigg(F_2^{-1}\psi_l{P(\xi)\over\tau+\xi^2}\bigg)(2^l\cdot,2^{2l}\cdot) \star(F_2^{-1}\varphi_k)(2^l\cdot,2^{2l}\cdot)\bigg](\bar y)\cdot\cr
&\quad\cdot(1+|\bar y|_a)^{d/2}\bigg|^2\bigg)^{1/2}\bigg|^p\bigg)^{1/p} \bigg(\intop_{\R^4}d\bar x|(F_2^{-1}\varphi_lF_2f)(\bar x)|^p\bigg)^{1/p}\equiv I_3.\cr}
$$
For the power $d>4$ we get
$$\eqal{
I_3&\le c\sum_{k=0}^\infty\bigg(\bigg|\sum_{l=0}^\infty 2^{(\sigma+2)k}\bigg( \intop_{\R^4}d\bar y\bigg|\bigg[\bigg(F_2^{-1}\psi_l{P(\xi)\over\tau+\xi^2}\bigg) (2^l\cdot,2^{2l}\cdot)\cr
&\quad\star(F_2^{-1}\varphi_k)(2^l\cdot,2^{2l}\cdot)\bigg](\bar y)(1+|\bar y|_a^{d/2})\bigg|^2\bigg)^{1/2}\bigg|^p\bigg)^{1/p}\cdot\cr
&\quad\cdot\bigg(\intop_{\R^4}d\bar x|(F_2^{-1}\varphi_lF_2f)(\bar x)|^p\bigg)^{1/p}\equiv I_4.\cr}
$$
By the Parseval identity, we have
$$\eqal{
I_4&=c\sum_{k=0}^\infty\sum_{l=0}^\infty 2^{(\sigma+2)(k-l)}2^{(\sigma+2)l} \bigg\|\bigg(\psi_l{P(\xi)\over\tau+\xi^2}\bigg)(2^l\cdot,2^{2l}\cdot)\cr
&\quad\cdot\varphi_k(2^l\cdot,2^{2l}\cdot)\bigg\|_{W_2^{d,d/2}(\R^4)} \bigg(\intop_{\R^4}d\bar x|(F_2^{-1}\varphi_lF_2f)(\bar x)|^p\bigg)^{1/p}\cr}
$$
In view of Lemma \ref{l3.2} (see below), we have
\begin{equation}
I_4\le c\sum_{k=0}^\infty\sum_{l=0}^\infty 2^{(\sigma+2+d+5-L)|l-k|}2^{\sigma l} \|F_2^{-1}\varphi_lF_2f\|_{L_p(\R^4)}\equiv I_5.
\label{3.20}
\end{equation}
For $L$ sufficiently large, we obtain
$$
I_5\le c\sum_{l=0}^\infty 2^{\sigma l}\|F_2^{-1}\varphi_lF_2f\|_{L_p(\R^4)}.
$$
Hence $(\ref{3.17})_1$ is proved. Similarly we show $(\ref{3.17})_2$. This concludes the proof.
\end{proof}

To prove (\ref{3.20}) we need

\begin{lemma}\label{l3.2}
Let $d$ be an even number such that $d>4$. Then
\begin{equation}\eqal{
I&\equiv\bigg\|\psi_l{P(\xi)\over\tau+\xi^2}(2^l\cdot,2^{2l}\cdot)\varphi_k (2^l\cdot,2^{2l}\cdot)\bigg\|_{W_2^{d,d/2}(\R^4)}\cr
&\le c2^{-2l}2^{(d+5-L)|l-k|},\cr}
\label{3.21}
\end{equation}
where $P(\xi)=\{\delta_{jk}-\xi_j\xi_k/|\xi|^2\}_{j,k=1,2,3}$, $s=\gamma+i\xi_0$, $\tau=\sqrt{s+\xi^2}$ and for $l=0$ the constant $c$ depends on $\gamma$ and $L>0$ may be assumed sufficiently large.
\end{lemma}

\begin{proof}
To estimate $I$, we express it explicitly as
$$\eqal{
I=&\sum_{\sum_{i=1}^3 s'_i+2s_i\le d}\bigg\|(\partial_\xi^{s'_1}\partial_{\xi_0}^{s_1}\psi_l)(2^l\cdot,2^{2l}\cdot) \bigg(\partial_\xi^{s'_2}\partial_{\xi_0}^{s_2}{P(\xi)\over\tau+\xi^2}\bigg) (2^l\cdot,2^{2l}\cdot)\cr
&\quad(\partial_\xi^{s'_3}\partial_{\xi_0}^{s_3}\varphi_k)(2^l\cdot,2^{2l}\cdot) \bigg\|_{L_2(\R^4)}.\cr}
$$
Since $\supp\psi_l(2^l\cdot,2^{2l}\cdot)\subset\{\bar\xi\colon|\bar\xi|_a\le4\}$ for $l=0$, $\supp\psi_l(2^l\cdot,2^{2l}\cdot)\subset$ \\ $\{\bar\xi\colon 1/4\le|\bar\xi|_a\le 4\}$ for $l\not=0$, and
\begin{equation}
|\partial_\xi^{s'_1}\partial_{\xi_0}^{s_1}\psi_l(2^l\cdot,2^{2l}\cdot)|\le c\quad {\rm for}\ \ \bar\xi\in A,
\label{3.22}
\end{equation}
where
$$\eqal{
A= & \bar\xi\colon|\bar\xi|_a\le 4\quad &{\rm for}\ \ l=0, \cr
A= & \bar\xi\colon 1/4\le|\bar\xi|_a\le4 \quad &{\rm for}\ \  l\not=0. }
$$
Then we have
$$
I\le c\sum_{\sum_{i=2}^3s'_i+2s_i\le d}\bigg\|\bigg(\partial_\xi^{s'_2}\partial_{\xi_0}^{s_2}{P(\xi)\over\tau+\xi^2}\bigg) (2^l\cdot,2^{2l}\cdot)\partial_\xi^{s'_3}\partial_{\xi_0}^{s_3} (2^l\cdot,2^{2l}\cdot)\bigg\|_{L_2(A)}.
$$
Next, we obtain
$$
\bigg\|\bigg(\partial_\xi^{s'_2}\partial_{\xi_0}^{s_2}{P(\xi)\over\tau+\xi^2}\bigg) (2^l\cdot,2^{2l}\cdot)\bigg\|_{L_4(A)}\le c(\gamma)2^{-2l}
$$
and finally, we estimate
$$\eqal{
&\|\partial_\xi^{s'_3}\partial_{\xi_0}^{s_3}\varphi_k(2^l\cdot,2^{2l}\cdot)\|_{L_4(A)} =\|\partial_\xi^{s'_3}\partial_{\xi_0}^{s_3}\varphi_k(2^k\cdot2^{l-k}\cdot, 2^{2k}\cdot2^{2(l-k)}\cdot)\|_{L_4(A)}\cr
&\le c2^{(d+2)|l-k|}\|\partial_y^{s'_3}\partial_{y_0}^{s_3}\varphi_k (2^k\cdot,2^{2k}\cdot)\|_{L_4(B)}\cr
&\le c2^{(d+4-L)|l-k|},\cr}
$$
where
$$\eqal{
B & =\{  \bar y\colon 2^{l-k-2}\le|\bar y|_a\le 2^{l-k+2}\} \quad &{\rm for}\ \ l\not=0, \cr
B & =\{ \bar y\colon|\bar y|_a\le 2^{4-k}\} \quad &{\rm for}\ \ l=0,\ \ L>0,}
$$
 $L$ is a chosen number, and we have been used the fact that $\{\varphi_k\}\in A_{aL}(\R^4)$ (see Definition \ref{d2.16}) and we change variables $y_0=2^{2(l-k)}\xi_0$, $y=2^{l-k}\xi$. Hence (\ref{3.20}) holds. This concludes the proof.
\end{proof}

\begin{theorem}\label{t3.3}
Assume that $f\in W_p^{\sigma,\sigma/2}(\R^3\times\R_+)$, $g\in W_p^{\sigma+1,\sigma/2+1/2}(\R^3\times\R_+)$, $v_0\in W_p^{\sigma+2-2/p}(\R^3)$, $p\in(1,\infty)$, $\sigma\not\in\N$.\\
Then there exists a solution to problem (\ref{3.1}) such that $v\in W_p^{2+\sigma,1+\sigma/2}(\R^3\times\R_+)$, $\nabla p\in W_p^{\sigma,\sigma/2}(\R^3\times\R_+)$ and the estimate holds
\begin{equation}\eqal{
&\|v\|_{W_p^{2+\sigma,1+\sigma/2}(\R^3\times\R_+)}+\|\nabla p\|_{W_p^{\sigma,\sigma/2}(\R^3\times\R_+)}\cr
&\le c(\|f\|_{W_p^{\sigma,\sigma/2}(\R^3\times\R_+)}+ \|g\|_{W_p^{\sigma+1,\sigma/2+1/2}(\R^3\times\R_+)}\cr
&\quad+\|v_0\|_{W_p^{\sigma+2-2/p}(\R^3)},\cr}
\label{3.23}
\end{equation}
where $c$ does not depend on $v$ and $p$.
\end{theorem}

\begin{proof}
By Lemma \ref{l3.1} and (\ref{3.10}) we have
$$
\|v''\|_{W_p^{2+\sigma,1+\sigma/2}(\R^4)}\le c\|f''\|_{W_p^{\sigma,\sigma/2}(\R^4)}.
$$
From (\ref{3.10}), it follows
$$
\|v'-\nabla\varphi\|_{W_p^{2+\sigma,1+_\sigma/2}(\R^4)}\le c\|f'-\nabla\varphi_t+\nu\Delta\varphi\|_{W_p^{\sigma,\sigma/2}(\R^4)}.
$$
Simplifying the above inequality yields
\begin{equation}\eqal{
&\|v'\|_{W_p^{2+\sigma,1+\sigma/2}(\R^4)}\le c(\|f'\|_{W_p^{\sigma,\sigma/2}(\R^4)}+ \|\nabla\varphi\|_{W_p^{2+\sigma,1+\sigma/2}(\R^4)}\cr
&\quad+\|\nabla\varphi_t\|_{W_p^{\sigma,\sigma/2}(\R^4)}+ \|\Delta\varphi\|_{W_p^{\sigma,\sigma/2}(\R^4)}).\cr}
\label{3.24}
\end{equation}
Next (\ref{3.9}) implies that
$$
\varphi=\Delta^{-1}g'.
$$
Hence
\begin{equation}\eqal{
&\|\nabla\Delta^{-1}g'\|_{W_p^{2+\sigma,1+\sigma/2}(\R^4)}\le c\|g'\|_{W_p^{1+\sigma,1/2+\sigma/2}(\R^4)},\cr
&\|\nabla\Delta^{-1}g'_t\|_{W_p^{\sigma,\sigma/2}(\R^4)}\le c\|g'\|_{W_p^{1+\sigma,1/2+\sigma/2}(\R^4)},\cr
&\|\Delta\Delta^{-1}g'\|_{W_p^{\sigma,\sigma/2}(\R^4)}\le c\|g'\|_{W_p^{\sigma,\sigma/2}(\R^4)}.\cr}
\label{3.25}
\end{equation}
Using (\ref{3.25}) in (\ref{3.24}) gives
\begin{equation}
\|v'\|_{W_p^{2+\sigma,1+\sigma/2}(\R^4)}\le c\|f'\|_{W_p^{\sigma,\sigma/2}(\R^4)}+c \|g'\|_{W_p^{1+\sigma,1/2+\sigma/2}(\R^4)}.
\label{3.26}
\end{equation}
In view of relation between $v'$, $f'$, $g^i$ and $\tilde v$, $\tilde f$, $\tilde g$, (\ref{3.27}) implies
\begin{equation}
\|\tilde v\|_{W_p^{2+\sigma,1+\sigma/2}(\R^4)}\le c(\|\tilde f\|_{W_p^{\sigma,\sigma/2}(\R^4)}+\|\tilde g\|_{W_p^{1+\sigma,1/2+\sigma/2}(\R^4)}).
\label{3.27}
\end{equation}
Finally, extension (\ref{3.3}) gives
\begin{equation}\eqal{
&\|v\|_{W_p^{2+\sigma,1+\sigma/2}(\R^3\times\R_+)}\le c\|f\|_{W_p^{\sigma,\sigma/2}(\R^3\times\R_+)}\cr
&\quad+c\|g\|_{W_p^{1+\sigma,1/2+\sigma/2}(\R^3\times\R_+)}+ c\|v_0\|_{W_p^{2+\sigma-2/p}(\R^3)}.\cr}
\label{3.28}
\end{equation}
\vskip-12pt
\end{proof}

\begin{lemma}\label{l3.4}
Let $l\in\N$, $\bar y=(y,y_0)$, $y\in\R^n$ and $\tilde y=(y,2^{-l}y_0)$. Then the following equality holds
\begin{equation}\eqal{
&[F_2^{-1}(f_1(2^l\cdot,2^{2l}\cdot))\star F_2^{-1}(f_2(2^l\cdot,2^{2l}\cdot))] (\bar y)\cr
&=2^{-(n+2)l}(F_2^{-1}f_1\star F_2^{-1}f_2)(2^{-l}\tilde y).\cr}
\label{3.29}
\end{equation}
\end{lemma}

\begin{proof}
Using the definition of convolution we have
$$\eqal{
I_1&=[F_2^{-1}(f_1(2^l\cdot,2^{2l}\cdot))\star F_2^{-1}(f_2(2^l\cdot,2^{2l}\cdot))](\bar y)\cr
&=\intop_{\R^{n+1}}F_2^{-1}(f_1(2^l\cdot,2^{2l}\cdot))(\bar y-\bar z)F_2^{-1}(f_2(2^l\cdot,2^{2l}\cdot))(\bar z)d\bar z,\cr}
$$
where $\bar z=(z,z_0)$, $z\in\R^n$, $z_0\in\R$.

From the definition of the Fourier transform it follows
$$\eqal{
I_1&=\intop_{\R^{n+1}}d\bar z\intop_{\R^{n+1}}e^{i(\bar y-\bar z)\cdot\bar\eta}f_1 (2^l\eta,2^{2l}\eta_0)d\bar\eta\cdot\cr
&\quad\cdot\intop_{\R^{n+1}}e^{i\bar z\bar\vartheta}f_2(2^l\vartheta,2^{2l}\vartheta_0)d\bar\vartheta,\cr}
$$
where $\bar\eta=(\eta,\eta_0)$, $\bar\vartheta=(\vartheta,\vartheta_0)$, $d\bar\eta=d\eta d\eta_0$, $d\bar\vartheta=d\vartheta d\vartheta_0$, $\eta,\vartheta\in\R^n$, $\eta_0,\vartheta_0\in\R$.

Introduce new variables
$$
\eta'_0=2^{2l}\eta_0,\quad \eta'=2^l\eta,\quad \vartheta'_0=2^{2l}\vartheta_0,\quad \vartheta'=2^l\vartheta.
$$
Then $d\bar\eta=2^{-(n+2)l}d\bar\eta'$, $d\bar\vartheta=2^{-(n+2)l}d\bar\vartheta'$, where $\bar\eta'=(\eta',\eta'_0)$, $\bar\vartheta'=(\vartheta',\vartheta'_0)$. Hence $I_1$ equals
$$\eqal{
I_1&=\intop_{\R^{n+1}}d\bar z\intop_{\R^{n+1}}e^{i(y_0-z_0)2^{-2l}\eta'_0+i(y-z)2^{-l}\cdot\eta'}f_1 (\eta',\eta'_0)d\bar\eta'2^{-(n+2)l}\cdot\cr
&\quad\cdot\intop_{\R^{n+1}}e^{iz_02^{-2l}\vartheta'_0+iz2^{-l}\cdot\vartheta'} f_2(\vartheta',\vartheta'_0)d\bar\vartheta'e^{-(n+2)l}\cr
&=2^{-2(n+2)l}\intop_{\R^{n+1}}d\bar z\intop_{\R^{n+1}} e^{i(2^{-2l}y_0-2^{-2l}z_0)\eta'_0+i(2^{-l}y-2^{-l}z)\cdot\eta'}\cdot\cr
&\quad\cdot f_1(\eta',\eta'_0)d\bar\eta'\intop_{\R^{n+1}} e^{i2^{-2l}z_0\vartheta'_0+i2^{-l}z\cdot\vartheta'}f_2(\vartheta',\vartheta'_0) d\bar\vartheta'\cr}
$$
Introducing new variables
$$
z'_0=2^{-2l}z_0,\quad z'=2^{-l}z,\quad \bar z'=(z',z'_0)
$$
we have $d\bar z=2^{(2+n)l}d\bar z'$. Then
$$\eqal{
I_1&=2^{-(n+2)l}\intop_{\R^{n+1}}d\bar z'\intop_{\R^{n+1}} e^{i(2^{-2l}y_0-z'_0)\eta'_0+i(2^{-l}y-z')\cdot\eta'}\cdot\cr
&\quad\cdot f_1(\eta',\eta'_0)d\bar\eta'\intop_{\R^{n+1}} e^{iz'_0\vartheta'_0+iz'\cdot\vartheta'}f_2(\vartheta',\vartheta'_0) d\bar\vartheta'.\cr}
$$
Finally, we introduce
$$
\tilde y=(y,2^{-l}y_0)
$$
Then $I_1$ takes the form
$$
I_1=2^{-(n+2)l}(F_2^{-1}f_1\star F_2^{-1}f_2)(2^{-l}\tilde y).
$$
This ends the proof.
\end{proof}

\section{The Stokes system in the half-space}\label{se4}

In this Section we prove the existence of solutions to the Stokes system with  slip boundary conditions in the half space. Let $\R_+^3=\{x\in\R^3\colon x_3>0\}$. Then the Stokes system has the form
\begin{equation}\eqal{
&v_t-\nu\Delta v+\nabla p=f\quad &{\rm in}\ \ \R_+^3\times\R_+,\cr
&\divv v=g\quad &{\rm in}\ \ \R_+^3\times\R_+,\cr
&v_{3,x_1}+v_{1,x_3}=b_1\ \ {\rm for}\ \ x_3=0\quad &{\rm in}\ \ \R^2\times\R_+,\cr
&v_{3,x_2}+v_{2,x_3}=b_2\ \ {\rm for}\ \ x_3=0\quad &{\rm in}\ \ \R^2\times\R_+,\cr
&v_3=b_3\ \ {\rm for}\ \ x_3=0\quad &{\rm in}\ \ \R^2\times\R_+,\cr
&v|_{t=0}=v_0\quad &{\rm in}\ \ \R_+^3.\cr}
\label{4.1}
\end{equation}
We are looking for solutions to (\ref{4.1}) under the following assumptions
\begin{equation}\eqal{
&f\in W_p^{\sigma,\sigma/2}(\R_+^3\times\R_+),\cr
&g\in W_p^{\sigma+1,\sigma/2+1/2}(\R_+^3\times\R_+),\cr
&b_\alpha\in W_p^{\sigma+1-1/p,\sigma/2+1/2-1/2p}(\R^2\times\R_+),\quad \alpha=1,2,\cr
&b_3\in W^{\sigma+2-1/p,\sigma/2+1-1/2p}(\R^2\times\R_+),\cr
&v_0\in W_p^{\sigma+2-2/p}(\R_+^3),\cr}
\label{4.2}
\end{equation}
where $p\in(1,\infty)$, $\sigma\not\in\N$.

\noindent
Since $v_0\in W_p^{\sigma+2-2/p}(\R_+^3)$ there exists the time extension\break $\tilde v_0\in W_p^{\sigma+2,\sigma/2+1}(\R_+^3\times\R_+)$ such that
\begin{equation}
\tilde v_0|_{t=0}=v_0
\label{4.3}
\end{equation}
and
\begin{equation}
\|\tilde v_0\|_{W_p^{2+\sigma,1+\sigma/2}(\R_+^3\times\R_+)}\le c\|v_0\|_{W_p^{\sigma+2-2/p}(\R_+^3)},
\label{4.4}
\end{equation}
where $c$ does not depend on $v_0$.

Having extension $\tilde v_0$ we can introduce the function
\begin{equation}
\tilde v=v-\tilde v_0
\label{4.5}
\end{equation}
such that $(\tilde v,p)$ is a solution to the Stokes system with vanishing initial data
\begin{equation}\eqal{
&\tilde v_t-\nu\Delta\tilde v+\nabla p=f-\tilde v_{0,t}+\nu\Delta\tilde v_0\equiv\tilde f,\cr
&\divv\tilde v=g-\divv\tilde v_0\equiv\tilde g,\cr
&\tilde v_{3,x_1}+\tilde v_{1,x_3}=b_1-\tilde v_{30,x_1}-\tilde v_{10,x_3}\equiv\tilde b_1,\cr
&\tilde v_{3,x_2}+\tilde v_{2,x_3}=b_2-\tilde v_{30,x_2}-\tilde v_{20,x_3}\equiv\tilde b_2,\cr
&\tilde v_3=b_3-\tilde v_{30}\equiv\tilde b_3,\cr
&\tilde v|_{t=0}=0.\cr}
\label{4.6}
\end{equation}
In view of Lemmas \ref{l2.14} and \ref{l2.15} we can extend functions $\tilde f,\tilde g,\tilde b=(\tilde b_1,\tilde b_2,\tilde b_3)$ by zero for $t<0$ in the same classes assuming the compatibility conditions
\begin{equation}\eqal{
&\partial_t^i\tilde f|_{t=0}=0\quad &{\rm for}\ \ i\le[\sigma/2],\cr
&\partial_t^i\tilde g|_{t=0}=0\quad &{\rm for}\ \ i\le[\sigma/2+1/2],\cr
&\partial_t^i\tilde b_\alpha|_{t=0}=0\quad &{\rm for}\ \ i\le\bigg[{\sigma\over2}+{1\over2}-{1\over p}\bigg],\ \ \alpha=1,2,\cr
&\partial_t^i\tilde b_3|_{t=0}=0\quad &{\rm for}\ \ i\le\bigg[{\sigma\over2}+1-{1\over p}\bigg].\cr}
\label{4.7}
\end{equation}
Denote  extended functions by $f'$, $g'$, $b'$, respectively. Then problem (\ref{4.6}) for extended functions takes the form
\begin{equation}\eqal{
&v'_t-\nu\Delta v'+\nabla p'=f'\quad &{\rm in}\ \ \R_+^3\times\R,\cr
&\divv v'=g'\quad &{\rm in}\ \ \R_+^3\times\R,\cr
&v'_{3,x_\alpha}+v'_{\alpha,x_3}=b'_\alpha,\ \ \alpha=1,2,\ \ {\rm for}\ \ x_3=0\quad &{\rm in}\ \ \R^2\times\R,\cr
&v'_3=b'_3\ \ {\rm for}\ \ x_3=0\quad &{\rm in}\ \ \R^2\times\R.\cr}
\label{4.8}
\end{equation}
Consider the Neumann problem
\begin{equation}\eqal{
&\Delta\varphi'=g'\quad &{\rm in}\ \ \R_+^3,\cr
&{\partial\over\partial x_3}\varphi'|_{x_3}=0\quad &{\rm in}\ \ \R^2\cr}
\label{4.9}
\end{equation}
Introducing the function
\begin{equation}
v''=v'-\nabla\varphi'
\label{4.10}
\end{equation}
we see that $(v'',p')$ is a solution to the problem
\begin{equation}\eqal{
&v''_t-\nu\Delta v''+\nabla p'=f'-\nabla\varphi'_t+\nu\Delta\nabla\varphi'\equiv f'',\cr
&\divv v''=0\cr
&v''_{3,x_\alpha}+v''_{\alpha,x_3}=b'_\alpha-2\varphi'_{,x_\alpha x_3}\equiv b''_\alpha,\ \ \alpha=1,2,\ \ x_3=0,\cr
&v''_3=b'_3\equiv b''_3,\ \ x_3=0.\cr}
\label{4.11}
\end{equation}
Consider the Stokes system
\begin{equation}\eqal{
&v''_t-\nu\Delta v''+\nabla p'=f''\quad &{\rm in}\ \ \R_+^3\times\R,\cr
&\divv v''=0\quad &{\rm in}\ \ \R_+^3\times\R.\cr}
\label{4.12}
\end{equation}
To apply the results from Section \ref{se3} we have to extend problem (\ref{4.12}) to the problem in the whole space.

Since $f''\in W_p^{\sigma,\sigma/2}(\R_+^3\times\R)$, $\sigma\not\in\N$, we have to extend $f''$ by zero for $x_3<0$. For this purpose we examine the norm
$$\eqal{
&\|f''\|_{W_p^{\sigma,\sigma/2}(\R_+^3\times\R)}=\|f''\|_{L_p(\R_+^3\times\R)}\cr
&\quad+\bigg(\intop_\R dt\intop_{\R^2}dx'\intop_{\R^2}dx''\intop_0^\infty dx_3 {|D_x^{[\sigma]}f''(x',x_3,t)-D_{x''}^{[\sigma]}f(x'',x_3,t)|^p\over |x'-x''|^{2+p(\sigma-[\sigma])}}\bigg)^{1/p}\cr
&\quad+\bigg(\intop_\R dt\intop_{\R^2}dx'\intop_0^\infty dx'_3\intop_0^\infty dx''_3{|\partial_{x'_3}^{[\sigma]}f''(x',x'_3,t)- \partial_{x''_3}^{[\sigma]}f''(x',x''_3,t)|^p\over |x'_3-x''_3|^{1+p(\sigma-[\sigma])}}\bigg)^{1/p}\cr
&\quad+\bigg(\intop_{\R_+^3}dx\intop_\R dt'\intop_\R dt'' {|\partial_{t'}^{[\sigma/2]}f''(x,t')-\partial_{t''}^{[\sigma/2]}f''(x,t'')|^p\over |t'-t''|^{1+p(\sigma/2-[\sigma/2])}}\bigg)^{1/p}\cr
&\equiv I_1+I_2+I_3+I_4.\cr}
$$
Let $\bar f''=\begin{cases}f''&{\rm for\ }x_3>0,\cr 0&{\rm for\ }x_3<0.\cr
\end{cases}$

It is  clear that terms $I_1$, $I_2$, $I_4$ holds also for $\bar f''$. To express $I_3$ for $\bar f''$ we need that
\begin{equation}
\bigg(\intop_0^\infty dx_3\intop_{\R^3}dx'dt {|\partial_{x_3}^{[\sigma]}f''|^p\over x_3^{p(\sigma-[\sigma])}}\bigg)^{1/p}
\label{4.13}
\end{equation}
is finite. Similarly as in Lemma \ref{l2.15} the term (\ref{4.13}) is bounded by
$$
\|\bar f''\|_{W_p^{\sigma,\sigma/2}(\R^3\times\R)}.
$$
Then we can transform (\ref{4.12}) to the form
\begin{equation}\eqal{
&\bar v''_t-\nu\Delta\bar v''+\nabla\bar p''=\bar f''\quad {\rm in}\ \ \R^4\cr
&\divv\bar v''=0\cr}
\label{4.14}
\end{equation}
From Theorem \ref{t3.3} we have the existence of solutions to problem (\ref{4.14}) and the estimate
\begin{equation}\eqal{
&\|\bar v''\|_{W_p^{\sigma+2,\sigma/2+1}(\R^4)}+\|\nabla\bar p''\|_{W_p^{\sigma,\sigma/2}(\R^4)}\cr
&\le c\|f''\|_{W_p^{\sigma,\sigma/2}(\R^3\times\R)}.\cr}
\label{4.15}
\end{equation}
Introducing new functions
\begin{equation}
u=v''-\bar v'',\quad q=p'-\bar p''
\label{4.16}
\end{equation}
we see that $(u,q)$ is a solution to the problem
\begin{equation}\eqal{
&u_t-\nu\Delta u+\nabla q=0\cr
&\divv u=0\cr
&u_{3,x_\alpha}+u_{\alpha,x_3}=b''_\alpha-\bar v''_{3,x_\alpha}-\bar v''_{\alpha,x_3}\equiv d_\alpha,\ \ \alpha=1,2,\cr
&u_3=b''_3-\bar v''_3\equiv d_3.\cr}
\label{4.17}
\end{equation}
To solve problem (\ref{4.17}) we use the Fourier-Laplace transform
\begin{equation}
\tilde u(\xi,x_3,s)=(F_2u)(\xi,x_3,s)=\intop_0^\infty e^{-st}\intop_{\R^2}e^{-ix'\cdot\xi}u(x,t)dx'dt,
\label{4.18}
\end{equation}
where $s=\gamma+i\xi_0$, $\gamma\in\R_+$, $\xi_0\in\R$, $\xi=(\xi_1,\xi_2)$, $x'=(x_1,x_2)$, $x=(x_1,x_2,x_3)$, $x'\cdot\xi=x_1\xi_1+x_2\xi_2$.

Applying the Fourier-Laplace transform (\ref{4.18}) to system $(\ref{4.17})_{1,2}$ we obtain
\begin{equation}\eqal{
&\nu\bigg(-{d^2\over dx_3^2}+\tau^2\bigg)\tilde u_\alpha+i\xi_\alpha\tilde q=0,\ \ \alpha=1,2,\cr
&\nu\bigg(-{d^2\over dx_3^2}+\tau^2\bigg)\tilde u_3+{d\tilde q\over dx_3}=0,\cr
&i\xi_1\tilde u_1+i\xi_2\tilde u_2+{d\tilde u_3\over dx_3}=0,\cr
&\tilde u\to 0,\ \ \tilde q\to 0\quad {\rm as}\quad x_3\to+\infty,\cr}
\label{4.19}
\end{equation}
where $\tau^2={s\over\nu}+\xi^2$, $\xi^2=\xi_1^2+\xi_2^2$, $\arg\tau\in(-\pi/4,\pi/4)$ and conditions $(\ref{4.19})_4$ are called the Shapiro-Lopatinskii conditions.

Solutions of (\ref{4.19}) have the form (see \cite{S2})
\begin{equation}\eqal{
&\tilde u=\Phi(\xi,s)e^{-\tau x_3}+\varphi(\xi,s)(i\xi_1,i\xi_2,-|\xi|) e^{-|\xi|x_3},\cr
&\tilde q=-s\varphi(\xi,s)e^{-|\xi|x_3},\cr}
\label{4.20}
\end{equation}
where $\Phi(\xi,s)=(\Phi_1,\Phi_2,(i\xi_1\Phi_1+i\xi_2\Phi_2)/\tau)$ and $\Phi_1$, $\Phi_2$, $\varphi$ are arbitrary parameters which must be calculated from boundary conditions $(\ref{4.17})_{3,4}$.

Applying the Fourier-Laplace transform to $(\ref{4.17})_{3,4}$ yields
\begin{equation}\eqal{
&i\xi_1\tilde u_3+\tilde u_{1,x_3}=\tilde d_1,\cr
&i\xi_2\tilde u_3+\tilde u_{2,x_3}=\tilde d_2,\cr
&\tilde u_3=\tilde d_3.\cr}
\label{4.21}
\end{equation}
Using (\ref{4.20}) in (\ref{4.21}) yields
\begin{equation}\eqal{
&(\tau^2+\xi_1^2)\Phi_1+\xi_1\xi_2\Phi_2+2i\xi_1|\xi|\tau\varphi=-\tau\tilde d_1,\cr
&\xi_1\xi_2\Phi_1+(\tau^2+\xi_2^2)\Phi+2i\xi_2|\xi|\tau\varphi=-\tau\tilde d_2,\cr
&i\xi_1\Phi_1+i\xi_2\Phi_2-|\xi|\tau\varphi=\tau\tilde d_3.\cr}
\label{4.22}
\end{equation}
Eliminating $\varphi$ from (\ref{4.22}) gives
\begin{equation}\eqal{
&(\tau^2-\xi_1^2)\Phi_1-\xi_1\xi_2\Phi_2=2i\xi_1\tau\tilde d_3-\tau\tilde d_1\equiv\tilde B_1,\cr
&-\xi_1\xi_2\Phi_1+(\tau^2-\xi_2^2)\Phi_2=2i\xi_2\tau\tilde d_3-\tau\tilde d_2\equiv\tilde B_2,\cr
&|\xi|\tau\varphi=i(\xi_1\Phi_1+\xi_2\Phi_2)-\tau\tilde d_3.\cr}
\label{4.23}
\end{equation}
From $(\ref{4.23})_{1,2}$ we calculate
\begin{equation}\eqal{
&\Phi_1={(\tau^2-\xi_2^2)\tilde B_1+\xi_1\xi_2\tilde B_2\over D},\cr
&\Phi_2={\xi_1\xi_2\tilde B_1+(\tau^2-\xi_1^2)\tilde B_2\over D},\cr}
\label{4.24}
\end{equation}
where $D=\tau^2(\tau^2-\xi^2)={s\over\nu}\tau^2$ because $\tau^2={s\over\nu}+\xi^2$.

Using (\ref{4.24}) in $(\ref{4.23})_3$ yields
\begin{equation}
\varphi={i(\xi_1\tilde B_1+\xi_2\tilde B_2)\over|\xi|\tau s/\nu}-{1\over|\xi|} \tilde d_3.
\label{4.25}
\end{equation}
Using the expressions of $\Phi_1$, $\Phi_2$ and $\varphi$ in (\ref{4.20}) we obtain
$$\eqal{
\tilde u_1&={(\tau^2-\xi_2^2)\tilde B_1+\xi_1\xi_2\tilde B_2\over\tau^2s/\nu} e^{-\tau x_3}-{\xi_1^2\tilde B_1+\xi_1\xi_2\tilde B_2\over|\xi|\tau s/\nu} e^{-|\xi|x_3}-{i\xi_1\over|\xi|}\tilde d_3e^{-|\xi|x_3}\cr
&=\bigg[{\tau^2-\xi_2^2\over\tau^2s/\nu}e^{-\tau x_3}-{\xi_1^2\over|\xi|\tau s/\nu}e^{-|\xi|x_3}\bigg]\tilde B_1\cr
&\quad+\bigg[{\xi_1\xi_2\over\tau^2 s/\nu}e^{-\tau x_3}-{\xi_1\xi_2\over|\xi|\tau s/\nu}e^{-|\xi|x_3}\bigg]\tilde B_2-{i\xi_1\over|\xi|}\tilde d_3e^{-|\xi|x_3}.\cr}
$$
Consider the coefficient next to $\tilde B_1$. We write it in the form
$$\eqal{
&{\tau^2-\xi^2\over\tau^2s/\nu}e^{-\tau x_3}+\xi_1^2\bigg({1\over\tau^2s/\nu}e^{-\tau x_3}-{1\over|\xi|\tau s/\nu} e^{-|\xi|x_3}\bigg)\cr
&={1\over\tau^2}e^{-\tau x_3}+\xi_1^2\bigg({1\over\tau^2s/\nu}-{1\over|\xi|\tau s/\nu}\bigg)e^{-\tau x_3}+{\xi_1^2\over|\xi|\tau s/\nu}(e^{-\tau x_3}-e^{-|\xi|x_3})\cr
&={1\over\tau^2}e^{-\tau x_3}+{\xi_1^2\over\tau}{|\xi|-\tau\over\tau|\xi|s/\nu} e^{-\tau x_3}+{\xi_1^2(\tau-|\xi|)\over|\xi|\tau x/\nu} {e^{-\tau x_3}-e^{-|\xi|x_3}\over\tau-|\xi|}\cr
&={1\over\tau^2}e^{-\tau x_3}-{\xi_1^2(\tau^2-\xi^2)\over\tau^2|\xi|s/\nu(\tau+|\xi|)}e^{-\tau x_3}+ {\xi_1^2(\tau^2-\xi^2)\over|\xi|\tau s/\nu(\tau+|\xi|)}{e^{-\tau x_3}-e^{-|\xi|x_3}\over\tau-|\xi|}\cr
&={1\over\tau^2}e^{-\tau x_3}-{\xi_1^2\over\tau^2|\xi|(\tau+|\xi|)}e^{-\tau x_3} +{\xi_1^2\over\tau|\xi|(\tau+|\xi|)}{e^{-\tau x_3}-e^{-|\xi|x_3}\over\tau-|\xi|}.\cr}
$$
Consider the coefficient next to $\tilde B_2$. We express it in the form
$$\eqal{
&\xi_1\xi_2\bigg[\bigg({1\over\tau^2s/\nu}-{1\over|\xi|\tau s/\nu}\bigg) e^{-\tau x_3}+{\tau-|\xi|\over|\xi|\tau s/\nu}{e^{-\tau x_3}-e^{-|\xi|x_3}\over\tau-|\xi|}\bigg]\cr
&=\xi_1\xi_2\bigg[{1\over\tau s/\nu}{|\xi|-\tau\over\tau|\xi|}e^{-\tau x_3}+ {1\over|\xi|\tau(\tau+|\xi|)}{e^{-\tau x_3}-e^{-|\xi|x_3}\over\tau-|\xi|}\bigg]\cr
&=\xi_1\xi_2\bigg[-{1\over\tau^2|\xi|(\tau+|\xi|)}e^{-\tau x_3}+{1\over|\xi|\tau(\tau+|\xi|)}{e^{-\tau x_3}-e^{-|\xi|x_3}\over\tau-|\xi|}\bigg].\cr}
$$
Introduce the notation
$$
e_0(x_3)=e^{-\tau x_3},\quad e_1(x_3)={e^{-\tau x_3}-e^{-|\xi|x_3}\over\tau-|\xi|},\quad e_2(x_3)=e^{-|\xi|x_3}.
$$
Then $\tilde u_1$ takes the form
\begin{equation}\eqal{
\tilde u_1&=\bigg[{1\over\tau^2}e_0-{\xi_1^2\over\tau^2|\xi|(\tau+|\xi|)}e_0+ {\xi_1^2\over\tau|\xi|(\tau+|\xi|)}e_1\bigg]\tilde B_1\cr
&\quad+\xi_1\xi_2\bigg(-{1\over\tau^2|\xi|(\tau+|\xi|)}e_0+ {1\over|\xi|\tau(\tau+|\xi|)}e_1\bigg)\tilde B_2-{i\xi_1\over|\xi|}\tilde d_3e_2\cr}
\label{4.26}
\end{equation}
Similarly,
\begin{equation}\eqal{
\tilde u_2&=\bigg(-{\xi_1\xi_2\over\tau^2|\xi|(\tau+|\xi|)}e_0+ {\xi_1\xi_2\over\tau|\xi|(\tau+|\xi|)}e_1\bigg)\tilde B_1\cr
&\quad+\bigg({1\over\tau^2}e_0-{\xi_2^2\over\tau^2|\xi|(\tau+|\xi|)}e_0+ {\xi_2^2\over\tau|\xi|(\tau+|\xi|)}e_1\bigg)\tilde B_2-{i\xi_2\over|\xi|}\tilde d_3e_2.\cr}
\label{4.27}
\end{equation}
Next, we consider
\begin{equation}\eqal{
\tilde u_3&=(i\xi_1\Phi_1+i\xi_2\Phi_2)/\tau e^{-\tau x_3}-\varphi|\xi|e^{-|\xi|x_3}\cr
&=\bigg[{i\xi_1\over\tau}{(\tau^2-\xi_2^2)\tilde B_1+\xi_1\xi_2\tilde B_2\over \tau^2s/\nu}+{i\xi_2\over\tau}{\xi_1\xi_2\tilde B_1+(\tau^2-\xi_1^2)\tilde B_2 \over\tau^2s/\nu}\bigg]e^{-\tau x_3}\cr
&\quad-|\xi|{i(\xi_1\tilde B_1+\xi_2\tilde B_2)\over|\xi|\tau s/\nu} e^{-|\xi|x_3}+\tilde d_3e^{-|\xi|x_3}\cr
&=\bigg[\bigg({i\xi_1\over\tau^3s/\nu}(\tau^2-\xi_2^2)+{i\xi_2\over\tau} {\xi_1\xi_2\over\tau^2s/\nu}\bigg)e^{-\tau x_3}-{i\xi_1\over\tau s/\nu} e^{-|\xi|x_3}\bigg]\tilde B_1\cr
&\quad+\bigg[\bigg({i\xi_1\over\tau}{\xi_1\xi_2\over\tau^2s/\nu}+{i\xi_2\over\tau} {(\tau^2-\xi_1^2)\over\tau^2s/\nu}\bigg)e^{-\tau x_3}-{i\xi_2\over\tau s/\nu} e^{-|\xi|x_3}\bigg]\tilde B_2+\tilde d_3e^{-|\xi|x_3}\cr
&=\bigg[{i\xi_1\over\tau s/\nu}e^{-\tau x_3}-{i\xi_1\over\tau s/\nu}e^{-|\xi|x_3}\bigg]\tilde B_1+\bigg[{i\xi_2\over\tau s/\nu}e^{-\tau x_3}-{i\xi_2\over\tau s/\nu}e^{-|\xi|x_3}\bigg]\tilde B_2\cr
&\quad+\tilde d_3e^{-|\xi|x_3}\cr
&={i\xi_1\over\tau s/\nu}(\tau-|\xi|){e^{-\tau x_3}-e^{-|\xi|x_3}\over\tau-|\xi|} \tilde B_1+{i\xi_2\over\tau s/\nu}(\tau-|\xi|){e^{-\tau x_3}-e^{-|\xi|x_3}\over\tau-|\xi|}\tilde B_2\cr
&\quad+\tilde d_3e^{-|\xi|x_3}={i\xi_1\over\tau(\tau+|\xi|)}e_1\tilde B_1+ {i\xi_2\over\tau(\tau+|\xi|)}e_1\tilde B_2+\tilde d_3e_2.\cr}
\label{4.28}
\end{equation}
Finally, we calculate
\begin{equation}
\tilde q=-{i\nu(\xi_1\tilde B_1+\xi_2\tilde B_2)\over|\xi|\tau}e^{-|\xi|x_3}+ {s\over|\xi|}\tilde d_3e^{-|\xi|x_3}.
\label{4.29}
\end{equation}
Using the form of $\tilde B_1$, $\tilde B_2$ and expressing $e_2$ in terms of $e_0$ and $e_1$ we have
\begin{equation}\eqal{
\tilde u_1&=\bigg[{1\over\tau^2}e_0-{\xi_1^2\over\tau^2|\xi|(\tau+|\xi|)}e_0+ {\xi_1^2\over\tau|\xi|(\tau+|\xi|)}e_1\bigg] (2i\xi_1\tau\tilde d_3-\tau\tilde d_1)\cr
&\quad+\xi_1\xi_2\bigg(-{1\over\tau^2|\xi|(\tau+|\xi|)}e_0+ {1\over|\xi|\tau(\tau+|\xi|)}e_1\bigg)(2i\xi_2\tau\tilde d_3-\tau\tilde d_2)\cr
&\quad-{i\xi_1\over|\xi|}\tilde d_3(e_0-(\tau-|\xi|)e_1)\cr
&=-{\xi_1^2\over|\xi|(\tau+|\xi|)}e_1\tilde d_1-{\xi_1\xi_2\over|\xi|(\tau+|\xi|)}e_1\tilde d_2+i\xi_1{|\xi|^2+\tau^2\over|\xi|(\tau+|\xi|)}e_1\tilde d_3\cr
&\quad+\bigg(-{1\over\tau}+{\xi_1^2\over\tau|\xi|(\tau+|\xi|)}\bigg)e_0\tilde d_1+{\xi_1\xi_2\over\tau|\xi|(\tau+|\xi|)}e_0\tilde d_2+{i\xi_1(|\xi|-\tau)\over|\xi|(\tau+|\xi|)}e_0\tilde d_3\cr
&\equiv\sum_{r=1}^3(g_{1r}e_1\tilde d_r+h_{1r}e_0\tilde d_r).\cr}
\label{4.30}
\end{equation}
Next, we calculate
\begin{equation}\eqal{
\tilde u_2&=\bigg(-{\xi_1\xi_2\over\tau^2|\xi|(\tau+|\xi|)}e_0+ {\xi_1\xi_2\over\tau|\xi|(\tau+|\xi|)}e_1\bigg)(2i\xi_1\tau\tilde d_3-\tau\tilde d_1)\cr
&\quad+\bigg({1\over\tau^2}e_0-{\xi_2^2\over\tau^2|\xi|(\tau+|\xi|)}e_0+ {\xi_2^2\over\tau|\xi|(\tau+|\xi|)}e_1\bigg)(2i\xi_2\tau\tilde d_3-\tau\tilde d_2)\cr
&\quad-{i\xi_2\over|\xi|}\tilde d_3(e_0-(\tau-|\xi|)e_1)\cr
&=-{\xi_1\xi_2\over|\xi|(\tau+|\xi|)}e_1\tilde d_1-{\xi_2^2\over|\xi|(\tau+|\xi|)}e_1\tilde d_2+i\xi_2{|\xi|^2+\tau^2\over|\xi|(\tau+|\xi|)}e_1\tilde d_3\cr
&\quad-{\xi_1\xi_2\over\tau|\xi|(\tau+|\xi|)}e_0\tilde d_1+\bigg(-{1\over\tau}+{\xi_2^2\over\tau|\xi|(\tau+|\xi|)}\bigg)e_0\tilde d_2\cr
&\quad+{i\xi_2(|\xi|-\tau)\over|\xi|(\tau+|\xi|)}e_0\tilde d_3\cr
&\equiv\sum_{r=1}^3(g_{2r}e_1\tilde d_r+h_{2r}e_0\tilde d_r).\cr}
\label{4.31}
\end{equation}
From (\ref{4.28}) we have
\begin{equation}\eqal{
\tilde u_3&={i\xi_1\over\tau(\tau+|\xi|)}e_1(2i\xi_1\tau\tilde d_3-\tau\tilde d_1)\cr
&\quad+{i\xi_2\over\tau(\tau+|\xi|)}e_1(2i\xi_2\tau\tilde d_3-\tau\tilde d_2)+\tilde d_3(e_0-(\tau-|\xi|)e_1)\cr
&=-{i\xi_1\over\tau+|\xi|}e_1\tilde d_1-{i\xi_2\over\tau+|\xi|}e_1\tilde d_2- {\tau^2+|\xi|^2\over\tau+|\xi|}e_1\tilde d_3+e_0\tilde d_3\cr
&\equiv\sum_{r=1}^3(g_{3r}e_1\tilde d_r+g_{3r}e_0\tilde d_r)\cr}
\label{4.32}
\end{equation}
Finally, (\ref{4.29}) yields
\begin{equation}\eqal{
\tilde q&=\bigg[-{i\nu\xi_1\over|\xi|\tau}(2i\xi_1\tau\tilde d_3-\tau\tilde d_1)-{i\nu\xi_2\over|\xi|\tau}(2i\xi_2\tau\tilde d_3-\tau\tilde d_2)\cr
&\quad+{s\over|\xi|}\tilde d_3\bigg](e_0-(\tau-|\xi|)e_1)\cr
&=-{i\nu\xi_1(\tau-|\xi|)\over|\xi|}e_1\tilde d_1-{i\nu\xi_2(\tau-|\xi|)\over|\xi|}e_1\tilde d_2\cr
&\quad+\nu{(\tau^2+\xi^2)(\tau-|\xi|)\over|\xi|}e_1\tilde d_3+{i\nu\xi_1\over|\xi|}e_0\tilde d_1+{i\nu\xi_2\over|\xi|}e_0\tilde d_2\cr
&\quad-{\nu(\tau^2+|\xi|^2)\over|\xi|}e_0\tilde d_3\cr
&\equiv\sum_{r=1}^3(g_{4r}e_1\tilde d_r+h_{4r}e_0\tilde d_r)\cr}
\label{4.33}
\end{equation}

\begin{lemma}\label{l4.1}
Let $\tau=\sqrt{\gamma+i\xi_0+\xi^2}$, $\xi^2=\xi_1^2+\xi_2^2$. Then the following estimates hold:
\begin{itemize}
\item[(1)] $|\partial_{\xi_0}g_{mr}|\le{c\over|\tau|^2}$, $|\partial_{\xi_0}^2g_{mr}|\le{c\over|\tau|^4}$, $m=1,2,3$, $r=1,2$.
\item[(2)] $|\partial_{xi'}^1g_{mr}|\le c/|\tau|$, $|\partial_{\xi'}^2g_{mr}|\le c/|\xi|\,|\tau|$, \ $|\partial_{\xi'}^3g_{mr}|\le{c\over|\xi|^2|\tau|}+{c\over|\xi|\,|\tau|^2}$,\\ $|\partial_{\xi'}^4g_{mr}|\le{c\over|\xi|^3|\tau|}+{c\over|\xi|^2|\tau|^2}+ {c\over|\xi|\,|\tau|^3}$, $m=1,2,3$, $r=1,2$.
\item[(3)] $|\partial_{\xi_0}g_{mr}|\le c/|\tau|$, $|\partial_{\xi_0}^2g_{mr}|\le c/|\tau|^3$, $r=3$.
\item[(4)] $|\partial_{\xi'}g_{m3}|\le c+c|\tau|/|\xi|$,\\
$|\partial_{\xi'}^2g_{m3}|\le c/|\tau|+c/|\xi|+c|\tau|/|\xi|^2$,\\
$|\partial_{\xi'}^3g_{m3}|\le c/|\xi|(|\tau|+|\xi|)+c/(|\tau|+|\xi|)^2$\\
$|\partial_{\xi'}^4g_{m3}|\le{c\over|\xi|^2(|\tau|+|\xi|)}+ {c\over|\xi|(|\tau|+|\xi|)^2}$.
\item[(5)] $|h_{mr,\xi_0}|\le c/|\tau|^3$, $|h_{mr,\xi_0\xi_0}|\le c/|\tau|^5$, $m=1,2,3$, $r=1,2$.
\item[(6)] $|\partial_{\xi'}h_{mr}|\le c/|\tau|^2$, $|\partial_{\xi'}^2h_{mr}|\le c/|\tau|^3+{c\over|\tau|^2|\xi|}$,\\
$|\partial_{\xi'}^3h_{mr}|\le c/|\tau|^4+{c\over|\tau|^3|\xi|}+{c\over|\tau|^2|\xi|^2}$,\\
$|\partial_{\xi'}^4h_{mr}|\le c/|\tau|^5+{c\over|\tau|^4|\xi|}+{c\over|\tau|^3|\xi|^2}+ {c\over|\tau|^2|\xi|^3}$, $m=1,2,3$, $r=1,2$.
\item[(7)] $|h_{m3,\xi_0}|\le c/|\tau|^2$, $|h_{m3,\xi_0\xi_0}|\le c/|\tau|^4$, $m=1,2$, $|\partial_{\xi'}h_{m3}|\le{c\over|\tau|}+{c\over|\xi|}$,\\
$|\partial_{\xi'}^2h_{m3}|\le{c\over|\tau|^2}+{c\over|\xi|^2}$,\\
$|\partial_{\xi'}^3h_{m3}|\le{c\over|\xi|^2|\tau|}+{c\over|\xi|\,|\tau|^2}$,\\
$|\partial_{\xi'}^4h_{m3}|\le{c\over|\xi|^3|\tau|}+{c\over|\xi|^2|\tau|^2}+ {c\over|\xi|\,|\tau|^3}$, $m=1,2$.
\end{itemize}
\end{lemma}

\begin{proof}
Consider $g_{11}$. Then
$$\eqal{
|g_{11,\xi_0}|&=\bigg|{\xi_1^2\over|\xi|(\tau+|\xi|^2)}{1\over\tau}\bigg|\le c{|\xi|\over(|\tau|+|\xi|)^2}{1\over|\tau|}\cr
&\le{c\over(|\tau|+|\xi|)|\tau|}\le{c\over|\tau|^2},\cr}
$$
$$\eqal{
|g_{11,\xi_0\xi_0}|&=\bigg|{\xi_1^2\over|\xi|(\tau+|\xi|)^3}{1\over\tau^2}+ {\xi_1^2\over|\xi|(\tau+|\xi|)^2}{1\over\tau^3}\bigg|\cr
&\le c{|\xi|\over(|\tau|+|\xi|)^3}{1\over|\tau|^2}+c{|\xi|\over(|\tau|+|\xi|)^2} {1\over|\tau|^3}\cr
&\le c{1\over(|\tau|+|\xi|)^2}{1\over|\tau|^2}+c{1\over(|\tau|+|\xi|)|\tau|^3}\le {c\over|\tau|^4}.\cr}
$$
The same estimates can be can proved for $g_{mr}$, $m=1,2,3$, $r=1,2$. Hence (1) holds
$$\eqal{
&|\partial_{\xi'}^1g_{11}|\le{|\xi_1|\over|\xi|(|\tau|+|\xi|)}\le{c\over|\tau|},\cr
&|\partial_{\xi'}^2g_{11}|\le{1\over|\xi|(|\tau|+|\xi|)}\le{c\over|\xi|\,|\tau|},\cr
&|\partial_{\xi'}^3g_{11}|\le{c\over|\xi|^2|\tau|}+{c\over|\xi|\,|\tau|^2},\cr
&|\partial_{\xi'}^4g_{11}|\le{c\over|\xi|^3|\tau|}+{c\over|\xi|^2|\tau|^2}+ {c\over|\xi|\,|\tau|^3}.\cr}
$$
The same estimates can be proved for $g_{mr}$, $m=1,2,3$, $r=1,2$. Hence (2) holds.

Consider $g_{13}$. Then
$$\eqal{
|\partial_{\xi_0}^1g_{13}|&\le c\bigg|{\xi_1\over|\xi|(\tau+|\xi|)}\bigg|+c \bigg|{\xi_1(|\xi|^2+\tau^2)\over|\xi|(\tau+|\xi|)^2}{1\over\tau}\bigg|\cr
&\le c{1\over|\tau|+|\xi|}+c{|\xi|^2+|\tau|^2\over(|\tau+|\xi|)^2}{1\over|\tau|}\le {c\over|\tau|},\cr
|\partial_{\xi_0}^2g_{13}|&\le{c\over|\tau|^3}.\cr}
$$
Next, we have
$$\eqal{
&|\partial_{\xi'}g_{13}|\le c+c{|\tau|\over|\xi|},\quad |\partial_{\xi'}^2g_{13}|\le{c\over|\tau|}+{c\over|\xi|}+{c|\tau|\over|\xi|^2},\cr
&|\partial_{\xi'}^3g_{13}|\le{c\over|\xi|(|\tau|+|\xi|)}+{c\over(|\tau|+|\xi|)^2},\cr
&|\partial_{\xi'}^4g_{13}|\le{c\over|\xi|^2(|\tau|+|\xi|)}+ {c\over|\xi|(|\tau|+|\xi|)^2}.\cr}
$$
The same estimates can be derived for $g_{m3}$, $m=2,3$. Hence (3) and (4) hold.

Consider $h_{11}$. Qualitatively, we have
$$\eqal{
&|h_{11,\xi_0}|\le{c\over|\tau|^3}+c{\xi_1^2\over|\tau|^3|\xi|(|\tau|+|\xi|)}+c {\xi_1^2\over|\tau|^2|\xi|(|\tau|+|\xi|)^2}\le c/|\tau|^3,\cr
&|h_{11,\xi_0\xi_0}|\le c/|\tau|^5.\cr}
$$
Next
$$\eqal{
|h_{11,\xi'}|&\le{c\over|\tau|^2}+c{|\xi_1|\over|\tau|\,|\xi|(|\tau|+|\xi|)}+ {c|\xi_1|^2\over|\tau|^2|\xi|(|\tau|+|\xi|)}\dots\le c/|\tau|^2,\cr
|\partial_{\xi'}^2h_{11}|&\le{c\over|\tau|^3}+{c\over|\tau\,|\xi|(|\tau|+|\xi|)}+ {c|\xi|\over|\tau|^2|\xi|(|\tau+|\xi|)}\dots\cr
&\le c/|\tau|^3+{c\over|\tau|^2|\xi|},\cr
|\partial_{\xi'}^3h_{11}|&\le c/|\tau|^4+{c\over|\tau|\xi^2(|\tau|+|\xi|)}+ {c\over|\tau|^2|\xi|(|\tau|+|\xi|)}+\cdots\cr
&\le c/|\tau|^4+{c\over|\tau|^2|\xi|^2}+{c\over|\tau|^3|\xi|}.\cr
|\partial_{\xi'}^4h_{11}|&\le c/|\tau|^5+{c\over|\tau|\,|\xi|^3(|\tau|+|\xi|)}+ {c\over|\tau|\,|\xi|(|\tau|+|\xi|)^3}\cr
&\quad+{c\over|\tau|\,|\xi|^2(|\tau|+|\xi|)^2}+\cdots\cr
&\le c/|\tau|^5+{c\over|\tau|^2|\xi|^3}+{c\over|\tau|^4|\xi|}+{c\over|\tau|^3|\xi|^2}.\cr}
$$
Similar estimates hold for $h_{mr}$, $m=1,2,3$, $r=1,2$. Hence (5) and (6) hold.

Finally, we examine $h_{13}$. Then we obtain
$$\eqal{
|h_{13,\xi_0}|&\le c/|\tau|^2,\quad |h_{13,\xi_0\xi_0}|\le c/|\tau|^4,\cr
|\partial_{\xi'}h_{13}|&\le{c\over|\tau|+|\xi|}+{c|\tau|\over|\xi|(|\tau|+|\xi|)} \le{c\over|\tau|}+{c\over|\xi|},\cr
|\partial_{\xi'}^2h_{13}|&\le{c\over|\xi|(|\tau|+|\xi|)}+ {|\xi|+|\tau|\over|\xi|^2(|\xi|+|\tau|)}+{|\xi|+|\tau|\over|\xi|(|\xi|+|\tau|)^2}\cr
&\le{c\over|\xi|^2}+{c\over|\tau|^2},\cr
|\partial_{\xi'}^3h_{13}|&\le{c\over|\xi|^2(|\tau|+|\xi|)}+ {c\over|\xi|(|\tau|+|\xi|)^2}+\cdots\cr
&\le{c\over|\xi|^2|\tau|}+{c\over|\xi|\,|\tau|^2},\cr
|\partial_{\xi'}^4h_{13}|&\le{c\over|\xi|^3(|\tau|+|\xi|)}+ {c\over|\xi|^2(|\tau|+|\xi|)^2}+{c\over|\xi|(|\tau|+|\xi|)^3}+\cdots\cr
&\le{c\over|\xi|^3|\tau|}+{c\over|\xi|^2|\tau|^2}+{c\over|\xi|\,|\tau|^3}.\cr}
$$
The same estimates hold for $h_{23}$. We do not need to estimate $h_{33}$ because $h_{33}=1$. Hence (7) and (8) are proved. Estimates for $g_{4r}$, $h_{4r}$, $r=1,2,3$, can be derived similarly. This concludes the proof of Lemma \ref{l4.1}.
\end{proof}

\begin{remark}\label{r4.2}
From (\ref{4.30})--(\ref{4.33}) the solution to (\ref{4.19}) is expressed in the form
\begin{equation}\eqal{
&F_2u_m\equiv\tilde u_m=\sum_{r=1}^3(g_{mr}e_1\tilde d_r+h_{mr}e_0\tilde d_r),\cr
&F_2q\equiv\tilde q=\sum_{r=1}^3(g_{4r}e_1\tilde d_r+h_{4r}e_0\tilde d_r),\cr}
\label{4.34}
\end{equation}
where $m=1,2,3$ and $g_{mr}$, $h_{mr}$, $g_{4r}$, $h_{4r}$, $r=1,2,3$ are defined by fromulas (\ref{4.30})--(\ref{4.33}).
\end{remark}

\begin{theorem}\label{t4.3}
Let $p,q\in(1,\infty)$, $\sigma\in(0,1)$, $d_k\in B_{p,q,\gamma}^{1+\sigma-1/p,1/2-\sigma/2-1/2p}(\R^2\times\R_+)$, $k=1,2$, $d_3\in B_{p,q,\gamma}^{2+\sigma-1/p,1+\sigma/2-1/2p}(\R^2\times\R_+)$.\\
Then there exists a solution to problem (\ref{4.17}) such that $u\in B_{p,q,\gamma}^{2+\sigma,1+\sigma/2}(\R_+^3\times\R_+)$, $\nabla q\in B_{p,q,\gamma}^{\sigma,\sigma/2}(\R_+^3\times\R_+)$ and
\begin{equation}
\|u\|_{B_{p,q,\gamma}^{2+\sigma,1+\sigma/2}(\R_+^3\times\R_+)}\le cI,
\label{4.35}
\end{equation}
\begin{equation}
\|\nabla q\|_{B_{p,q,\gamma}^{\sigma,\sigma/2}(\R_+^3\times\R_+)}\le cI,
\label{4.36}
\end{equation}
where
$$\eqal{
I&=\sum_{k=1}^2 \|d_k\|_{B_{p,q,\gamma}^{1+\sigma-1/p,(1+\sigma-1/p)/2}(\R^2\times\R_+)}\cr
&\quad+\|d_3\|_{B_{p,q,\gamma}^{2+\sigma-1/p,(2+\sigma-1/p)/2}(\R^2\times\R_+)}.\cr}
$$
\end{theorem}

To prove estimates (\ref{4.35}) and (\ref{4.36}) we recall Definition \ref{d2.5} to describe the norms from the l.h.s. of (\ref{4.35}) and (\ref{4.36}). We restrict our considerations to prove estimate (\ref{4.35}) only.

Hence, we have
\begin{equation}\eqal{
&\|u\|_{B_{p,q,\gamma}^{2+\sigma,1+\sigma/2}(\R_+^3\times\R)}=\sum_{m=1}^3\|u_{m\gamma}\|_{B_{p,q}^{2+\sigma,1+\sigma/2}(\R_+^3\times\R)}=\sum_{m=1}^3\|u_{m\gamma}\|_{L_p(\R_+^3\times\R)}\cr
&\quad+\sum_{m=1}^3\bigg[\sum_{k=0}^\infty\bigg(\sum_{j\le2}\intop_{\R_+^3\times\R} |2^{(2+\sigma-j)k}(F_2^{-1}\varphi_kF_2\partial_{x_3}^ju_{m\gamma})(x,t)|^p dxdt\bigg)^{q/p}\bigg]^{1/q}\cr
&\quad+\sum_{m=1}^3\bigg[\sum_{k=0}^\infty\bigg(\intop_\R dt\intop_{\R_+}dx_3\intop_{\R_+}dz\intop_{\R^2}dx'\cdot\cr
&\quad\cdot{|F_2^{-1}\varphi_kF_2(\partial_{x_3}^2u_{m\gamma}(\bar x',x_3+z)-\partial_{x_3}^2u_{m\gamma}(\bar x',x_3)|^p\over z^{1+p\sigma}}\bigg)^{q/p}\bigg]^{1/q}\cr
&\equiv\sum_{m=1}^3(\|u_{m\gamma}\|_{L_p(\R_+^3\times\R)}+ \|u_m\|_{1,B_{p,q,\gamma}^{2+\sigma,1+\sigma/2}(\R_+^3\times\R)}+\|u_m\|_{2,B_{p,q,\gamma}^{2+\sigma,1+\sigma/2}(\R_+^3\times\R)}),\cr}
\label{4.37}
\end{equation}
where $\bar x'=(x_1,x_2,t)$.

\begin{lemma}\label{l4.4}
Let the assumptions of Theorem \ref{t4.3} be satisfied. Then
\begin{equation}\eqal{
&\|u\|_{1,B_{p,q,\gamma}^{2+\sigma,1+\sigma/2}(\R_+^3\times\R)}\cr
&=\sum_{m=1}^3\|u_m\|_{1,B_{p,q,\gamma}^{2+\sigma,1+\sigma/2}(\R_+^3\times\R)}\cr
&\le c\bigg(\sum_{k=1}^2 \|d_k\|_{B_{p,q,\gamma}^{1+\sigma-1/p,(1+\sigma-1/p)/2}(\R^2\times\R)}\cr
&\quad+ \|d_3\|_{B_{p,q,\gamma}^{2+\sigma-1/p,(2+\sigma-1/p)/2}(\R^2\times\R)}\bigg).\cr}
\label{4.38}
\end{equation}
\end{lemma}

\begin{proof}
Using (\ref{4.34}) we have
$$\eqal{
&\|u_m\|_{1,B_{p,q,\gamma}^{2+\sigma,(1+\sigma)/2}(\R_+^3\times\R)}\cr
&=\bigg[\sum_{k=0}^\infty\bigg(\sum_{j\le 2}\intop_{\R_+^3\times\R} \bigg|2^{(2+\sigma-j)k}F_2^{-1}\varphi_k\sum_{r=1}^3(g_{mr}\partial_{x_3}^je_1\cr
&\quad+h_{mr}\partial_{x_3}^je_0)\tilde d_{r\gamma}\bigg|^pdxdt\bigg)^{q/p}\bigg]^{1/q}\cr
&\le c\bigg(\sum_{k=0}^\infty\sum_{j\le 2}\sum_{r=1}^3I_{1kjr\gamma}^q\bigg)^{1/q},\cr}
$$
where
$$\eqal{
I_{1kjr\gamma}&=\bigg(\intop_{\R_+^3\times\R}|2^{(2+\sigma-j)k}F_2^{-1}\varphi_k (g_{mr}\partial_{x_3}^je_1\cr
&\quad+h_{mr}\partial_{x_3}^je_0)F_2d_{r\gamma}|^pdxdt\bigg)^{1/p}.\cr}
$$
Introduce the family of functions $\{\psi_j(\bar\xi)\}$, $\bar\xi=(\xi',\xi_0)$, $\xi'=(\xi_1,\xi_2)$ such that $\supp\psi_0\subset\{\bar\xi:|\bar\xi|_a\le 4\}$, $\supp\psi_j\subset\{\bar\xi:2^{j-2}\le|\bar\xi|_a\le 2^{j+2}\}$ and $\psi_j(\bar\xi)=1$ for $\bar\xi\in\supp\varphi_j$. Then
$$\eqal{
I_{1kjr\gamma}&=\bigg(\intop_{\R_+^3\times\R}\bigg|\sum_{l=0}^\infty 2^{(2+\sigma-j)k}F_2^{-1}\psi_l(g_{mr}\partial_{x_3}^je_1\cr
&\quad+h_{mr}\partial_{x_3}^je_0)\varphi_k\varphi_lF_2d_{r\gamma}\bigg|^p dxdt\bigg)^{1/p}\cr
&=\bigg(\intop_{\R\times\R_+}\bigg|\sum_{l=0}^\infty 2^{(2+\sigma-j)k} (F_2^{-1}\psi_l(g_{mr}\partial_{x_3}^je_1+h_{mr}\partial_{x_3}^je_0\bigg)\cdot\cr
&\quad\cdot F_2F_2^{-1}\varphi_kF_2F_2^{-1}\varphi_lF_2d_{r\gamma}\bigg|^pd\bar xdx_3\bigg)^{1/p},\cr}
$$
where $\bar x=(x',t)=(x',x_0)$, $x'=(x_1,x_2)\in\R^2$.

Continuing, we can rewrite $I_{1kjr\gamma}$ as
$$\eqal{
I_{1kjr\gamma}&=\bigg(\intop_{\R^3\times\R_+}\bigg|\sum_{l=0}^\infty 2^{(2+\sigma-j)k}[F_2^{-1}\psi_l(g_{mr}\partial_{x_3}^je_1\cr
&\quad+h_{mr}\partial_{x_3}^je_0)F_2F_2^{-1}\varphi_k\star F_2^{-1}\varphi_lF_2d_{r\gamma}](\bar x,x_3)\bigg|^pd\bar xdx_3\bigg)^{1/p}\cr
&=\bigg(\intop_{\R^3\times\R_+}\bigg|\sum_{l=0}^\infty 2^{(2+\sigma-j)k}\intop_{\R^3}d\bar y[F_2^{-1}\psi_l(g_{mr}\partial_{x_3}^je_1+h_{mr}\partial_{x_3}^je_0)\cdot\cr
&\quad\cdot F_2(F_2^{-1}\varphi_k)](\bar y,x_3)(F_2^{-1}\varphi_lF_2d_{r\gamma})(\bar x-\bar y)]\bigg|^pd\bar xdx_3\bigg)^{1/p},\cr}
$$
where $\bar y=(y',y_0)$, $y'=(y_1,y_2)\in\R^2$.

Continuing, we have
$$\eqal{
I_{1kjr\gamma}&=\bigg(\intop_{\R^3\times\R_+}\bigg|\sum_{l=0}^\infty 2^{(2+\sigma-j)k}F_2^{-1}(F_2F_2^{-1}(g_{mr}\partial_{x_3}^je_1\cr
&\quad+h_{mr}\partial_{x_3}^je_0)F_2\intop_{\R^3}F_2^{-1}\varphi_k(\bar y,x_3)F_2^{-1}\varphi_lF_2d_{r\gamma}(\bar x-\bar y)d\bar y\bigg|^pd\bar xdx_3\bigg)^{1/p}\cr
&=\bigg(\intop_{\R^3\times\R_+}\bigg|\sum_{l=0}^\infty 2^{(2+\sigma-j)k}\intop_{\R^3}d\bar y[F_2^{-1}\psi_l(g_{mr}\partial_{x_3}^je_1+h_{mr}\partial_{x_3}^je_0)\star\cr
&\quad\star F_2^{-1}\varphi_k](\bar y,x_3)(F_2^{-1}\varphi_lF_2d_{r\gamma})(\bar x-\bar y)\bigg|^pd\bar xdx_3\bigg)^{1/p}.\cr}
$$
From Lemma \ref{l3.4} we have
\begin{equation}\eqal{
&[F_2^{-1}(f(2^l\cdot,2^{2l}\cdot))\star F_2^{-1}(g(2^l\cdot,2^{2l}\cdot))](\bar y)\cr
&=2^{-4l}(F_2^{-1}f\star F_2^{-1}g)(2^{-l}\tilde y),\cr}
\label{4.39}
\end{equation}
where $f=\psi_l(g_{mr}\partial_{x_3}^je_1+h_{mr}\partial_{x_3}^je_0)$, $g=\varphi_k$, $\tilde y=(y',2^{-l}y_0)$, $y'=(y_1,y_2)$.

Moreover, we used in (\ref{4.38}) the notation
$$
h(\xi',\xi_0)=\tilde h(\xi',\gamma+i\xi_0),\quad h\in\{f,g\}.
$$
Using the change of variables $y_i=2^{-l}\omega_i$, $i=1,2$, $y_0=2^{-2l}\omega_0$ and the notation $\tilde\omega=(\omega',2^{-l}\omega_0)$, $\omega'=(\omega_1,\omega_2)$, $\bar\omega=(\omega',\omega_0)$ we obtain
$$\eqal{
&I_{1kjr\gamma}=\bigg[\intop_{\R^3\times\R_+}\bigg|\sum_{l=0}^\infty 2^{(2+\sigma-j)k}2^{-4l}\intop_{\R^3}d\bar\omega[F_2^{-1}\psi_l(g_{mr}\partial_{x_3}^je_1\cr
&\quad+h_{mr}\partial_{x_3}^je_0)\star F_2^{-1}\varphi_k](2^{-l}\tilde\omega,x_3)(F_2^{-1}\varphi_lF_2d_{r\gamma})(\bar x-2^{-l}\tilde\omega)\bigg|^pd\bar xdx_3\bigg]^{1/p}.\cr}
$$
Then (\ref{4.39}) yields
$$\eqal{
I_{1kjr\gamma}&=\bigg\{\intop_{\R^3\times\R_+}d\bar xdx_3\bigg|\sum_{l=0}^\infty 2^{(2+\sigma-j)k}\intop_{\R^3}d\bar\omega\cdot\cr
&\quad\cdot[F_2^{-1}(\psi_l(g_{mr}\partial_{x_3}^je_1+h_{mr}\partial_{x_3}^je_0)) (2^l\cdot,2^{2l}\cdot)\star\cr
&\quad\star F_2^{-1}(\varphi_k(2^l\cdot,2^{2l}\cdot))] (\bar\omega,x_3)(F_2^{-1}\varphi_lF_2d_{r\gamma})(\bar x-\bar\omega)\bigg|^p\bigg\}^{1/p}\cr
&=\bigg\{\intop_{\R_+}dx_3\intop_{\R^3}d\bar x\bigg|\sum_{l=0}^\infty 2^{(2+\sigma-j)k}\intop_{\R^3}d\bar\omega[F_2^{-1}(\psi_l(g_{mr}\partial_{x_3}^je_1\cr
&\quad+h_{mr}\partial_{x_3}^je_0)(2^l\cdot,2^{2l}\cdot))\star F_2^{-1}(\varphi_k(2^l\cdot,2^{2l}\cdot))](\bar\omega,x_3)\cdot\cr
&\quad\cdot(F_2^{-1}\varphi_lF_2d_{r\gamma})(\bar x-\bar\omega)\bigg|^p\bigg\}^{1/p}.\cr}
$$
Next, the Minkowski inequality with respect to $\bar x$ gives
$$\eqal{
I_{1kjr\gamma}&\le\bigg[\intop_{\R_+}dx_3\bigg|\sum_{l=0}^\infty 2^{(2+\sigma-j)k}\intop_{\R^3}d\bar\omega[F_2^{-1}(\psi_l(g_{mr}\partial_{x_3}^je_1\cr
&\quad+h_{mr}\partial_{x_3}^je_0)(2^l\cdot,2^{2l}\cdot))\star F_2^{-1}(\varphi_k(2^l\cdot,2^{2l}\cdot))](\bar\omega,x_3)\cdot\cr
&\quad\cdot\bigg(\intop_{\R^3}d\bar x|(F_2^{-1}\varphi_lF_2d_{r\gamma})(\bar x-\bar\omega)|^p\bigg)^{1/p}\bigg|^p\bigg]^{1/p}\equiv I_{1kjr\gamma}^1.\cr}
$$
The change of variables $\bar z=\bar x-\bar\omega$ in the integral with respect to $\bar x$ implies
$$\eqal{
I_{1kjr\gamma}^1&=\bigg\{\intop_{\R_+}dx_3\bigg|\sum_{l=0}^\infty 2^{(2+\sigma-j)k}\intop_{\R^3}d\bar\omega[F_2^{-1}(\psi_l(g_{mr}\partial_{x_3}^je_1\cr
&\quad+h_{mr}\partial_{x_3}^je_0)(2^l\cdot,2^{2l}\cdot)\star F_2^{-1}(\varphi_k(2^l\cdot,2^{2l}\cdot))](\bar\omega,x_3)\cdot\cr
&\quad\cdot\bigg(\intop_{\R^3}d\bar z|(F_1^{-1}\varphi_lF_1d_{r\gamma})(\bar z)|^p\bigg)^{1/p}\bigg|^p\bigg\}^{1/p},\cr}
$$
where $\bar z=(z',z_0)$, $z'=(z_1,z_2)$ and we used the fact that $F_2d_{r\gamma}=F_1d_{r\gamma}$.

Next, applying the H\"older inequality in the integral with respect to $\bar\omega$ and replacing $\bar\omega$ by $\bar y$ and $\bar z$ by $\bar x$ we obtain
$$\eqal{
I_{1kjr\gamma}^1&\le\bigg\{\intop_{\R_+}dx_3\bigg|\sum_{l=0}^\infty 2^{(2+\sigma-j)k}\bigg(\intop_{\R^3}d\bar y{1\over(1+|\bar y|_a)^4}\bigg)^{1/2}\cdot\cr
&\quad\cdot\bigg(\intop_{\R^3}d\bar y|F_2^{-1}(\psi_l(g_{mr}\partial_{x_3}^je_1+ h_{mr}\partial_{x_3}^je_0)(2^l\cdot,2^{2l}\cdot)\star\cr
&\quad\star F_2^{-1}(\varphi_k(2^l\cdot,2^{2l}\cdot))](\bar y,x_3)(1|+|\bar y|_a)^2|^2\bigg)^{1/2}\bigg|^p\bigg\}^{1/p}\cdot\cr
&\quad\cdot\bigg(\intop_{\R^3}d\bar x|(F_2^{-1}\varphi_lF_2d_{r\gamma})(\bar x)|^p\bigg)^{1/p}\equiv I_{1kjr\gamma}^2.\cr}
$$
Using that
$$
\bigg(\intop_{\R^3}d\bar y{1\over(1+|\bar y|_a)^4}\bigg)^{1/2}\le c
$$
and the Parseval identity we have
$$\eqal{
I_{1kjr\gamma}^2&\le c\bigg\{\intop_{\R_+}dx_3\bigg|\sum_{l=0}^\infty 2^{(2+\sigma-j)(k-l)}2^{(2+\sigma-j)l}\cdot\cr
&\quad\cdot\|\psi_l(g_{mr}\partial_{x_3}^je_1+h_{mr}\partial_{x_3}^je_0) (2^l\cdot,x_3,2^{2l}\cdot)\varphi_k (2^l\cdot,2^{2l}\cdot)\|_{W_2^{4,2}(\R^3)}\bigg|^p\bigg\}^{1/p}\cdot\cr
&\quad\cdot\bigg(\intop_{\R^3}d\bar x|F_1^{-1}\varphi_lF_1d_{r\gamma}(\bar x)|^p\bigg)^{1/p}\cr
&\le c\sum_{l=0}^\infty 2^{(2+\sigma-j)(k-l)} 2^{(2+\sigma-j)l}\bigg(\intop_{\R_+}dx_3\|\psi_l(g_{mr}\partial_{x_3}^je_1\cr
&\quad+h_{mr}\partial_{x_3}^je_0)(2^l\cdot,x_3,2^{2l}\cdot)\varphi_k (2^l\cdot,2^{2l}\cdot)\|_{W_2^{4,2}(\R^3)}^p\bigg)^{1/p}\cdot\cr
&\quad\cdot\bigg(\intop_{\R^3}d\bar x|(F_1^{-1}\varphi_kF_1d_{r\gamma})(\bar x)|^p\bigg)^{1/p}\equiv I_{1kjr\gamma}^3.\cr}
$$
Using (\ref{4.41}) (see Lemma \ref{l4.5} below)we get
\begin{equation}\eqal{
I_{1kjr\gamma}^3&\le c\sum_{l=0}^\infty 2^{(\delta-j-L)|l-k|}2^{(\delta-1/p-c_r)l}\cdot\cr
&\quad\cdot\|F_1^{-1}\varphi_lF_1d_{r\gamma}\|_{L_p(\R^3)},\cr}
\label{4.40}
\end{equation}
where $\delta=2+\sigma+8$, $c_r=1$ for $r=1,2$, $c_r=0$ for $r=3$.

Therefore, by the H\"older inequality
$$\eqal{
&\|u_m\|_{1,B_{p,q,\gamma}^{2+\sigma,1+\sigma/2}(\R_+^3\times\R)}\le c\bigg\{\sum_{k=0}^\infty\sum_{r=1}^3\bigg|\sum_{l=0}^\infty\cdot\cr
&\quad\cdot 2^{(\delta-L)|l-k|}2^{(2+\sigma-1/p-c_r)l}\|F_1^{-1}\varphi_lF_1 d_{r\gamma}\|_{L_p(\R^3)}\bigg|^q\bigg\}^{1/q}\cr
&\le c\bigg(\sum_{k=0}^\infty\sum_{r=1}^3\sum_{l=0}^\infty 2^{(\delta+\varepsilon-L)|l-k|q}2^{(2+\sigma-1/p-c_r)lq}\cdot\cr
&\quad\cdot\|F_1^{-1}\varphi_lF_1d_{r\gamma}\|_{L_p(\R^3)}^q\bigg)^{1/q},\cr}
$$
where $\varepsilon>0$ is arbitrary small. Assuming that $L>\delta+\varepsilon$ we obtain (\ref{4.37}). This ends the proof.
\end{proof}

To prove (\ref{4.40}) we need

\begin{lemma}\label{l4.5}
We have
\begin{equation}\eqal{
J_{mr}&=\bigg(\intop_{\R_+}dx_3\|[\psi_1(g_{mr}\partial_{x_3}^je_1+ h_{mr}\partial_{x_3}^je_0)](2^l\cdot,2^{2l}\cdot)\cdot\cr
&\quad\cdot\varphi_k(2^l\cdot,2^{2l}\cdot)\|_{W_2^{4,2}(\R^2\times\R)}^p\bigg)^{1/p}\le c2^{l(j-1/p-c_r)}2^{(\gamma-L)|l-k|},\cr}
\label{4.41}
\end{equation}
where $c$ depends on $\gamma$ for $l=0$; $c_r=1$ for $r=1,2$; $c_r=0$ for $r=3$; $m=1,2,3$; $L>0$ can be chosen sufficiently large.
\end{lemma}

\begin{proof}
Introduce the notation $\partial_\xi^{\bar s_i}=\partial_{\xi'}^{s'_i}\partial_{\xi_0}^{s_i}$, where $\bar s_i=|s'_i|+2s_i$, $s'_i=(s_{1i},s_{2i})$. Then $J_{mr}$ can be written as
$$\eqal{
J_{mr}&=\bigg(\intop_{\R_+}dx_3\sum_{\sum_{i=1}^4(s'_i+2s_i)\le 4} \|(\partial_\xi^{\bar s_1}\psi_l)(2^l\cdot,2^{2l}\cdot)\cr
&\quad(\partial_\xi^{\bar s_2}g_{mr}\partial_\xi^{\bar s_3}\partial_{x_3}^je_1+ \partial_\xi^{\bar s_2}h_{mr}\partial_\xi^{\bar s_3}\partial_{x_3}^je_0) (2^l\cdot,x_3,2^{2l}\cdot)\cdot\cr
&\quad\cdot(\partial_\xi^{\bar s_4}\varphi_k)(2^l\cdot,2^{2l}\cdot)\|_{L_2(\R^3)}^p\bigg)^{1/p},\cr}
$$
where
$$
\partial_{\xi'}^\sigma=\partial_{\xi_1}^{\sigma_1}\partial_{\xi_2}^{\sigma_2},\quad |\sigma|=\sigma_1+\sigma_2.
$$
Applying the Minkowski inequality we get
$$\eqal{
J_{mr}&\le c\sum_{\sum_{i=1}^4(s'_i+2s_i)\le 4}\|(\partial_\xi^{\bar s_1}\psi_l) (2^l\cdot,2^{2l}\cdot)\cdot\cr
&\quad\cdot(\partial_\xi^{\bar s_2}g_{mr}(2^l\cdot,2^{2l}\cdot)\|\partial_\xi^{\bar s_3}\partial_{x_3}^je_1\|_{L_p(\R_+)}\cdot\cr
&\quad\cdot(\partial_\xi^{\bar s_4}\varphi_k)(2^l\cdot,2^{2l}\cdot)\|_{L_2(\R^3)}\cr
&\quad+\|(\partial_\xi^{\bar s_1}\psi_l)(2^l\cdot,2^{2l}\cdot)\partial_\xi^{\bar s_2}h_{mr}(2^l\cdot,2^{2l}\cdot)\cdot\cr
&\quad\cdot\|\partial_\xi^{\bar s_3}\partial_{x_3}^je_0\|_{L_p(\R_+)}(2^l\cdot,2^{2l}\cdot)(\partial_\xi^{\bar s_4}\varphi_k)(2^l\cdot,2^{2l}\cdot)\|_{L_2(\R^3)}\cr
&\equiv J_{mr}^1.\cr}
$$
From the properties of $\psi_l$ it follows that
$$\eqal{
&\supp\psi_l(2^l\cdot,2^{2l}\cdot)\subset\{\bar\xi\colon|\bar\xi|_a\le 4\}\equiv A\quad &{\rm for}\ \ l=0,\cr
&\supp\psi_l(2^l\cdot,2^{2l}\cdot)\subset\{\bar\xi\colon 1/4\le|\bar\xi|_a\le 4\}\equiv A\quad &{\rm for}\ \ l\not=0\cr}
$$
and
\begin{equation}
|\partial_\xi^{\bar s_1}\psi_l(2^l\cdot,2^{2l}\cdot)|\le c\quad {\rm for}\ \ \bar\xi\in A.
\label{4.42}
\end{equation}
Therefore, we obtain
\begin{equation}\eqal{
J_{mr}^1&\le c\sum_{\sum_{i=1}^4(s'_i+2s_i)\le 4}(\|\partial_\xi^{\bar s_2}g_{mr}(2^l\cdot,2^{2l}\cdot)\cdot\cr
&\quad\cdot\|\partial_\xi^{\bar s_3}\partial_{x_3}^je_1\|_{L_p(\R_+)}(2^l\cdot,2^{2l}\cdot)(\partial_\xi^{\bar s_4}\varphi_k)(2^l\cdot,2^{2l}\cdot)\|_{L_2(A)}\cr
&\quad+\|\partial_\xi^{\bar s_2}h_{mr}(2^l\cdot,2^{2l}\cdot)\|\partial_\xi^{\bar s_3}\partial_{x_3}^je_0\|_{L_p(\R_+)}(2^l\cdot,2^{2l}\cdot)\cdot\cr
&\quad\cdot(\partial_\xi^{\bar s_4}\varphi_k)(2^l\cdot,2^{2l}\cdot)\|_{L_2(A)}.\cr}
\label{4.43}
\end{equation}
Recall that $\tau^2=\gamma+2^{2l}\xi'^2+2^{2l}\xi_0i$, $\gamma\le 2^{2l}$. Then
\begin{equation}
c_12^l\le|\tau|\le c_22^l
\label{4.44}
\end{equation}
and from Lemma \ref{l4.1} we have
$$
c_12^l\le|\xi|\le c_22^l
$$
for $\bar\xi\in A$.

Then
\begin{equation}\eqal{
&|\partial_\xi^{\bar s_2}g_{mr}(2^l\cdot,2^{2l}\cdot)|\le c,\quad r=1,2,\cr
&|\partial_\xi^{\bar s_2}g_{m3}(2^l\cdot,2^{2l}\cdot)|\le c2^l,\cr
&|\partial_\xi^{\bar s_2}h_{mr}(2^l\cdot,2^{2l}\cdot)|\le c2^{-l},\quad r=1,2,\cr
&|\partial_\xi^{\bar s_2}h_{m3}(2^l\cdot,2^{2l}\cdot)|\le c.\cr}
\label{4.45}
\end{equation}
From (\ref{2.13}) and (\ref{2.15}) of Lemma \ref{l2.19} we have
\begin{equation}\eqal{
&\|\partial_{x_3}^j\partial_\xi^{\bar s_3}e_1\|_{L_p(\R_+)}(2^l\cdot,2^{2l}\cdot)\le c2^{l(j-1/p-1)},\cr
&\|\partial_{x_3}^j\partial_\xi^{\bar s_3}e_0\|_{L_p(\R_+)}(2^l\cdot,2^{2l}\cdot)\le c2^{l(j-1/p)}.\cr}
\label{4.46}
\end{equation}
Since $\{\varphi_k\}\in{\cal A}_{aL}(\R^4)$ (see Definition \ref{d2.16}) we have
\begin{equation}\eqal{
&\|\partial_\xi^{\bar s_4}\varphi_k(2^l\cdot,2^{2l}\cdot)\|_{L_2(A)}\cr
&=\|\partial_\xi^{\bar s_4}\varphi_k(2^k2^{(l-k)},2^{2k}2^{2(l-k)}\cdot)\|_{L_2(A)}\cr
&\le c2^{6|l-k|}\|\partial_{y'}^{s'_4}\partial_{y_0}^{s_4}\varphi_k (2^k\cdot,2^{2k}\cdot)\|_{L_2(B)}\cr
&\le c2^{(8-L)|l-k|},\cr}
\label{4.47}
\end{equation}
where $L$ is chosen sufficiently large,
$$\eqal{
&B=\{\bar y\colon 2^{l-k-2}\le|\bar y|_a\le 2^{l-k+2}\}\quad &{\rm for}\ \ l\not=0,\cr
&B=\{\bar y\colon|\bar y|_a\le 2^{l-k+2}\}\quad &{\rm for}\ \ l=0\cr}
$$
and we have used the change of variables
$$
y_0=2^{2(l-k)\xi_0},\quad y'=2^{(l-k)\xi'},\quad \xi'=(\xi_1,\xi_2),\quad y'=(y_1,y_2).
$$
In view of the above estimates (\ref{4.41}) holds. This ends the proof of Lemma \ref{l4.5}.
\end{proof}

Now, we shall derive the estimate for
$$\eqal{
&\|u\|_{2,B_{p,q,\gamma}^{2+\sigma,1+\sigma/2}(\R_+^3\times\R)}=\bigg[ \sum_{k=0}^\infty\bigg(\intop_{\R_+}\intop_{\R_+^3\times\R}\cdot\cr
&\quad\cdot{|F_2^{-1}\varphi_kF_2U_\gamma)(\bar x,x_3,z)|^p\over z^{1+p\sigma}} d\bar xdx_3dz\bigg)^{q/p}\bigg]^{1/q},\cr}
$$
where $U=\partial_{x_3}^2u_\gamma(\bar x,x_3+z)-\partial_{x_3}^2u(\bar x,x_3)$.

\begin{lemma}\label{l4.6}
Let the assumptions of Theorem \ref{t4.3} be satisfied. Then
\begin{equation}\eqal{
&\|u\|_{2,B_{p,q,\gamma}^{2+\sigma,1+\sigma/2}(\R_+^3\times\R)}\cr
&\le c\bigg(\sum_{k=1}^2 \|d_k\|_{B_{p,q,\gamma}^{1+\sigma-1/p,(1+\sigma-1/p)/2}(\R^2\times\R_+)}\cr
&\quad+\|d_3\|_{B_{p,q,\gamma}^{2+\sigma-1/p,(2+\sigma-1/p)/2}(\R^2\times\R)}.\cr}
\label{4.48}
\end{equation}
\end{lemma}

\begin{proof}
Since
$$
F_2U_m=\sum_{r=1}^3(g_{mr}\partial_{x_3}^2E_1+h_{mr}\partial_{x_3}^2E_0)\tilde d_{r\gamma},
$$
where
$$\eqal{
&E_1=e_1(\bar\xi,x_3+z)-e_1(\bar\xi,x_3),\cr
&E_0=e_0(\bar\xi,x_3+z)-e_0(\bar\xi,x_3)\cr}
$$
we have
$$\eqal{
&\|u\|_{2,B_{p,q,\gamma}^{2+\sigma,1+\sigma/2}(\R_+^3\times\R)}\cr
&=\!\!\bigg[\!\sum_{k=0}^\infty\sum_{m=1}^3\!\bigg(\!\intop_{\R_+^3\times\R}\intop_{\R_+}\! {|F_2^{-1}\varphi_k\!\sum_{r=1}^3(g_{mr}\partial_{x_3}^2E_1\!+\!h_{mr}\partial_{x_3}^2E_0)\tilde d_{r\gamma}|^p\over z^{1+p\sigma}}dxdzdt\!\bigg)^{\!\!q/p}\bigg]^{\!1/q}\cr
&=c\bigg(\sum_{k=0}^\infty\sum_{m=1}^3\sum_{r=1}^3I_{2kmr\gamma}^q\bigg)^{1/q},\cr}
$$
where
$$
I_{2kmr\gamma}=\intop_{\R^3\times\R_+}\intop_{\R_+} {|F_2^{-1}\varphi_k(g_{mr}\partial_{x_3}^2E_1+h_{mr}\partial_{x_3}^2E_0)\tilde d_{r\gamma}|^p\over z^{1+p\sigma}}d\bar xdx_3dz,
$$
where $\bar x=(x',t)$, $x'=(x_1,x_2)$.

Introducing the same family of functions $\{\psi_l(\bar\xi)\}$ as in the proof of Lemma \ref{l4.4} we get
$$\eqal{
I_{2kmr\gamma}&=\bigg(\intop_{\R_+}\intop_{\R^3\times\R_+}\bigg|\sum_{l=0}^\infty {[F_2^{-1}\psi_l(g_{mr}\partial_{x_3}^2E_1+h_{mr}\partial_{x_3}^2E_0)\over z^{1/p+\sigma}}\cdot\cr
&\quad\cdot{F_2F_2^{-1}\varphi_kF_2F_2^{-1}\varphi_lF_2d_{r\gamma}](\bar x,x_3,z)\over z^{1/p+\sigma}}\bigg|^pd\bar xdx_3dz\bigg)^{1/p}\cr
&=\bigg(\intop_{\R_+}\intop_{\R^3\times\R_+}\bigg|\sum_{l=0}^\infty {[F_2^{-1}\psi_l(g_{mr}\partial_{x_3}^2E_1+h_{mr}\partial_{x_3}^2E_0)\over z^{1/p+\sigma}}\cdot\cr
&\quad\cdot{F_2F_2^{-1}\varphi_k\star F_2^{-1}\varphi_lF_2d_{r\gamma}\over z^{1/p+\sigma}}\bigg|^pd\bar xdx_3dz\bigg)^{1/p}\cr
&=\bigg(\intop_{\R_+}\intop_{\R^3\times\R_+}\bigg|\sum_{l=0}^\infty\intop_{\R^3}d\bar y{[F_2^{-1}\psi_l(g_{mr}\partial_{x_3}^2E_1+h_{mr}\partial_{x_3}^2E_0)\over z^{1/p+\sigma}}\cdot\cr
&\quad\cdot{F_2F_2^{-1}\varphi_k](\bar y,x_3)\cdot(F_2^{-1}\varphi_lF_2d_{r\gamma})(\bar x-\bar y)\over z^{1/p+\sigma}}\bigg|^pd\bar xdx_3dz\bigg)^{1/p}.\cr}
$$
Using the formula
\begin{equation}\eqal{
&F_2^{-1}(fF_2F_2^{-1}g)=F_2^{-1}(F_2F_2^{-1}fF_2F_2^{-1}g)\cr
&=F_2^{-1}F_2F_2^{-1}f\star F_2^{-1}g,\cr}
\label{4.49}
\end{equation}
where $f=\psi_l(g_{mr}\partial_{x_3}^2E_1+h_{mr}\partial_{x_3}^2E_0)$, $g=\varphi_k$ we obtain
$$\eqal{
I_{2kmr\gamma}=&\bigg[\intop_{\R_+}\intop_{\R^3\times\R_+}\bigg|\sum_{l=0}^\infty 2^{-4l}\intop_{\R^3}d\bar\omega{[F_2^{-1}\psi_l(g_{mr}\partial_{x_3}^2E_1+h_{mr}\partial_{x_3}^2E_0)\star\over\ }\cr
&{\star F_2^{-1}\varphi_k](2^{-l}\bar\omega,x_3,z)(F_2^{-1}\varphi_lF_2d_{r\gamma})(\bar x-2^{-l}\tilde\omega)\over z^{1/p+\sigma}}\bigg|^pd\bar xdx_3dz\bigg)^{1/p},\cr}
$$
where we used the change of variables $y_i=2^{-l}\omega_i$, $i=1,2$, $y_0=2^{-2l}\omega_0$ and the notation $\tilde\omega=(\omega',2^{-l}\omega_0)$, $\omega'=(\omega_1,\omega_2)$, $\bar\omega=(\omega',\omega_0)$.

Next, formula (\ref{4.49}) yields
$$\eqal{
I_{2kmr\gamma}&=\bigg[\intop_{\R_+}dx_3\intop_{\R^3}d\bar x\intop_{\R_+}{dz\over z^{1+p\sigma}}\bigg|\sum_{l=0}^\infty\intop_{\R^3}d\bar\omega\cdot\cr
&\quad\cdot\{|F_2^{-1}\psi_l(g_{mr}\partial_{x_3}^2E_1+h_{mr}\partial_{x_3}^2E_0)] (2^l\cdot,2^{2l}\cdot)\star\cr
&\quad\star(F_2^{-1}\varphi_k)(2^l\cdot,2^{2l}\cdot)\}(\bar\omega,x_3,z)(F_2^{-1} \varphi_lF_2d_{r\gamma})(\bar x-\bar\omega)\bigg|^p\bigg]^{1/p}\cr
&\le\bigg\{\intop_{\R_+}dx_3\intop_{\R_+}{dz\over z^{1+p\sigma}} \bigg|\sum_{l=0}^\infty\intop_{\R^3}d\bar\omega\{[F_2^{-1}\psi_l(g_{mr}\partial_{x_3}^2E_1+ h_{mr}\partial_{x_3}^2E_0)]\cr
&\quad(2^l\cdot,2^{2l}\cdot)\star (F_2^{-1}\varphi_k)(2^l\cdot,2^{2l}\cdot)(\bar\omega,x_3,z\}\cdot\cr
&\quad\cdot\bigg(\intop_{\R^3}d\bar x|(F_2^{-1}\varphi_lF_2d_{r\gamma})(\bar x-\bar\omega)|^p\bigg)^{1/p}\bigg|^p\bigg\}^{1/p}\equiv I_{2kmr\gamma}^1,\cr}
$$
where we also used the Minkowski inequality.

Applying the change of variables $\bar\zeta=\bar x-\bar\omega$ in the integral with respect to $\bar x$ gives
$$\eqal{
I_{2kmr\gamma}^1&=\bigg\{\intop_{\R_+}dx_3\intop_{\R_+}{dz\over z^{1+p\sigma}} \bigg|\sum_{l=0}^\infty\intop_{\R^3}d\bar\omega\{[F_2^{-1}\psi_l(g_{mr}\partial_{x_3}^2E_1\cr
&\quad+h_{mr}\partial_{x_3}^2E_0)](2^l\cdot,2^{2l}\cdot)\star (F_2^{-1}\varphi_k)(2^l\cdot,2^{2l}\cdot)\}(\bar\omega,x_3,z)\cdot\cr
&\quad\cdot\bigg(\intop_{\R^3}d\bar\zeta|(F_2^{-1}\varphi_lF_2 d_{r\gamma})(\bar\zeta)|^p\bigg)^{1/p}\bigg|^p\bigg\}^{1/p},\cr}
$$
where $\bar\zeta=(\zeta',\zeta_0)$, $\zeta'=(\zeta_1,\zeta_2)$.

Using the H\"older inequality in the integral with respect to $\bar\omega$ and replacing $\bar\omega$ by $\bar y$ and $\bar\zeta$ by $\bar x$ we get
$$\eqal{
I_{2kmr\gamma}^1&\le\bigg\{\intop_{\R_+}dx_3\intop_{\R_+}{dz\over z^{1+p\sigma}} \bigg|\sum_{l=0}^\infty\bigg(\intop_{\R^3}d\bar y{1\over(1+|\bar y|_a)^4}\bigg)^{1/2}\cdot\cr
&\quad\cdot\bigg(\intop_{\R^3}d\bar y|\{[F_2^{-1}\psi_l(g_{mr}\partial_{x_3}^2E_1+ h_{mr}\partial_{x_3}^2E_0)](2^l\cdot,2^{2l}\cdot)\star\cr
&\quad\star(F_2^{-1}\varphi_k)(2^l\cdot,2^{2l}\cdot)\}(\bar y,x_3,z)(1+|\bar y|_a)^2|^{1/2}|^p\bigg\}^{1/p}\cdot\cr
&\quad\cdot\|F_1^{-1}\varphi_lF_1d_{r\gamma}\|_{L_p(\R^3)}\equiv I_{2kmr\gamma}^2.\cr}
$$
Using that
$$
\bigg(\intop_{\R^3}d\bar y{1\over(1+|\bar y|_a)^4}\bigg)^{1/2}\le c
$$
and applying the Parseval identity, we obtain
$$\eqal{
I_{2kmr\gamma}^2&\le c\bigg(\intop_{\R_+}dx_3\intop_{\R_+}{dz\over z^{1+p\sigma}}\bigg|\sum_{l=0}^\infty\|[\psi_l(g_{mr}\partial_{x_3}^2E_1\cr
&\quad+h_{mr}\partial_{x_3}^2E_0)](2^l\cdot,2^{2l}\cdot,x_3,z)\varphi_k (2^l\cdot,2^{2l}\cdot)\|_{W_2^{4,2}(\R^3)}\bigg|^p\bigg)^{1/p}\cdot\cr
&\quad\cdot\|F_1^{-1}\varphi_lF_1d_{r\gamma}\|_{L_p(\R^3)}\equiv I_{2kmr\gamma}^3.\cr}
$$
Lemma \ref{l4.7} below implies
$$
I_{2kmr\gamma}^3\le c\sum_{l=0}^\infty 2^{l(2+\sigma-1/p-c_r)}2^{(8-L_1)|l-k|} \|F_1^{-1}\varphi_lF_1d_{r\gamma}\|_{L_p(\R^3)},
$$
where $c_r=1$ for $r=1,2$, $c_r=0$ for $r=3$.

Hence,
$$\eqal{
&\|u_m\|_{2,B_{p,q,\gamma}^{2+\sigma,1+\sigma/2}(\R_+^3\times\R)}\cr
&\le c\bigg(\sum_{k=0}^\infty\sum_{r=1}^3\sum_{l=0}^\infty 2^{l(2+\sigma-1/p-c_r)q}2^{(8+\varepsilon-L_1)|l-k|} \|F_1^{-1}\varphi_lF_1d_{r\gamma}\|_{L_p(\R^3)}^q\bigg)^{1/q},\cr}
$$
where $\varepsilon>0$ is arbitrary small. Let $L_1>8+\varepsilon$. Then (\ref{4.48}) holds. This ends the proof of Lemma \ref{l4.6}.
\end{proof}

Now, we prove

\begin{lemma}\label{l4.7}
The following inequality holds
\begin{equation}\eqal{
K_{mr}&\equiv\bigg(\intop_{\R_+^3}dx_3\intop_{\R_+}{dz\over z^{1+p\sigma}} \|[\psi_l(g_{mr}\partial_{x_3}^2E_1+h_{mr}\partial_{x_3}^2E_0)]\cr
&\quad(2^l\cdot,2^{2l}\cdot,x_3,z)\varphi_k (2^l\cdot,2^{2l}\cdot)\|_{W_2^{4,2}(\R^2\times\R)}^p\bigg)^{1/p}\cr
&\le c2^{l(2+\sigma-1/p-c_r)}2^{(8-L_1)|l-k|},\cr}
\label{4.50}
\end{equation}
where for $l=0$ the constant $c$ depends on $\gamma$, $L_1$ can be chosen sufficiently large.
\end{lemma}

\begin{proof}
We can write $K_{mr}$ in the form
$$\eqal{
K_{mr}&=\bigg(\intop_{\R_+}dx_3\sum_{\sum_{i=1}^4(s'_i+2s_i)\le 4} \|(\partial_{\xi'}^{s'_1}\partial_{\xi_0}^{s_1}\psi_l)(2^l\cdot,2^{2l}\cdot)\cdot\cr
&\quad\cdot(\partial_{\xi'}^{s'_2}\partial_{\xi_0}^{s_2}g_{mr}\partial_{\xi'}^{s'_3} \partial_{\xi_0}^{s_3}\partial_{x_3}^2E_1+\partial_{\xi'}^{s'_2} \partial_{\xi_0}^{s_2}h_{mr}\partial_{\xi'}^{s'_3}\partial_{\xi_0}^{s_3} \partial_{x_3}^2E_0)(2^;\cdot,x_3,2^{2l}\cdot)\cdot\cr
&\quad\cdot(\partial_{\xi'}^{s'_4}\partial_{\xi_0}^{s_4})\varphi_k (2^l\cdot,2^{2l}\cdot)\|_{L_2(\R^3)}^p\bigg)^{1/p}.\cr}
$$
Using the notation
$$
\partial_\xi^{\bar s_i}=\partial_{\xi'}^{s'_i}\partial_{\xi_0}^{s_i},\quad
\bar s_i=|s'_i|+2s_i,\quad i=1,2,3,4
$$
and the Minkowski inequality we obtain
$$\eqal{
K_{mr}&\le c\sum_{\sum_{i=1}^4\bar s_i\le 4}\bigg[\bigg\|(\partial_\xi^{\bar s_1}\psi_l)(2^l\cdot,2^{2l})(\partial_\xi^{\bar s_2}g_{mr})(2^l\cdot,2^{2l}\cdot)\cr
&\quad\bigg\|{\partial_\xi^{\bar s_3}\partial_{x_3}^2E_1\over z^{1/p+\sigma}}\bigg\|_{L_p(\R_+^2)}(2^l\cdot,2^{2l}\cdot)(\partial_\xi^{\bar s_4}\varphi_k)(2^l\cdot,2^{2l}\cdot)\bigg\|_{L_2(\R^3)}\cr
&\quad+\bigg\|(\partial_\xi^{\bar s_1}\psi_l)(2^l\cdot,2^{2l}\cdot)(\partial_\xi^{\bar s_2}h_{mr})(2^l\cdot,2^{2l}\cdot)\cdot\cr
&\quad\cdot\bigg\|{\partial_\xi^{\bar s_3}\partial_{x_3}^2E_0\over z^{1/p+\sigma}}\bigg\|_{L_p(\R_+^2)}(2^l\cdot,2^{2l}\cdot)(\partial_\xi^{\bar s_4}\varphi_k)(2^l\cdot,2^{2l}\cdot)\bigg\|_{L_2(\R^3)}.\cr}
$$
Using that
$$\eqal{
&|\partial_\xi^{\bar s_1}\psi_l(2^l\cdot,2^{2l}\cdot)|\le c\quad &{\rm for}\ \ \bar\xi\in A,\cr
&\bigg\|{\partial_\xi^{\bar s_2}\partial_{x_3}^2E_1\over z^{1/p+\sigma}}\bigg\|_{L_p(\R_+^2)}\le c2^{l(\sigma-1/p-1)}\quad &{\rm for}\ \ \bar\xi\in A,\ \ l\not=0,\cr
&\bigg\|{\partial_\xi^{\bar s_2}\partial_{x_3}^2E_0\over z^{1/p+\sigma}}\bigg\|_{L_p(\R_+^2)}\le c2^{l(\sigma-1/p)}\quad &{\rm for}\ \ \bar\xi\in A,\ \ l\not=0,\cr
&\sum_{i=0}^1\bigg\|{\partial_\xi^{\bar s_2}\partial_{x_3}^2E_i\over z^{1/p+\sigma}}\bigg\|_{L_p(\R_+^2)}\le c(\gamma)\quad {\rm for}\ \ \bar\xi\in A,\ \ l=0,\cr}
$$
where $A$ is defined in Lemma \ref{l4.5}.

By the above estimates, inequality (\ref{4.50}) follows. This ends the proof.
\end{proof}

The above lemmas imply Theorem \ref{t4.3}.

In view of Theorem \ref{t4.3} and properties of the transformation from problem (\ref{4.1}) to (\ref{4.17}) we have.

\begin{theorem}\label{t4.8}
Let $p\in(1,\infty)$, $\sigma\in(0,1)$. Assume that\\ $f\in B_{p,p}^{\sigma,\sigma/2}(\R_+^3\times\R_+)$, $b_\alpha\in B_{p,p}^{1+\sigma-1/p,(1+\sigma-1/p)/2}(\R^2\times\R_+)$, $\alpha=1,2$,\break
$b_3\in B_{p,p}^{2+\sigma-1/p,1+\sigma/2-1/2p}(\R^2\times\R_+)$, $v_0\in B_{p,p}^{2+\sigma-2/p}(\R_+^3)$.\\
Then there exists a solution to problem (\ref{4.1}) such that\\
$v\in B_{p,p}^{2+\sigma,1+\sigma/2}(\R_+^3\times\R_+)$, $\nabla p\in B_{p,p}^{\sigma,\sigma/2}(\R_+^3\times\R_+)$ and
\begin{equation}\eqal{
&\|v\|_{B_{p,p}^{2+\sigma,1+\sigma/2}(\R_+^3\times\R_+)}+\|\nabla p\|_{B_{p,p}^{\sigma,\sigma/2}(\R_+^3\times\R_+)}\cr
&\le\bigg(c(\|f\|_{B_{p,p}^{\sigma,\sigma/2}(\R_+^3\times\R_+)}+ \|g\|_{B_{p,p}^{1+\sigma,1/2+\sigma/2}(\R_+^3\times\R_+)}\cr
&\quad+\sum_{\alpha=1}^2 \|b_\alpha\|_{B_{p,p}^{1+\sigma-1/p,1/2+\sigma/2-1/2p}(\R^2\times\R_+)}\cr
&\quad+\|b_3\|_{B_{p,p}^{2+\sigma-1/p,1+\sigma/2-1/2p}(\R_+^3\times\R_+)}+ \|v_0\|_{B_{p,p}^{2+\sigma-2/p}(\R_+^3)}\bigg)\cr}
\label{4.51}
\end{equation}
where $c$ does not depend on $v$ neither on $p$.
\end{theorem}

\subsection{The Stokes system in neighborhoods of edges.}

Finally, we want to solve the Stokes system in a neighborhood $\Omega(\xi)$ of a point $\xi\in L_i$, $i=1,2$. Along the edge $L_i$, $i=1,2$, $S_2$ meets $S_1$ under angle $\pi/2$. Let $\zeta$ be a smooth function from the partition of unity such that
$$
\supp\zeta=\Omega(\xi).
$$
Introduce a local system of coordinates $(x_1,x_2,x_3)$ with origin at $\xi$ such that $S_2\cap\Omega(\xi)$ is described by $x_3=0$. Next we transform $S_1\cap\Omega(\xi)$ to the plane $x_1=0$ by making an appropriate extension. Then $L_i$, $i=1,2$, becomes a straight line $x_1=0$, $x_3=0$ so it is the $x_2$-axis.

Therefore, the transformed Stokes system takes the form
\begin{equation}\eqal{
&v_t-\nu\Delta v+\nabla p=f,\cr
&\divv v=0\cr}
\label{4.52}
\end{equation}
in the dihedral angle $\pi/2$ located between planes $x_1=0$ and $x_3=0$, denoted by ${\cal D}_{\pi/2}\times\R_+$. On the plane $x_3=0$ we have the boundary conditions
\begin{equation}\eqal{
&v_3=b'_3,\cr
&v_{3,x_1}+v_{1,x_3}=b'_1,\cr
&v_{3,x_2}+v_{2,x_3}=b'_2\cr}
\label{4.53}
\end{equation}
and on $x_1=0$ we have
\begin{equation}\eqal{
&v_1=b''_1,\cr
&v_{1,x_2}+v_{2,x_1}=b''_2,\cr
&v_{1,x_3}+v_{3,x_1}=b''_3.\cr}
\label{4.54}
\end{equation}
Boundary conditions (\ref{4.54}) can be expressed in the form
\begin{equation}\eqal{
&v_1=b''_1,\cr
&v_{2,z_1}=b''_2-b''_{1,x_2},\cr
&v_{3,x_1}=b''_3-b''_{1,x_3}.\cr}
\label{4.55}
\end{equation}
Hence $(\ref{4.55})_1$ is the Dirichlet boundary condition and $(\ref{4.55})_{2,3}$ are the Neumann boundary conditions.

We transform solutions to problem (\ref{4.52}), (\ref{4.53}), (\ref{4.54}) in such a way that the boundary conditions (\ref{4.55}) become homogeneous. Then the solutions are extended by reflection on $x_1<0$. Thus we obtain problem (\ref{4.52}), (\ref{4.53}) in the half space $x_3>0$.

Then Theorem \ref{t4.3} is also valid in this case.

\section{The Stokes system in the cylindrical\\ domain $\Omega$}\label{se5}

We consider the following initial-boundary value problem for the Stokes system
\begin{equation}\eqal{
&v_t-\nu\Delta v+\nabla p=f_0\quad &{\rm in}\ \ \Omega\times(0,\tau),\cr
&\divv v=g_0\quad &{\rm in}\ \ \Omega\times(0,\tau),\cr
&\bar n\cdot\D(v)\cdot\bar\tau_\alpha=b_{0\alpha},\ \ \alpha=1,2,\quad &{\rm on}\ \ S\times(0,\tau),\cr
&v\cdot\bar n=b_{03}\quad &{\rm on}\ \ S\times(0,\tau),\cr
&v|_{t=0}=v_0\quad &{\rm in}\ \ \Omega,\cr}
\label{5.1}
\end{equation}
where $S=S_1\cup S_2$, $\bar n$ is the unit outward normal vector to $S$ and $\bar\tau_1$, $\bar\tau_2$ are tangent vectors to $S$.

Let $\tilde v_0$ be the time extension of $v_0$ such that
\begin{equation}
\tilde v_0|_{t=0}=v_0
\label{5.2}
\end{equation}
In view of (\ref{4.5}) we can introduce the new function
\begin{equation}
u=v-\tilde v_0
\label{5.3}
\end{equation}
and $(u,p)$ is a solution to the following initial-boundary value problem with vanishing initial data
\begin{equation}\eqal{
&u_t-\nu\Delta u+\nabla p=f_0-\tilde v_{0t}+\nu\Delta\tilde v_0\equiv f\quad &{\rm in}\ \ \Omega\times(0,\tau)\equiv\Omega^\tau,\cr
&\divv u=g_0-\divv\tilde v_0\equiv g\quad &{\rm in}\ \ \Omega\times(0,\tau)\equiv\Omega^\tau,\cr
&B_\alpha(u)\equiv\bar n\cdot\D(u)\cdot\bar\tau_\alpha=b_{0\alpha}-\bar n\cdot\D(\tilde v_0)\cdot\bar\tau_\alpha\equiv b_\alpha,\ \ &\alpha=1,2,\cr
&B_3(u)\equiv u\cdot\bar n=b_{03}-\tilde v\cdot\bar n\equiv b_3\quad &{\rm on}\ \ S\times(0,\tau)\equiv S^\tau,\cr}
\label{5.4}
\end{equation}
where $b_\alpha(u)$, $\alpha=1,2$, is the Neumann boundary condition and $B_3(u)$ the Dirichlet condition.

Using the technique of regularizer we prove

\begin{lemma}\label{l5.1}\ \\
Let $p\in(1,\infty)$, $\sigma\not\in\N$. Let $f\in W_p^{\sigma,\sigma/2}(\Omega^\tau)$, $g\in W_p^{1+\sigma,1/2+\sigma/2}(\Omega^\tau)$, $b_\alpha\in W_p^{1+\sigma-1/p,1/2+\sigma/2-1/2p}(S^\tau)$, $\alpha=1,2$, $b_3\in W_p^{2+\sigma-1/p,1/2+\sigma/2-1/2p}(S^\tau)$.\\
Then there exists a solution to problem (\ref{5.4}) such that $u\in W_p^{2+\sigma,1+\sigma/2}(\Omega^\tau)$, $\nabla p\in W_p^{\sigma,\sigma/2}(\Omega^\tau)$, where $\tau$ is sufficiently small, and there exists a constant $c$ independent of $u$, $p$ such that
\begin{equation}\eqal{
&\|u\|_{W_p^{2+\sigma,1+\sigma/2}(\Omega^\tau)}+\|\nabla p\|_{W_p^{\sigma,\sigma/2}(\Omega^\tau)}\cr
&\le c\bigg(\|f\|_{W_p^{\sigma,\sigma/2}(\Omega^\tau)}+ \|g\|_{W_p^{1+\sigma,1/2+\sigma/2}(\Omega^\tau)}+\sum_{\alpha=1}^2\|b_\alpha\|_{W_p^{1+\sigma-1/p,1/2+\sigma/2-1/2p}(S^\tau)}\cr
&\quad+\|b_3\|_{W_p^{2+\sigma-1/p,1+\sigma/2-1/2p}(S^\tau)}\bigg).\cr}
\label{5.5}
\end{equation}
\end{lemma}

\begin{proof}
To prove the lemma we use the partition of unity introduced in Definition~\ref{d2.17}. We introduce the simplified notation
$$\eqal{
&L(\partial_x,\partial_t)=\left(\matrix{\partial_t-\nu\Delta,&\nabla\cr \divv&0\cr}\right),\cr
&B(\partial_x)=\left(\matrix{\bar n\cdot\D(\cdot)\bar\tau_1\cr \bar n\cdot\D(\cdot)\bar\tau_2\cr \bar n\cdot\cr}\right),\cr}
$$
where
$$
L(\partial_x,\partial_t)(u,p)=[L(\partial_x,\partial_t){u\choose p}]^T=\left( \matrix{\partial_tu-\nu\Delta u+\nabla p\cr \divv u\cr}\right)^T
$$
and
$$
B(\partial_x)|_{x_3=0}\equiv[B(\partial_x)|_{x_3=0}u]^T=\left(
\matrix{\bar e_3\cdot\D(u)\cdot\bar e_1\cr \bar e_3\cdot\D(u)\cdot\bar e_2\cr \bar e_3\cdot u\cr}\right)^T,
$$
where $\bar e_1=(1,0,0)$, $\bar e_2=(0,1,0)$, $\bar e_3=(0,0,1)$.
\goodbreak
Let $k\in{\cal M}_1$ and $f^{(k)}(x,t)=\zeta^{(k)}(x)f(x,t)$, $g^{(k)}(x,t)=\zeta^{(k)}g(x,t)$.

Let $f^{(k)}\in W_p^{\sigma,\sigma/2}(\R^3\times\R_+)$, $g^{(k)}\in W_p^{1+\sigma,1/2+\sigma/2}(\R^3\times\R_+)$. By $\R^{(k)}$, $k\in{\cal M}$, we denote the operator which solves the Cauchy problem with vanishing initial data
$$
L(\partial_x,\partial_t)(u^{(k)},p^{(k)})=(f^{(k)}(x,t),g^{(k)}(x,t)).
$$
In view of Theorem \ref{t3.3} there exists an operator $R^{(k)}$  such that\break $(u^{(k)},p^{(k)})=R^{(k)}(f^{(k)},g^{(k)})$ and
\begin{equation}\eqal{
&\|R^{(k)}(f^{(k)},g^{(k)})\|_{W_p^{2+\sigma,1+\sigma/2}(\R^3\times(0,\tau))\times W_p^{\sigma,\sigma/2}(\R^3\times(0,\tau))}\cr
&\le c(\|f^{(k)}\|_{W_p^{\sigma,\sigma/2}(\R^3\times(0,\tau))}+ \|g^{(k)}\|_{W_p^{1+\sigma,1/2+\sigma/2}(\R^3\times(0,\tau))}).\cr}
\label{5.6}
\end{equation}
Let $k\in{\cal N}_1$. Then $\supp\zeta^{(k)}$ is a neighborhood of a point $\xi\in S_1$ located at a positive distance from edges $L_1$, $L_2$.

Then after a transformation $x=x(z)$ which makes flat the part of $S_1$ equal to $S_1\cap\supp\zeta^{(k)},$ we consider the problem
\begin{equation}\eqal{
&L(\partial_z,\partial_t)(u^{(k)},p^{(k)})=(f^{(k)}(z,t),g^{(k)}(z,t)),\cr
&B(\partial_z)|_{z_3=0}u^{(k)}=b^{(k)}(z,t),\ \ z_3=0.\cr}
\label{5.7}
\end{equation}
Let $R^{(k)}$, $k\in{\cal N}_1$, present a solution to problem (\ref{5.7}) by
$$
(u^{(k)},p^{(k)})=R^{(k)}(f^{(k)},g^{(k)},b^{(k)}),\ \ k\in{\cal N}_1
$$
and
\begin{equation}\eqal{
&\|R^{(k)}(f^{(k)},g^{(k)})\|_{W_p^{2+\sigma,1+\sigma/2}(\R^3\times(0,\tau))\times W_p^{\sigma,\sigma/2}(\R^3\times(0,\tau)}\le c\bigg(\|f^{(k)}\|_{W_p^{\sigma,\sigma/2}(\R^3\times(0,\tau))}\cr
&\quad+\|g^{(k)}\|_{W_p^{1+\sigma,1/2+\sigma/2}(\R^3\times(0,\tau))}+ \sum_{\alpha=1}^2 \|b_\alpha^{(k)}\|_{W_p^{1+\sigma-1/p,1/2+\sigma/2-1/2p}(\R^2\times(0,\tau))}\cr
&\quad+\|b_3^{(k)}\|_{W_p^{2+\sigma-1/p,1+\sigma/2-1/2p}(\R^2\times(0,\tau))} \bigg).\cr}
\label{5.8}
\end{equation}
Let $k\in{\cal N}_2$. Then $\supp\zeta^{(k)}$ is a neighborhood of a point $\xi\in S_2$ located at a positive distance from edges $L_\alpha$, $\alpha=1,2$. Since $S_2$ is flat we do not need to pass to variables $z$. Therefore the considered problem in $\supp\zeta^{(k)}$, $k\in{\cal N}_2$, can be formulated in the original coordinates $x$. Hence it has the form
\begin{equation}\eqal{
&L(\partial_x,\partial_t)(u^{(k)},p^{(k)})=(f^{(k)}(x,t),g^{(k)}(x,t)),\cr
&B(\partial_x)|_{x_3=0}u=b^{(k)}(x',t),\cr}
\label{5.9}
\end{equation}
where $x'=(x_1,x_2)$.

Let $R^{(k)}$, $k\in{\cal N}_2$, present a solution to problem (\ref{5.9}). This has the form
$$
(u^{(k)},p^{(k)})=R^{(k)}(f^{(k)},g^{(k)},b^{(k)}),\quad k\in{\cal N}_2
$$
and
\begin{equation}\eqal{
&\|u^{(k)}\|_{W_p^{2+\sigma,1+\sigma/2}(\R_+^3\times(0,\tau))}+\|\nabla p^{(k)}\|_{W_p^{\sigma,\sigma/2}(\R_+^3\times(0,\tau))}\cr
&=\|R^{(k)}(f^{(k)},g^{(k)},b^{(k)})\|_{W_p^{2+\sigma,1+\sigma/2}(\R_+^3\times(0,\tau)) \times W_p^{\sigma,\sigma/2}(\R_+^3\times(0,\tau))}\cr
&\le c\bigg(\|f^{(k)}\|_{W_p^{\sigma,\sigma/2}(\R_+^3\times(0,\tau))}+ \|g^{(k)}\|_{W_p^{1+\sigma,1/2+\sigma/2}(\R_+^3\times(0,\tau))}\cr
&\quad+\sum_{\alpha=1}^2 \|b_\alpha^{(k)}\|_{W_p^{1+\sigma-1/p,1/2+\sigma/2-1/2p}(\R^2\times(0,\tau))}\cr
&\quad+\|b_3^{(k)}\|_{W_p^{2+\sigma-1/p,1+\sigma/2-1/2p}(\R^2\times(0,\tau))} \bigg).\cr}
\label{5.10}
\end{equation}
For $k\in{\cal N}_3$, $\supp\zeta^{(k)}$ is a neighborhood of a point of $L_\alpha$, $\alpha\in\{1,2\}$. In order to consider the Stokes system in this neighborhood we have to use a transformation which makes $S_1\cap\supp\zeta^{(k)}$ flat. Then $L_\alpha$, $\alpha\in\{1,2\}$, becomes a straight line. Therefore, the considered domain becomes the dihedral angle of the magnitude $\pi/2$, where the one side is the plane $x_3=0$ derived by extension of $S_2\cap\supp\zeta^{(k)}$ and the other side is the plane $x_2=0$ which is an extension of the flatten part of $S_1\cap\supp\zeta^{(k)}$.

Now we formulate the problem in the case $k\in{\cal N}_3$. On the plane $x_3=0$ we have the boundary conditions
\begin{equation}\eqal{
&u_{3,x_\alpha}^{(k)}+u_{\alpha,x_3}^{(k)}=b_\alpha^{(k)},\ \ \alpha=1,2,\cr
&u_3^{(k)}=b_3^{(k)}\cr}
\label{5.11}
\end{equation}
and on the plane $x_2=0$ we have
\begin{equation}\eqal{
&u_{2,x_\alpha}^{(k)}+u_{\alpha,x_2}^{(k)}=\bar b_\alpha^{(k)},\ \ \alpha=1,2,\cr
&u_2^{(k)}=0,\cr}
\label{5.12}
\end{equation}
where $\bar b_\alpha^{(k)}$, $\alpha=1,2$, are derived from $b_\alpha$, $\alpha=1,2$, from (\ref{5.4}) restricted to $S_1\cap\supp\zeta^{(k)}$ and transformed by transformation which makes $S_1\cap\supp\zeta^{(k)}$ flat.

We simplify boundary conditions (\ref{5.12}) to the form
\begin{equation}\eqal{
&u_{\alpha,x_2}^{(k)}=\bar b_\alpha^{(k)},\ \ \alpha=1,3,\quad &{\rm on}\ \ x_2=0,\cr
&u_2^{(k)}=0\quad &{\rm on}\ \ x_2=0.\cr}
\label{5.13}
\end{equation}
Now, we perform a transformation which makes the Neumann boundary conditions homogeneous. Then we can extend $u^{(k)}$, $p^{(k)}$ by the reflection with respect to the plane $x_2=0$. Formally, the problem has the same form as problem (\ref{5.9}) for $k\in{\cal N}_2$.

Let $R^{(k)}$ solves the problem. Then we have
$$
(u^{(k)},\nabla p^{(k)})=R^{(k)}(f^{(k)},g^{(k)},b^{(k)}|_{S_1},b^{(k)}|_{S_2})
$$
and the estimate
\begin{equation}\eqal{
&\|u^{(k)}\|_{W_p^{2+\sigma,1+\sigma/2}(\R_+^3\times(0,\tau))}+\|\nabla p^{(k)}\|_{W_p^{\sigma,\sigma/2}(\R_+^3\times(0,\tau))}\cr
&\le c\bigg(\|f^{(k)}\|_{W_p^{\sigma,\sigma/2}(\R_+^3\times(0,\tau))}+ \|g^{(k)}\|_{W_p^{1+\sigma,1/2+\sigma/2}(\R_+^3\times(0,\tau))}\cr
&\quad+\sum_{\beta=1}^2\sum_{\alpha=1}^2 \|b_\alpha^{(k)}|_{S_\beta}\|_{W_p^{1+\sigma-1/p,1/2+\sigma/2-1/2p}(\R^2\times(0,\tau))}\cr
&\quad+\|b_3^{(k)}|_{S_2}\|_{W_p^{2+\sigma-1/p,1+\sigma/2-1/2p}(\R^2\times(0,\tau))} \bigg).\cr}
\label{5.14}
\end{equation}
Now, we construct an operator $R$ called the regularizer.

Let
$$\eqal{
&h^{(k)}=(f^{(k)},g^{(k)})\quad &{\rm for}\ \ k\in{\cal M},\cr
&h^{(k)}=(f^{(k)},g^{(k)},\nad{b}1^{(k)})\quad &{\rm for}\ \ k\in{\cal N}_1,\cr
&h^{(k)}=(f^{(k)},g^{(k)},\nad{b}2^{(k)})\quad &{\rm for}\ \ k\in{\cal N}_2,\cr
&h^{(k)}=(f^{(k)},g^{(k)},\nad{b}1^{(k)},\nad{b}2^{(k)})\quad &{\rm for}\ \ k\in{\cal N}_3,\cr}
$$
where $\nad{b}1^{(k)}$ is defined on $S_1\cap\supp\zeta^{(k)}\times(0,\tau)$ and $\nad{b}2^{(k)}$ on $S_2\cap\supp\zeta^{(k)}\times(0,\tau)$.

Next, we introduce
$$\eqal{
\|h\|_{S_p^\sigma}&=\sum_{k\in{\cal M}\cup{\cal N}_1\cup{\cal N}_2\cup{\cal N}_3}(\|f^{(k)}\|_{W_p^{\sigma,\sigma/2}(\R_+^3\times(0,\tau))}+\|g^{(k)}\|_{W_p^{1+\sigma,1/2+\sigma/2}(\R_+^3\times(0,\tau))}\cr
&\quad+\sum_{k\in{\cal N}_1\cup{\cal N}_3}\bigg(\sum_{\alpha=1}^2 \|\nad{b}1_\alpha^{(k)}\|_{W_p^{1+\sigma-1/p,1/2+\sigma/2-1/2p}(\R^2\times(0,\tau))}\cr
&\quad+\|\nad{b}1_3^{(k)}\|_{W_p^{2+\sigma-1/p,1+\sigma/2-1/2p}(\R^2\times(0,\tau))} \bigg)\cr
&\quad+\sum_{k\in{\cal N}_2\cup{\cal N}_3}\bigg(\sum_{\alpha=1}^2 \|\nad{b}2_\alpha^{(k)}\|_{W_p^{1+\sigma-1/p,1/2+\sigma/2-1/2p}(\R^2\times(0,\tau))}\cr
&\quad+\|\nad{b}2_3^{(k)}\|_{W_p^{2+\sigma-1/p,1+\sigma/2-1/2p}(\R^2\times(0,\tau))} \bigg).\cr}
$$
Then we define an operator $R$ by
$$
Rh=\sum_k\eta^{(k)}(u^{(k)},\nabla p^{(k)})\equiv\sum_k\eta^{(k)}w^{(k)},
$$
where
$$
w^{(k)}=(u^{(k)},\nabla p^{(k)}=\left\{\eqal{
&R^{(k)}(f^{(k)},g^{(k)})\quad &k\in{\cal M},\cr
&Z_kR^{(k)}(Z_k^{-1}f^{(k)},Z_k^{-1}g^{(k)},Z_k^{-1}\nad{b}1^{(k)})\quad &k\in{\cal N}_1,\cr
&R^{(k)}(f^{(k)},g^{(k)},\nad{b}2^{(k)})\quad &k\in{\cal N}_2,\cr
&R^{(k)}(f^{(k)},g^{(k)},\nad{b}1^{(k)},\nad{b}2^{(k)})\quad &k\in{\cal N}_3,\cr}\right.
$$
where $Z_k$ is a map from coordinates $z$, which makes $S_1\cap\supp\zeta^{(k)}$ flat, to coordinates $x$ and
$$\eqal{
\|Rh\|_{{\cal B}_p^{\sigma+2}}&=\sum_k(\|u^{(k)}\|_{W_p^{2+\sigma,1+\sigma/2}(\R_+^3\times(0,\tau))}\cr
&\quad+\|\nabla p^{(k)}\|_{W_p^{\sigma,\sigma/2}(\R_+^3\times(0,\tau))}).\cr}
$$

\begin{lemma}\label{l5.2}
The operator $R\colon S_p^\sigma\to{\cal B}_p^{\sigma+2}$ is a bounded operator and
\begin{equation}
\|Rh\|_{{\cal B}_p^{\sigma+2}}\le c\|h\|_{{\cal S}_p^\sigma}.
\label{5.15}
\end{equation}
Let $w=(u,\nabla p)$ and let the considered problem be denoted by $A$. Then our aim is to show the existence of operators $T$ and $W$ such that
$$\eqal{
&ARh=h+Th,\cr
&RAw=w+Ww.\cr}
$$
Hence, to show the existence of solutions to problem (\ref{5.4}) we have to prove that $\|T\|<1$, $\|W\|\le 1$.
\end{lemma}

This will be shown for sufficiently small $\lambda$ and $\tau$.

Now we construct operator $T$. The operator can be divided into two parts: $T=(T_1,T_2)$, where
$$\eqal{
&LRh=h+T_1h,\cr
&BRh=h+T_2h.\cr}
$$
First we describe $T_1$,
$$\eqal{
T_1h&=\sum_{k\in{\cal M}}(L\eta^{(k)}w^{(k)}-\eta^{(k)}Lw^{(k)})+\sum_{k\in{\cal N}_1}[(L\eta^{(k)}w^{(k)}-\eta^{(k)}Lw^{(k)})\cr
&\quad+\eta^{(k)}Z_k(L(\partial_z-\nabla F_k\partial_{z_3},\partial_t)-L(\partial_z,\partial_t))Z_k^{-1}R^{(k)}h^{(k)}]\cr
&\quad+\sum_{k\in{\cal N}_2}(L\eta^{(k)}w^{(k)}-\eta^{(k)}Lw^{(k)})\cr
&\quad+\sum_{k\in{\cal N}_3}(L\circ\Phi_1^{(k)}\eta^{(k)}w^{(k)}-\eta^{(k)}L\circ \Phi_1^{(k)}w^{(k)}),\cr}
$$
where $\Phi_1^{(k)}$ is a map which is a composition of the following transformations:
\begin{itemize}
\item[1.] A neighborhood $\supp\zeta^{(k)}$ of a point $\xi^{(k)}\in L_i$, $i=1,2$, is transformed to the $\pi/2$-dihedral angle with edge $\bar L_i$ and sides $\bar S_1^{(k)}$ and $\bar S_2^{(k)}$, where $\bar L_i$ is a straight line passing through $\xi^{(k)}$, $\bar S_i^{(k)}$, $i=1,2$, are planes derived by extension of $S_2\cap\supp\zeta^{(k)}$ and $S_1\cap\supp\zeta^{(k)}$. We have to emphasize that $S_1\cap\supp\zeta^{(k)}$ can be transformed locally to a plane.
\item[2.] Nonhomogeneous Neumann and Dirichlet problems on $\bar S_1^{(k)}$ are transformed to homogeneous.
\item[3.] Finally, the considered localized problem on $\supp\zeta^{(k)}$ is reflected with respect to plane $\bar S_1^{(k)}$.
\end{itemize}

Next we construct $T_2$,
$$\eqal{
T_2h&=\sum_{k\in{\cal N}_1}\bigg\{\sum_{\alpha=1}^3[B_\alpha\eta^{(k)}w^{(k)}- \eta^{(k)}B_\alpha w^{(k)})\cr
&\quad+\eta^{(k)}(B_\alpha(x,t,\partial_x)-B_\alpha(\xi^{(k)},0,\partial_x)) w^{(k)}|_{S_1^{(k)}}]\cr
&\quad+\sum_{\alpha=1}^2\eta^{(k)}Z_k(B_\alpha(\xi^{(k)},0,\partial_z-\nabla F_k\partial_{z_3})-B_\alpha(\xi^{(k)},0,\partial_z))R^{(k)}h^{(k)}\bigg\}\cr
&\quad+\sum_{k\in{\cal N}_2}\sum_{\alpha=1}^3[(B_\alpha\eta^{(k)}w^{(k)}- \eta^{(k)}B_\alpha w^{(k)})\cr
&\quad+\eta^{(k)}(B_\alpha(x,t,\partial_x)-B_\alpha(\xi^{(k)},0,\partial_x)) w^{(k)}]|_{S_2^{(k)}}\cr
&\quad+\sum_{k\in{\cal N}_3}\sum_{\alpha=1}^3(B_\alpha\circ\Phi_2^{(k)}\eta^{(k)} w^{(k)}-\eta^{(k)}B_\alpha\circ\Phi_2^{(k)}w^{(k)}),\cr}
$$
where $\Phi_2^{(k)}$ is equal to map $\Phi_1^{(k)}$ restricted to the boundary $S$.

For $\tau$ and $\lambda$ sufficiently small the norm $\|T\|_{S_p^\sigma\to S_p^\sigma}$ is less than 1.

Next we construct an operator $W$
$$\eqal{
Ww&=\sum_{k\in{\cal M}}\eta^{(k)}R^{(k)}(\zeta^{(k)}Lw-L\zeta^{(k)}w)\cr
&\quad+\sum_{k\in{\cal N}_1}\sum_{\alpha=1}^3\eta^{(k)}Z_kR^{(k)}[Z_k^{-1} (\zeta^{(k)}Lw-L\zeta^{(k)}w),\cr
&\quad Z_k^{-1}(\zeta^{(k)}B_\alpha w-B_\alpha\zeta^{(k)}w)|_{S_1^{(k)}}]\cr
&\quad+\sum_{k\in{\cal N}_1}\sum_{\alpha=1}^3\eta^{(k)}Z_kR^{(k)}[Z_k^{-1} (L(x,\partial_x,\partial_t)-L(\xi^{(k)},\partial_x,\partial_t))\zeta^{(k)}w,\cr
&\quad Z_k^{-1}(B_\alpha(x,\partial_x)-B_\alpha(\xi^{(k)},\partial_x)) \zeta^{(k)}w|_{S_1^{(k)}}]\cr
&\quad+\sum_{k\in{\cal N}_1}\sum_{\alpha=1}^3\eta^{(k)}Z_kR^{(k)}[(L(\xi^{(k)}, \partial_z-\nabla F_k\partial_{z_3})-L(\xi^{(k)},\partial_z))Z_k^{-1}\zeta^{(k)}w,\cr
&\quad (B_\alpha(\xi^{(k)},\partial_z-\nabla F_k\partial_{z_3})-B_\alpha(\xi^{(k)},\partial_z))Z_k^{-1}\zeta^{(k)}w|_{S_1^{(k)}}]\cr
&\quad+\sum_{k\in{\cal N}_2}\sum_{\alpha=1}^3\eta^{(k)}R^{(k)}[(\zeta^{(k)}Lw-L\zeta^{(k)}w),\cr
&\quad (\zeta^{(k)}B_\alpha w-B_\alpha\zeta^{(k)}w)|_{S_2^{(k)}}]\cr
&\quad+\sum_{k\in{\cal N}_2}\sum_{\alpha=1}^3\eta^{(k)}R^{(k)}[(L(x,t,\partial_x)-L(\xi^{(k)},0,\partial_x)) \zeta^{(k)}w,\cr
&\quad (B_\alpha(x,\partial_x)-B_\alpha(\xi^{(k)},\partial_x))\zeta^{(k)}w|_{S_2^{(k)}}]\cr
&\quad+\sum_{k\in{\cal N}_3}\sum_{\alpha=1}^3\eta^{(k)}R^{(k)} [(\zeta^{(k)}L\circ\Phi_1^{(k)}w-L\circ\Phi_1^{(k)}\zeta^{(k)}w),\cr
&\quad(\zeta^{(k)}B_\alpha\circ\Phi_2^{(k)}w-B_\alpha\circ\Phi_2^{(k)} \zeta^{(k)}w)|_{\bar S_1^{(k)}\cup\bar S_2^{(k)}}],\cr}
$$
where transformations $\Phi_1^{(k)}$ and $\Phi_2^{(k)}$ are defined above.

For $\tau$ and $\lambda$ sufficiently small the norm
$$
\|W\|_{{\cal B}_p^{\sigma+2}\to{\cal B}_p^{\sigma+2}}
$$
is less than 1. This ends the proof of Lemma \ref{l5.1}.
\end{proof}

Repeating the proof of Lemma \ref{l5.1} we have

\begin{lemma}\label{l5.3}
Let $p\in(1,\infty)$, $\sigma\in\N_0\equiv\N\cup\{0\}$. Let the assumptions of Lemma \ref{l5.1} hold. Then there exists a solution to problem (\ref{5.4}) such that $u\in W_p^{2+\sigma,1+\sigma/2}(\Omega^\tau)$, $\nabla p\in W_p^{\sigma,\sigma/2}(\Omega^\tau)$, where $\tau$ is sufficiently small and estimate (\ref{5.5}) holds.
\end{lemma}

\section{Existence in Sobolev spaces}\label{se6}

\begin{theorem}\label{t6.1} \null\\
Assume that $f\in L_r(\Omega^T)$, $g\in W_r^{1,1/2}(\Omega^T)$, $b_\alpha\in W_r^{1-1/r,1/2-1/2r}(S^T)$,\break $\alpha=1,2$, $b_3\in W_r^{2-1/r,1-1/2r}(S^T)$, $v_0\in W_r^{2-2/r}(\Omega)$, $r\in(1,\infty)$, $S=S_1\cup S_2$, $S_1\in C^2$, $L\in C^2$. Then there exists a solution to problem (\ref{1.1}) such that
$$
v\in W_r^{2,1}(\Omega^T),\quad \nabla p\in L_r(\Omega^T)
$$
and
\begin{equation}\eqal{
&\|v\|_{W_r^{2,1}(\Omega^T)}+\|\nabla p\|_{L_r(\Omega^T)}\le c\bigg(\|f\|_{L_r(\Omega^T)}\cr
&\quad+\|g\|_{W_r^{1,1/2}(\Omega^T)}+\sum_{\alpha=1}^2 \|b_\alpha\|_{W_r^{1-1/r,1/2-1/2r}(S^T)}\cr
&\quad+\|b_3\|_{W_r^{2-1/r,1-1/2r}(S^T)}+\|v_0\|_{W_r^{2-2/r}(\Omega)}\bigg).\cr}
\label{6.1}
\end{equation}
\end{theorem}

\begin{proof}
Using the partition of unity we consider localized problems from (\ref{1.1}) near an interior point of $\Omega$, near $S_1$, $S_2$ and near edges $L_1$, $L_2$. Existence of solutions of these problems and estimates in $W_r^{2,1}\times L_r$ can be shown. Finally, the technique of regularizer (see Section \ref{se5}) ends the proof of the theorem.
\end{proof}

\


\begin{thebibliography}{XXXX}

\bibitem[A]{A} Adams, R.A.: Sobolev spaces, Acad. Press, New York, 1975.

\bibitem[Al]{Al} Alame, W.: On existence of solutions for the nonstationary Stokes system with boundary slip conditions, Appl. Math. 32 (2) (2005),195--223.

\bibitem[Am]{Am} Amann, H.: Operator-valued Fourier Multipliers, vector-valued Besov spaces, and applications, Math. Nachr. 186 (1997),5--56.

\bibitem[B]{B} Bae, H.-O.: Analiticity for the Stokes operator in Besov spaces, J. Korean Math. Soc. 40 (2003), no. 6, 1061--1074.

\bibitem[BIN]{BIN} Besov, O.V.; Il'in, V.P.; Nikolskii, S.M.: Integral Representation of Functions and Theorems of Imbedding, Nauka, Moscow 1975 (in Russian); English transl., vol. I. Scripta Series in Mathematics. V. H. Winston, New York 1978.

\bibitem[CZ]{CZ} Chen, Q.; Zhang, Z.: Space-time estimates in the Besov spaces and the Navier-Stokes equations, Meth. and Appl. of Anal., 13 (2006), No. 1, 107-–122.

\bibitem[DM]{DM} Danchin, R; Mucha, P.B.: A critical functional framework for the inhomogeneous Navier–-Stokes equations in the half-space, Journal of Functional Analysis, 256 (2009), Issue 3, Pages 881--927.

\bibitem[FQ]{FQ} Fang, D.; Qian, C.: Regularity criterion for 3D Navier-Stokes
equations in Besov spaces, Comm. on Pure and Appl. Anal. (2014) 13, no. 2, 585--603.

\bibitem[G]{G} Golovkin, K.K.: On equivalent normalizations of fractional spaces, Trudy Mat. Inst. Steklov 66 (1962), 364--383.

\bibitem[KM]{KM} Kobayashi, T.; Muramatu, T.: Abstract Besov space approach to the nonstationary Navier-Stokes equations, Math. Methods Appl. Sci. 15 (1992), no. 9, 599–-620.

\bibitem[KOT1]{KOT1} Kozono, H.; Ogawa, T.; Taniuchi, Y.: The critical Sobolev inequalities in Besov spaces and regularity criterion to some semi-linear evolution equations, Math. Z., 242 (2002), 251–-278.

\bibitem[KOT2]{KOT2} Kozono, H.; Ogawa, T.; Taniuchi, Y.: Navier-Stokes equations in Besov space near $L^{\infty}$ and BMO,  Kyushu J. Math. 57 (2003), 303--324.

\bibitem[KS]{KS}  Kozono, H.; Shimizu, S.: Navier–Stokes equations with external forces in time-weighted Besov spaces, Mathematische Nachrichten. 291 (2018), 1781-–1800.

\bibitem[LSU]{LSU} Ladyzhenskaya, O.A.; Solonnikov, V.A.; Uraltseva, N.N.: Linear and Quasilinear Equations of Parabolic Type. Nauka, Moscow 1967 (in Russian); Translated from the Russian by S. Smith. Translations of Mathematical Monographs, vol. 23, xi+648 pp. American Mathematical Society, Providence 1968.

\bibitem[N]{N} Nikolskii, S.M.: Approximation of Functions with many variables and Imbedding Theorems, Nauka, Moscow 1977 (in Russian).

\bibitem[OS]{OS} Ogawa, T.; Shimizu, S.: Global well-posedness for the incompressible Navier-Stokes equations in the critical Besov space under the Lagrangian coordinates, Journal of Differential Equations 274 (2021), 613--651.

\bibitem[R]{R} Ri, M.-H.: Global well-posedness for inhomogeneous Navier–Stokes equations in endpoint critical Besov spaces, J. Math. Fluid Mech. 23(1), (2021), art. no. 16.

\bibitem[S]{S} Sawada, O.: On time-local solvability of the Navier-Stokes equations in Besov spaces, Advances in Differential Equations, 8 (2003), No. 4, 385-–412.

\bibitem[S1]{S1} Solonnikov, V.A.: A priori estimates for second order parabolic equations, Trudy Mat. Inst. Steklov 70 (1964), 133--212 (in Russian).

\bibitem[S2]{S2} Solonnikov, V.A.: An initial-boundary value problem for the Stokes system that arises in the study of free boundary problem, Trudy Steklov Mat. Inst. 188 (1990), 150--188.

\bibitem[S3]{S3} Solonnikov, V.A.: Estimates for solutions of nonstationary linearized system of the Navier-Stokes equations, Trudy MIAN SSSR 70 (1964), 213--317.

\bibitem[S4]{S4} Solonnikov, V.A.: Estimates for solutions to nonstationary Navier-Stokes system, Zap. Nauchn. Sem. LOMI 38 (1973), 153--231.

\bibitem[Tr1]{Tr1} Triebel, H.: Theory of function spaces, Akademische Verlagsgesellschaft, Geest\&Portig K.-G., Leipzig 1983, pp. 284.

\bibitem[ZZ1]{ZZ1} Zadrzy\'nska, E.; Zaj\c{a}czkowski, W.M.: The Cauchy-Dirichlet problem for the heat equation in Besov spaces, J. Math. Sc. 152 (5) (2008), 638--673.

\bibitem[ZZ2]{ZZ2} Zadrzy\'nska, E.; Zaj\c{a}czkowski, W.M.: Nonstationary Stokes system in Besov spaces, Math. Meth. Appl. Sci. 37 (2014), 360--383.

\bibitem[ZZ3]{ZZ3} Zadrzy\'nska, E.; Zaj\c{a}czkowski, W.M.: Some linear parabolic system in Besov spaces, Banach Center Publ. 81 (2008), 567--612.

\end {thebibliography}
\end{document}